\documentclass[11pt,reqno,a4paper]{amsart}
\usepackage[british]{babel}
\usepackage{amssymb,amstext,amsthm,eucal, float, upgreek}
\usepackage{latexsym,amsmath,amsfonts,amscd,mathrsfs, booktabs}
\usepackage{enumerate}
\usepackage{graphicx}
\usepackage{caption}
\usepackage{subcaption}
\usepackage{array,color}
\usepackage{enumitem}
\usepackage{tgpagella}
\usepackage{tikz, bbm}
\usepackage[utf8]{inputenc}
\usepackage[margin=2.8cm]{geometry}
\usepackage[small]{eulervm}
\usepackage[unicode]{hyperref}
\usepackage{framed}
\usepackage{xcolor}
\usepackage[arrow, matrix, curve]{xy}

\usetikzlibrary{arrows,decorations.markings,positioning}
\tikzset{inner sep=0pt, node distance=5mm,
  root/.style={circle,draw,minimum size=5pt,thick},
  broot/.style={circle,draw,minimum size=5pt,thick,fill},
  xroot/.style={circle,draw,minimum size=5pt,thick,fill=gray!70!white},
  crossroot/.style={circle,draw,minimum size=5pt,thick,label=below:$\times$},
  doublearrow/.style={postaction={decorate},   decoration={markings,mark=at position .6 with {\arrow[line width=1.2pt]{>}}},double distance=1.6pt,thick},
  doublenoarrow/.style={double distance=1.6pt,thick},
  rdoublearrow/.style={postaction={decorate},   decoration={markings,mark=at position .4 with {\arrowreversed[line width=1.2pt]{>}}},double distance=1.6pt,thick},
  rtriplearrow/.style={postaction={decorate},   decoration={markings,mark=at position .4 with {\arrowreversed[line width=1.2pt]{>}}},double distance=2.5pt,thick},
  curvedline/.style={bend=right}
}

\usetikzlibrary{positioning}
\tikzset{main node/.style={rectangle,rounded corners,fill=blue!20,draw,minimum size=1cm,inner sep=0pt},
            }
\tikzset{sec node/.style={rectangle,rounded corners,fill=green!20,draw,minimum size=1cm,inner sep=0pt},
            }
\usetikzlibrary{shapes.misc}
\tikzset{cross/.style={cross out, draw=black, minimum size=2*(#1-\pgflinewidth), inner sep=0pt, outer sep=0pt},
cross/.default={3pt}}

\definecolor{mygray}{gray}{0.7}

\usepackage[draft]{todonotes}   

\usetikzlibrary{arrows,decorations.markings,positioning}
\tikzset{inner sep=0pt, node distance=5mm,
  root/.style={circle,draw,minimum size=5pt,thick},
  broot/.style={circle,draw,minimum size=5pt,thick,fill},
  xroot/.style={circle,draw,minimum size=5pt,thick,label=below:$\times$},
  doublearrow/.style={postaction={decorate},   decoration={markings,mark=at position .6 with {\arrow[line width=1.2pt]{>}}},double distance=1.6pt,thick},
  rdoublearrow/.style={postaction={decorate},   decoration={markings,mark=at position .4 with {\arrowreversed[line width=1.2pt]{>}}},double distance=1.6pt,thick},
	rtriplearrow/.style={postaction={decorate},   decoration={markings,mark=at position .4 with {\arrowreversed[line width=1.2pt]{>}}},double distance=2.5pt,thick},
	ltriplearrow/.style={postaction={decorate},   decoration={markings,mark=at position .6 with {\arrow[line width=1.2pt]{>}}},double distance=2.5pt,thick},
  curvedline/.style={bend=right}
}
\hypersetup{
  pdftitle   = {Some notes on \texorpdfstring{$\mathrm{G}(3)$}{}},
  pdfkeywords = {},
  pdfauthor  = {Boris Kruglikov, Andrea Santi and Dennis The},
  pdfcreator = {\LaTeX\ with package \flqq hyperref\frqq}
}
\numberwithin{equation}{section}
\theoremstyle{plain}
\newtheorem{theorem}{Theorem}[section]
\newtheorem{proposition}[theorem]{Proposition}
\newtheorem{prop}[theorem]{Proposition}
\newtheorem{corollary}[theorem]{Corollary}
\newtheorem{cor}[theorem]{Corollary}
\newtheorem{lemma}[theorem]{Lemma}
\newtheorem{lem}[theorem]{Lemma}
\theoremstyle{definition}
\newtheorem{rem}[theorem]{Remark}
\newtheorem{rk}[theorem]{Remark}
\newtheorem{example}[theorem]{Example}
\newtheorem{definition}[theorem]{Definition}
\newcommand{\Hom}{\operatorname{Hom}}
\newcommand{\Aut}{\operatorname{Aut}}

\newcommand{\GL}{\operatorname{GL}}

\newcommand{\Spin}{\operatorname{Spin}}
\renewcommand{\Im}{\operatorname{Im}}

\newcommand{\der}{\mathfrak{der}}

\newcommand{\fgl}{\mathfrak{gl}}
\newcommand{\fsl}{\mathfrak{sl}}
\newcommand{\fso}{\mathfrak{so}}

\newcommand{\fm}{\mathfrak{m}}
\newcommand{\fp}{\mathfrak{p}}
\newcommand{\fs}{\mathfrak{s}}

\newcommand{\fg}{\mathfrak{g}}
\newcommand{\fh}{\mathfrak{h}}

\renewcommand{\1}{\mathbbm{1}}

\newcommand{\RR}{\mathbb{R}}

\newcommand{\gr}{\operatorname{gr}}

\renewcommand{\a}{\alpha}
\renewcommand{\b}{\beta}

\newcommand{\one}{\mathbbm{1}}

\newenvironment{psm}
  {\left(\begin{smallmatrix}}
  {\end{smallmatrix}\right)}
 \newcommand\cA{\mathcal{A}}

 \newcommand\cO{\mathcal{A}}
 \newcommand\bbC{\mathbb{C}}
 
 \newcommand\sfZ{\mathsf{Z}}
 \newcommand{\fe}{\mathfrak{e}}
 
 \newcommand{\fF}{\mathfrak{F}}
 
 \newcommand\op{\oplus}
 \newcommand\fpe{\mathfrak{pe}}
 \newcommand\fspe{\mathfrak{spe}}
 \newcommand\fcpe{\mathfrak{cpe}}
 \newcommand\fcspe{\mathfrak{cspe}}

 \newcommand\id{\operatorname{id}}

 \newcommand\cU{\mathcal{U}}
 \newcommand\cV{\mathcal{V}}
 
 \newcommand{\fG}{\mathfrak{G}}

 \newcommand\cC{\mathcal{C}}
 \newcommand\cD{\mathcal{D}}
 \newcommand\cF{\mathcal{F}}
 \newcommand\cG{\mathcal{G}}
 \newcommand\cH{\mathcal{H}}
 \newcommand\cL{\mathcal{L}}
 \newcommand\cP{\mathcal P}
 \newcommand\cQ{\mathcal Q}

 \newcommand\bS{\mathbf{S}}

 \newcommand\cE{\mathcal{E}}

 \newcommand\pr{\operatorname{pr}}

 \newcommand\Lie{\opp{Lie}}

 \newcommand\g{\mathfrak{g}}
 \newcommand\p{\partial}
 \newcommand\N{{\mathbb N}}
 \newcommand\R{{\mathbb R}}
 \newcommand\Z{{\mathbb Z}}
 \newcommand{\comm}[1]{}
 \newcommand\opp[1]{\mathop{\rm #1}\nolimits}

\allowdisplaybreaks
\begin{document}
\title[Symmetries of superdistributions and supergeometries]{Symmetries
of supergeometries\\ related to nonholonomic superdistributions}
\author{Boris Kruglikov}
\address{Department of Mathematics and Statistics, UiT The Arctic University of Norway, Troms\o{} 90-37, Norway}
\email{boris.kruglikov@uit.no}

\author{Andrea Santi}
\address{Department of Mathematics and Statistics, UiT The Arctic University of Norway, Troms\o{} 90-37, Norway}
\email{andrea.santi@uit.no}

\author{Dennis The}
\address{Department of Mathematics and Statistics, UiT The Arctic University of Norway, Troms\o{} 90-37, Norway}
\email{dennis.the@uit.no}

\thanks{}

 \begin{abstract}
We extend Tanaka theory to the context of supergeometry and obtain an upper bound
on the supersymmetry dimension of geometric structures related to strongly regular bracket-generating
distributions on supermanifolds and their structure reductions.
 \end{abstract}

\maketitle


\section{Introduction and the main results}\label{S1}
\vskip0.3cm\par

A theorem of Kobayashi states that $G$-structures of finite type have
finite-dimensional symmetry algebras and automorphism groups, and that the dimension of both is bounded via
Sternberg prolongation \cite{Ko,St}.
This also applies to the class of Cartan geometries that allow higher order reductions of the structure group.
Similarly, the Tanaka prolongation dimension bounds the symmetry dimension in the case of strongly regular bracket-generating
nonholonomic distributions and related geometric structures \cite{T,Y}.

 An analog of the Tanaka prolongation in the super-setting
is well-defined and was used by Kac in his classification \cite{Kac},
based on the ideas of Weisfeiler filtration \cite{W}.
We use this algebraic
{\em Tanaka--Weisfeiler prolongation} to construct a super-analog
of the Cartan--Tanaka frame bundle, with normalization conditions induced
from the generalized Spencer complexes. This in turn implies a bound
on the dimension of the symmetry superalgebra.

As we will recall in \S\ref{sec:superdistributions}, to every distribution\footnote{We will sometimes refer to a ``distribution on a supermanifold'' as simply a ``superdistribution'' for short.} $\cD$ on a supermanifold $M=(M_o,\cA_M)$,
one can associate a sheaf of negatively-graded Lie superalgebras $\opp{gr}(\mathcal T M)$ over $\cA_M$, which are fundamental for a bracket-generating $\cD$.
Under a strong regularity assumption, the sheaf is associated to a
classical bundle over the reduced manifold $M_o$ with fiber given by the symbol $\fm$ of $\cD$ (also called
the Carnot algebra of $\cD$).
Assume the Tanaka--Weisfeiler prolongation $\g=\opp{pr}(\fm)$ of $\fm=\g_-$
is finite-dimensional and let $G$ be a Lie supergroup with Lie superalgebra $\opp{Lie}(G) = \fg$. We will then construct the Cartan--Tanaka prolongation of the structure, which is a
fiber bundle over $M$ with typical fiber diffeomorphic to a subsupergroup $P\subset G$ having $\opp{Lie}(P)=\g_{\ge0}$.

Sometimes the geometric structure $(M,\cD)$ allows reductions, resulting in a reduction of $\g$.
In this paper, we will mainly consider reductions of the (graded) automorphism supergroup $\Aut_{gr}(\fm)$ to a smaller $G_0$, which implies a reduction of $\opp{pr}_0(\fm)=\der_{gr}(\fm)$ to $\Lie(G_0) = \fg_0$, but we will also briefly discuss higher order reductions.
In many important cases, the reduction is given by a choice of auxiliary geometric structure $q$
on the superdistribution $\cD$, which corresponds to $\fm_{-1}$. (For instance a tensor or a span of those. For higher order reductions $q$ is not
tensorial.) We will refer to $(M,\cD,q)$ as a filtered $G_0$-structure. In this case the Tanaka--Weisfeiler prolongation is denoted by
$\g=\opp{pr}(\fm,\g_0)
$  and
pure prolongation
$\g_0=\opp{pr}_0(\fm)$ corresponds to the structure $q$ void.

 \begin{theorem}\label{T1}
Let $\fs$ be the symmetry superalgebra of a bracket-generating, strongly regular filtered $G_0$-structure $(M,\cD,q)$,
with Tanaka--Weisfeiler prolongation $\g=\opp{pr}(\fm,\g_0)$ of $(\fm,\fg_0)$. Assume the reduced manifold $M_o$ is connected. Then $\dim\fs\leq\dim\g$ in the strong sense:
the inequality applies to the dimensions of the even and odd parts respectively.
 \end{theorem}
The Lie superalgebra $\fs$ can be considered
as a superalgebra of supervector fields localized in a fixed neighborhood $U_o\subset M_o$ or as germs of
those - the result holds in both cases.

Assuming $\dim \fg$ is finite, the above bound is sharp, meaning that there exists a standard model
with symmetry superalgebra $\fs$ equal to $\g$:
the homogeneous supermanifold $G/P$
gives a geometric structure of type $(\cD,q)$ with a maximal space of automorphisms, meaning that $G$ is the automorphism supergroup
(or differs from it by a discrete quotient) and $\dim G$ is the maximal possible dimension of such a supergroup.

 \begin{theorem}\label{T2}
Let $(M,\cD,q)$ be a bracket-generating, strongly regular filtered $G_0$-structure with a finite-dimensional Tanaka--Weisfeiler prolongation $\g=\opp{pr}(\fm,\fg_0)$.
If $M_o$ has finitely many connected components,
then $\opp{Aut}(M,\cD,q)$ is a Lie supergroup.
If $M_o$ is connected, then $\dim\opp{Aut}(M,\cD,q)\leq\dim\g$ in the strong sense as above.
 \end{theorem}

As noted above the dimension bound is sharp. 
In fact we have $\dim\opp{Aut}(M,\cD,q)\leq\dim\fs$ and the Lie superalgebra
$Lie(\opp{Aut}(M,\cD,q))$ is the subalgebra of $\fs$ consisting of
the complete supervector fields. (We recall that any supervector field
possesses a local flow in a suitable sense and it is called complete if its maximal flow domain is $\R^{1|1}\times M$, cf. \cite{MSV,GW}.
Moreover, it is complete if and only if the associated vector field on the reduced manifold $M_o$ is so.) Thus in many cases the inequality is strict.

The structure of the paper is as follows.
After introducing the main tools
for working with geometric structures on supermanifolds in \S\ref{S2},
we will show that the main ideas behind the classical results can be
carried over to the super-setting.
However, special care should be taken with the reduction of the structure group and usage of superpoints in frame bundles.
We manage this through a geometric-algebraic correspondence, elaborated for
principal bundles.
In \S\ref{S3}, we recall the algebraic prolongation following the ideas of Tanaka--Weisfeiler and construct the prolonged frame bundle with an absolute parallelism. Introduction of normalization conditions via the generalized
Spencer complex is inspired by a previous work by Zelenko \cite{Z}.
One of our main technical features is the geometric realization as
supermanifolds of the sheaves of frames introduced in \cite{A}
(in the context of $G$-structures). This is crucial to carry out
the inductive geometric prolongation argument.

In \S\ref{S4},
we give the proof of the main theorems, using the constructed frame bundles,
and discuss supersymmetry dimension bounds.
Furthermore we exploit a relation of the prolongation to the Lie equation
and note that the symmetry algebra $\fs$ of a filtered geometric structure
can be obtained by a filtered subdeformation of $\g$, i.e., by passing to
a graded Lie subsuperalgebra and changing its filtered structure while
 preserving its associated-graded.
We also discuss the maximal supersymmetry models there.
Some applications, in particular new symmetry bounds,
are given in \S\ref{S5}.
This covers holonomic supermanifolds, equipped with affine, metric,
symplectic, periplectic and projective structures, as well as
nonholonomic ones such as exceptional $G(3)$-contact structures, equations
of super Hilbert--Cartan type, super-Poincar\'e structures and
some scalar odd ODEs.

The automorphism supergroup in the case of $G$-structures was studied
by Ostermayr \cite{O} though this reference does not contain
the supersymmetry dimension bounds. Our class of geometries is
considerably larger, and in addition we consider the infinitesimal
symmetry superalgebra that gives a finer dimension bound.
We can also vary smoothness in the real case, to which we restrict for simplicity,
and our results hold true in the complex analytic or algebraic cases too, as well as in the mixed case
(cs manifolds, allowing for real bodies and complex odd directions) considered in \cite{O}.
Indeed, our arguments do not rely on any Batchelor realization of $M=(M_o,\cA_M)$ \cite{B}, which is well-known to fail for most classes of supermanifolds.

 For algebraic computations of prolongations,
related to certain geometric structures, we refer to the works of
Leites et al \cite{LSV,P} (see also references therein).

Finally, we remark that while Theorems \ref{T1} and \ref{T2} are formulated for strongly regular
distributions (with possible reductions),
we expect them to hold in the general case, allowing singularities.
This would superize the result of \cite{K1}.
One should only require the existence of a dense set of
localizations where the derived sheaves give rise to distributions.

\section{Bundles on supermanifolds and geometric-algebraic correspondence}\label{S2}
For details on the background material on supermanifolds we refer to \cite{CCF,DM,LSV,Sa,V}.
Here we elaborate the geometric-algebraic correspondence
for the description of fiber, vector and principal bundles.
To illustrate it, here are three definitions of tangent bundles:
 \begin{itemize}
\item[(i)]
Tangent bundle: This is the datum of an appropriate supermanifold $TM$ with
a surjective submersion $\pi:TM\to M$ of supermanifolds and typical fiber
$\simeq T_xM$, $x\in M_o$.
\item[(ii)]
Reduced tangent bundle: $\imath^*TM=TM|_{M_o}=\bigcup_{x\in M_o}T_xM$
is a classical $\Z_2$-graded vector bundle over $\imath:M_o\hookrightarrow
M$. The supervector space $T_xM$ is the fiber over $x\in M_o$.
\item[(iii)]
Tangent sheaf: superderivations $\mathcal TM=\opp{Der}(\cO_M)$
form a sheaf on $M_o$ of $\cO_M$-modules, whose global sections $\mathfrak X(M)=
\mathcal TM(M_o)$ consist of the supervector fields on $M$.
 \end{itemize}
Approaches (ii) and (iii) appear e.g.\ in \cite{A,G}; we will elaborate
upon (i) below. It will be shown that the geometric approach (i)
and the algebraic approach (iii) are equivalent, while the reduced tangent bundle $TM|_{M_o}$ encodes less information than $TM$ or $\mathcal TM$. We will also establish a similar geometric-algebraic correspondence for principal bundles, crucial for our developments.

\subsection{Supermanifolds}\label{sec:supermanifolds}
A {\it supermanifold} is understood in the sense of Berezin--Kostant--Leites, i.e., a ringed space $M=(M_o,\cO_M)$ such that
$\cO_M|_{\cU_o}\cong \mathcal{C}^{\infty}_{M_o}|_{\cU_o}\otimes\Lambda^\bullet \mathbb S^*$ as sheaves of superalgebras
for any sufficiently small open subset $\cU_o\subset M_o$.
Here $\mathbb S$ is a vector space  of fixed dimension. We set $\dim(M)=(m|n)=(\dim M_o|\dim \mathbb S)$,
call $M_o$ the reduced manifold and $\cO_M$ the structure sheaf, which is
$\Z_2$-graded: $\cO_M=(\cO_M)_{\bar0}\oplus(\cO_M)_{\bar1}$.
We shortly call {\it superdomain} the supermanifold $\cU=(\cU_o,\cO_M|_{\cU_o})$
associated to any open subset $\cU_o\subset M_o$, even if $\cU_o$ is not connected.
This terminology is also used for
$\varphi^{-1}(\cU)=(\varphi_o^{-1}(\cU_o),\cO_N|_{\varphi_o^{-1}(\cU_o)})$,
where $\varphi=(\varphi_o,\varphi^*):N=(N_o,\cO_N)\to M=(M_o,\cO_M)$ is a morphism of supermanifolds.
Despite its notation, the superdomain $\varphi^{-1}(\cU)$ is defined only in terms of $\cU_o$ and $\varphi_o$.
 Finally, an open cover of $M$ is a family of superdomains $\big\{\cU_i:i\in I\big\}$
such that $\bigcup_{i\in I}(\cU_i)_o=M_o$ and $\cU_{ij}=\cU_i\cap \cU_j$ is the superdomain with reduced manifold
$(\cU_{ij})_o=(\cU_i)_o\cap (\cU_j)_o$, for all $i,j\in I$.


For any sheaf $\cE$ over $M_o$, its restriction to an open subset $\cU_o\subset M_o$
will be denoted by $\cE|_{\cU_o}$, the space of its sections on $\cU_o$
simply by $\cE(\cU)$ and the stalk at $x\in M_o$ by
$\cE_x=\lim\limits_{\longrightarrow}{\!}_{\cU_o\ni x}\cE(\cU)$.
In particular,
we set
$\cO(\cU):=\cO_M(\cU_o)$ and $\cO_{M,x}:=(\cO_M)_x$ for the structure sheaf.

Let $\mathcal{J}=\langle\cO_{\bar1}\rangle$ be the subsheaf generated
by nilpotents: $\mathcal{J}=\mathcal{J}_{\bar0}\oplus\mathcal{J}_{\bar 1}$ with $\mathcal{J}_{\bar1}=\cO_{\bar1}$ and $\mathcal{J}_{\bar0}=\cO^2_{\bar1}$.
For any sheaf $\cE$  of $\cO_M$-modules on $M_o$ we consider the evaluation
$\opp{ev}:\cE\to\cE/(\mathcal{J}\cdot\cE)$. In particular
we get the reduction of superfunctions
$\opp{ev}:\cO_M\to C^\infty_{M_o}$, $f\mapsto\opp{ev}(f)$, and in turn the canonical morphism of supermanifolds $\imath=(\1_{M_o},\opp{ev}):M_o=(M_o,C^\infty_{M_o})\hookrightarrow M=(M_o,\cA_M)$.
Evaluation of
the classical function $\opp{ev}(f)$ at $x\in M_o$ is denoted by $\opp{ev}_x(f)$.
We stress, however, that there is no
{\it canonical} morphism from $M$ to $M_o$ -- this is a key feature of supergeometry.

For any supermanifold $S$, we will denote the {\it set of $S$-points} of $M$ by $M[S]=\Hom(S,M)_{\bar 0}$,
the set of all morphisms of supermanifolds from $S$ to $M$. (By definition morphisms are even, so the subscript ``$\bar 0$''
might look redundant. However, we reserve symbols like $\Hom$ and $\Aut$
for {\it superspaces} of morphisms, see \S\ref{sec:supergroups} below.) The functor of points $M[-]:SMan^{op}\to Set$
from the category of supermanifolds to the category of sets is
a (contravariant) functor that fully determines $M$.
However, there exist functors that are not representable, i.e.,
do not necessarily arise as the functor of points of a supermanifold.
We refer to them as {\it superspaces} (also known as generalized supermanifolds,
and not to be confused with the superspaces introduced by Manin \cite{Ma}).

A very useful criterion to check identities involving morphisms is the {\it Yoneda lemma}: each morphism $\varphi:M\to N$ defines a natural transformation $\varphi[-]:M[-]\to N[-]$ (i.e., a family of maps between sets $\varphi[S]:M[S]\to N[S]$ that depends functorially on $S$) and {\it any} natural transformation between $M[-]$ and $N[-]$ arises from a unique morphism in this way.

For any point $x\in M_o$ and supermanifold $S$, we let
$\hat x=(\hat x_o,\hat x^*):S\to M$ be the unique morphism such that
$\hat x^*(f)=\opp{ev}_x(f)\cdot 1\in\cO(S)$ for all $f\in\cO(M)$.
This gives $\hat x_o(S_o)=x\in M_o$.

\subsection{Lie supergroups and their actions}\label{sec:supergroups}

A Lie supergroup is a supermanifold $G=(G_o,\cO_G)$ endowed with a multiplication morphism $m:G\times G\to G$, an inverse morphism
$i:G\to G$ and a unit morphism $e:\R^{0|0}\to G$ with
usual compatibilities, which make $G$
a group object in the category of supermanifolds.
The reduced manifold $G_o$ is a classical Lie group.

The associated functor of points $G[-]:SMan^{op}\to Group$ is particularly useful in the case of linear Lie supergroups.
For example, consider the general linear Lie supergroup
$G=\GL(V)$ associated to a supervector space
$V=V_{\bar 0}\oplus V_{\bar 1}$ of $\dim V=(p|q)$.
The set of $S$-points $G[S]=\Hom(S,G)_{\bar 0}$ of $G$ is the group
 \begin{multline}\label{eq:Liesupergroup-functor}
G[S]=\Big\{\text{invertible}\;\begin{pmatrix} A & B \\ C & D \end{pmatrix}\mid A=(a^i_{j}), B=(b^i_{\beta}), C=(c^\alpha_{j}), D=(d^\alpha_{\beta})\\
\quad\text{with}\; a^i_{j}, d^\alpha_{\beta}\in \cO_{\bar 0}(S),\; b^i_{\beta},c^\alpha_{j}\in\cO_{\bar 1}(S)\;\text{for all}\;1\leq i,j\leq p,\;1\leq  \alpha,\beta\leq q\Big\}
 \end{multline}
of the even invertible $(p|q)\times (p|q)$ matrices with entries in $\cO(S)$. This group acts on the set of $S$-points of $V$
\begin{equation}
V[S]\cong \big(V\otimes \cO(S)\big)_{\bar 0}=\left\{
\begin{pmatrix}
v_1\\
\vdots\\
v_{p+q}\end{pmatrix}
\mid v_1,\ldots, v_p\in \cO_{\bar 0}(S),\;v_{p+1},\ldots, v_{p+q}\in \cO_{\bar 1}(S)
\right\}\;,
\end{equation}
where $V=V_{\bar 0}\oplus V_{\bar 1}$ is thought as the linear supermanifold
$V=(V_{\bar 0},\cC^{\infty}_{V_{\bar0}}\otimes\Lambda^\bullet V_{\bar 1}^*)$.
By Yoneda, we then have an action morphism of supermanifolds $\alpha:G\times V\to V$.
We note that $G$ is a superdomain of the linear supermanifold
$\fgl(V)=\fgl(V)_{\bar 0}\oplus\fgl(V)_{\bar 1}$, with
extended morphism of supermanifolds $\alpha:\fgl(V)\times V\to V$.

 One may similarly define a linear supergroup $G\subset\GL(V)$
with a morphism $\alpha:G\times V\to V$
that satisfies the usual properties of a linear action. See \cite{CCF} for more details.

\rem\label{Pi-rep}
If $\Pi V=V\otimes\mathbb R^{0|1}$ is the parity change supervector space, then $GL(V)\cong GL(\Pi V)$ as Lie supergroups via
the natural transformation
\begin{small}
$\begin{pmatrix} A & B \\ C & D \end{pmatrix}\to \begin{pmatrix} D & C \\ B & A \end{pmatrix}$,
\end{small}
however $V$ and $\Pi V$ are not equivalent as representations. 
We also note for later use that the action of $G[S]$ on $V[S]$ defined above extends to the whole $\cO(S)^{p|q}\cong V\otimes \cO(S)$, whence $G[S]\cong \Aut_{\cO(S)}(\cO(S)^{p|q})_{\bar 0}$.

\medskip
We may define an action of a supergroup on a supermanifold in the same vein,
namely via the functor of points $G[S]\times M[S]\to M[S]$, which by
Yoneda yields a morphism $G\times M\to M$ with the usual properties
of a (nonlinear) action. If the action is effective, this leads to an embedding $G\subset\Aut(M)$
to the ``supergroup of diffeomorphisms", but the sheaf approach is not
sufficient to define the superstructure of an infinite-dimensional
supermanifold \cite[\S2.6]{DM}, so we treat $\Aut(M)$ as a superspace
via the functor of points. More precisely, given two supermanifolds $M$ and $N$, the {\it superspace} of morphisms $\Hom(M,N)$ is required to satisfy the usual adjunction formula
$$
\Hom(M,N)[S]=\Hom(S,\Hom(M,N))_{\bar 0}\cong \Hom(S\times M,N)_{\bar 0}
$$
for all supermanifolds $S$. See \cite[\S 5.2]{SW} for an explicit description. If $M=N$, the supergroup of diffeomorphisms $\Aut(M)$ is defined as a subfunctor of $\Hom(M,M)$, and its ``reduced space'' $Aut(M)_{\bar 0}\subset \Hom(M,M)_{\bar 0}$ is the group of all diffeomorphisms of $M$
\cite[\S 5.1, \S 6.1]{SW}. It is then a straightforward task to define, e.g.,
the supergroup $\Aut(M,M')$ of diffeomorphisms of $M$ preserving a subsupermanifold $M'\subset M$ in terms of commutative diagrams.

The chart approach developed
below is better adapted, but we will mainly be interested in the
structures of finite type, where the automorphisms form a genuine Lie
supergroup.

\subsection{Fiber bundles and sections}
Recall that a morphism $\pi:E\to M$ is a submersion if
$d\pi|_{E_o}:TE|_{E_o}\to TM|_{M_o}$ is surjective \cite{L}.
Locally, this is a product of supermanifolds, and this is the basis of the following.

 \begin{definition}\label{def:geometric-fiber-bundle}
A morphism $\pi:E\to M$ is a {\em fiber bundle} with typical fiber $F$ if
$\forall x\in M_o$ $\exists$ a superdomain $\cU\subset M$, $x\in\cU_o$,
and a diffeomorphism (local trivialization)
$\varphi:\pi^{-1}(\cU)\rightarrow \cU\times F$ such that $\pr_\cU\circ\;\varphi=\pi$.
A morphism of fiber bundles	$\pi_1:E_1\to M_1$ and $\pi_{2}:E_2\to M_2$
is defined via the commutative diagram
 \begin{equation*}\xymatrix{
E_1 \ar[r]^{\psi}\ar[d]^{\pi_1} & E_2 \ar[d]^{\pi_2}\\
M_1 \ar[r]^{\psi_\flat} & M_2
 }\end{equation*}
Locally this means $\psi=\big(\psi_\flat\circ\pi_1,\varphi\big):\cU_1\times F_1\to \cU_2\times F_2$ for a morphism $\varphi:\cU_1\times F_1\to F_2$.
	\end{definition}

Let $\big\{\cU_i:i\in I\big\}$ be an open cover of $M$. A family $\big\{\varphi_{ij}:\cU_{ij}\times F\to \cU_{ij}\times F\big\}_{i,j\in I}$ of isomorphisms of trivial fiber bundles over the identity is called a {\em cocycle} if
 $
\varphi_{ii}=\1_{\cU_{ii}\times F}$ and $\varphi_{ij}=\varphi_{ik}\circ\varphi_{kj}$ where all three are defined.
Since the $\varphi_{ij}$ cover the identity in the
first component, they can be equivalently written as
$\varphi_{ij}:\cU_{ij}\times F\to F$, by abusing the notation, or even as morphisms
$\widetilde\varphi_{ij}:\cU_{ij}\to \Aut(F)$. We refer to the latter as ``transition morphisms''.

 \begin{proposition}\label{prop:fiber-bundles-cocycles}
Let $\pi:E\to M$ be a fiber bundle with local trivializations $(\cU_i,\varphi_i)$, $i\in I$. Then
the family $\big\{\varphi_{ij}=\varphi_i\circ\varphi_j^{-1}\big\}_{i,j\in I}$ is a cocycle.
Conversely any cocycle determines a unique fiber bundle. 
 \end{proposition}

 \begin{proof}
This is \cite[Prop 4.1.2]{Ke} for the case of vector bundles,
see also \cite[Prop. 4.9]{BCC}.
The idea is to glue $E$ from the local data $E_i=\cU_i\times F$
through the cocycles $\varphi_{ij}$ on $E_{ij}=E_i\cap E_j$.
First, glue $E_o$ from the reduced data $(E_i)_o$ and $(\varphi_{ij})_o$
as in the classical theory with $\pi_o:E_o\to M_o$.
Then glue the sheaves of superfunctions into
a sheaf $\cO_E:\cV_o\mapsto\cO_E(\cV_o)$ over $E_o$ defined as
 $$
\cO_E(\cV_o)=\big\{\big(s_i\in\cO_{E_i}((E_i)_o\cap\cV_o)\big)_{i\in I}\,:\,
\varphi_{ij}^*\big(s_i|_{(E_{ij})_o\cap\cV_o}\big)=s_j|_{(E_{ij})_o\cap\cV_o}\;
\forall\;i,j\in I\big\}.
 $$
We note that $\cO_E|_{\pi_o^{-1}(\cU_i)_o}\cong\cO_{E_i}$ by virtue of the cocycle conditions, hence the ringed space $E=(E_o,\cO_E)$ is a supermanifold.
\end{proof}	

 \begin{corollary}\label{cor:fiber-bundles-pull-back}
Let $\pi:E\to M$ be a fiber bundle with cocycle $\{\varphi_{ij}\}_{i,j\in I}$ and $\psi:N\to M$ a morphism of supermanifolds. Then the pull-back fiber bundle $p:\psi^*E\to N$ exists and is determined by the cocycle
$\varphi_{ij}\circ(\psi\times\one_F):\cV_{ij}\times F\to F$, where
$\cV_i=\psi^{-1}(\cU_{i})$ and $\cV_{ij}=\cV_i\cap\cV_j$ for all $i,j\in I$.
 \end{corollary}

Pull-back fits in a commutative diagram,
where the fiber bundle morphism $\psi^\sharp$ is determined by $\psi$:
\begin{equation*}	
\label{def:pull-back-fiber-bundles}
\xymatrix{
\psi^*E\ar[r]^{\psi^\sharp}\ar[d]^{p} & E \ar[d]^{\pi}\\
N \ar[r]^{\psi} & M
}
\end{equation*}

Recall that a subsupermanifold is called {\it closed} if its reduction is a closed submanifold.
 \begin{definition}
The {\em fiber} $E_x=\pi^{-1}(x)\hookrightarrow E$ at $x\in M_o$ is the closed subsupermanifold
given as the pullback
 $$
\xymatrix{ \pi^{-1}(x)  \ar[r]^{}\ar[d] & E \ar[d]^\pi\\
\mathbb R^{0|0} \ar[r]^{\hat x} & M }
 $$
 \end{definition}

\rem It can be specified more concretely via \cite[Prop. 3.4]{BCC}:
the algebra of global superfunctions of the fiber is $\cO(E_x)\cong\cO(E)/\cO(E)\pi^*(\mu_x)$, where
$\mu_x=\bigl\{f\in\cO(M)\mid\opp{ev}_x(f)=0\bigr\}$ is the maximal ideal.
The fiber $E_x$ is (non-canonically) diffeomorphic to the typical fiber $F$.

\medskip

The reduced fiber bundle is defined by
 $$
\imath^*E=E|_{M_o}=\displaystyle\bigcup_{x\in M_o}E_x
 $$
and it is a fiber bundle over $M_o$ with typical fiber $F$, according to
Definition \ref{def:geometric-fiber-bundle}. Its defining cocycle
$(\widetilde\varphi_{ij})_o:
(\cU_{ij})_o\to\Aut(F)_{\bar 0}$ is the reduced morphism of the cocycle of $E$.

 \begin{definition}\label{def:section}
An {\em even section} of a fiber bundle $\pi:E\to M$ is a morphism of
supermanifolds $\sigma:M\to E$ such that $\pi\circ\sigma=\one_{M}$.
Locally $\sigma=\big(\1_{\cU}, s\big):\cU\to \pi^{-1}(\cU)\cong\cU\times F$
with $s:\cU\to F$.
 \end{definition}

The same notion applies to superdomains $\cU\subset M$ and
we denote the set of all even sections by
$\Gamma_E(\cU)_{\bar 0}=\big\{\sigma:\cU\to\pi^{-1}(\cU)\mid\pi\circ\sigma=\one_\cU\big\}$.
(A superspace of sections 
can also be introduced
using the functor of points \cite[\S 4.3]{SW}, but we won't need that for our arguments.)


\subsection{Vector bundles on supermanifolds}\label{sec:superVB}
Let $V$ be a finite-dimensional supervector space.

 \begin{definition}[{\it Geometric approach}]\label{def:geometric-VB}
A {\em geometric vector bundle} is a fiber bundle $\pi:E\to M$ with
typical fiber $V$ such that the
transition morphisms are fiberwise linear, i.e., they take values in a linear supergroup:
$\widetilde\varphi_{ij}:\cU_{ij}\to G\subseteq\GL(V)$.
If $G$ is a proper subsupergroup, then $\pi$ is called a $G$-vector bundle.
 \end{definition}

More concretely, any cocycle $\varphi_{ij}:\cU_{ij}\times V\to V$ acts
on linear coordinates $(x^a,\theta^\alpha)$ of $V$ by
 \begin{equation}\label{eq:change-trivialization}
 \begin{aligned}
\varphi_{ij}^*(x^a)&=\widetilde\varphi_{ij}^*(A^{a}_b)\cdot x^b
+\widetilde\varphi_{ij}^*(B^{a}_\beta)\cdot\theta^\beta,\\
\varphi_{ij}^*(\theta^\alpha)&=\widetilde\varphi_{ij}^*(C^{\alpha}_b)\cdot x^b
+\widetilde\varphi_{ij}^*(D^{\alpha}_\beta)\cdot \theta^\beta,
 \end{aligned}
 \end{equation}
where $\begin{pmatrix}A^a_b,B^a_\beta, C^\alpha_b,D^\alpha_\beta\end{pmatrix}$ are coordinates on $GL(V)$, so that
 $
\begin{pmatrix} \widetilde\varphi_{ij}^*(A^a_b) &
\widetilde\varphi_{ij}^*(B^a_\beta) \\
\widetilde\varphi_{ij}^*(C^{\alpha}_b) &
\widetilde\varphi_{ij}^*(D^{\alpha}_\beta) \end{pmatrix}\in GL(V)[\cU_{ij}].
 $
In other words, $\varphi_{ij}=\alpha\circ\big(\widetilde\varphi_{ij}
\times\one_V\big):\cU_{ij}\times V\to V$. Note also that fiberwise linearity for the reduced bundle is a weaker condition
than that for the geometric vector bundle because the off-diagonal blocks
of the above matrix are odd, hence vanish
upon evaluation \cite[Ex. 4.15]{BCC}.

\medskip

A morphism of geometric vector bundles $E_1,E_2$ is a fiber bundle morphism
$\psi:E_1\to E_2$ that is fiberwise linear. More concretely,
let $\pi_1:E_1\to M_1$ and $\pi_{2}:E_2\to M_2$ be geometric vector bundles
with typical fibers $V$ and $W$, respectively.
Denote by $\alpha:\Hom(V,W)\times V\to W$ the natural composition morphism of supermanifolds.
Then, similar to \cite[Def. 4.12]{BCC}, we define a morphism of vector bundles
in a local component $\varphi:\cU_1\times V\to W$ as
$\varphi=\alpha\circ\big(\rho\times\one_V\big):\cU_1\times V\to W$
for some morphism $\rho:\cU_1\to\Hom(V,W)$.

Proposition \ref{prop:fiber-bundles-cocycles} and Corollary
\ref{cor:fiber-bundles-pull-back} specialize straightforwardly
to geometric vector bundles.
Next, formula \eqref{eq:change-trivialization} shows that
the assignment $\cU_o\to\Gamma_E(\cU)_{\bar 0}$
describing local even sections,
gives a sheaf of right $(\cO_M)_{\bar 0}$-modules on $M_o$.
(This can be converted to a left module with the usual rule of signs.)
This sheaf however is {\it not} locally free
(e.g., for $E=TM$, $M=\R^{0|2}(\theta^1,\theta^2)$, the module of even supervector fields
$\langle\theta^\alpha\p_{\theta^\beta}\rangle$ is not free over $(\cO_M)_{\bar 0}=\langle1,\theta^1\theta^2\rangle$),
but it can be enlarged to a locally free sheaf $\cU_o\to\Gamma_E(\cU)=\Gamma_E(\cU)_{\bar 0}\oplus \Gamma_E(\cU)_{\bar 1}$
of (right) $\cO_M$-modules of rank $(p|q)$ as follows.

First we define the parity change vector bundle $\pi:\Pi E\to M$ as the geometric vector bundle with typical fiber $\Pi V$ determined by the
transition morphisms
$\Pi\widetilde\varphi_{ij}:\cU_{ij}\to GL(V)\cong GL(\Pi V)$. Then we let
$\Gamma_E(\cU)_{\bar 1}=\Gamma_{\Pi E}(\cU)_{\bar 0}$ and henceforth
$\Gamma_{E}(\cU)=\Gamma_E(\cU)_{\bar 0}\oplus \Gamma_E(\cU)_{\bar 1}$.
 \begin{definition}
The sheaf $\Gamma_E$ on $M_o$ is called the {\em sheaf of sections}
of $\pi:E\to M$.
 \end{definition}

This leads to an alternative definition of a vector bundle. For simplicity of exposition, we assume that $M_o$ is connected for the remaining part of \S\ref{S2}.

 \begin{definition}[{\it Algebraic approach}]\label{def:algebraic-VB}\hfill
\begin{enumerate}
\item A 
locally free sheaf $\mathcal E$ on $M_o$ of (right)
$\cO_M$-modules of finite rank is called an {\em algebraic vector bundle} over $M$.

\item A morphism $\psi :\cG\to \cF$ of sheaves on $M_o$ of $\cO_M$-modules
consists of a family of morphisms $\psi _{\cU_o}:\cG(U_o)\to \cF(U_o)$
of $\cO(\cU)$-modules for each open subset $\cU_o\subset M_o$, subject to the natural
compatibility conditions with restrictions to $\cV_o\subset\cU_o$.
\end{enumerate}
 \end{definition}

 \begin{proposition}\label{prop:equivalence-vector-bundles}
Assume $M_o$ is connected. Then the category of geometric vector bundles on $M$ with morphisms of vector bundles covering the identity $\one_M:M\to M$
is equivalent to the category of algebraic vector bundles over $M$ with morphisms of sheaves of $\cO_M$-modules.
 \end{proposition}

 \begin{proof}
This is \cite[Prop. 4.22]{BCC}, to which we refer the reader.
%
%
%
%
\end{proof}

Recall that a {\em coherent sheaf} on $M=(M_o,\cO_M)$ is
a sheaf $\cF$ on $M_o$ of $\cO_M$-modules that has
a {\it finite local presentation}: every point $x\in M_o$ has an open neighborhood $\cU_o$ in which there is an exact sequence
$\cO_M^r|_{\cU_o}\to\cO_M^s|_{\cU_o}\to\cF|_{\cU_o}\to0$ for some $r,s\in\N$.
The algebraic vector bundles are therefore the locally free coherent sheaves.

Let $\varphi:M\to N$ be a morphism and $\cF$ a sheaf on $M_o$ of $\cO_M$-modules. Then the {\em direct image sheaf} $\varphi_*\cF$
is the sheaf of $\cO_N$-modules over $N_o$ given by the law
 $
\varphi_*\cF:N_o\supset\cU_o\mapsto\cF(\varphi_o^{-1}(\cU_o))
 $
and the homomorphism $\varphi^*:\cO_N\to(\varphi_o)_*\cO_M$.
The kernel, cokernel and direct image of a morphism of coherent sheaves
are coherent sheaves. The inverse image sheaf $\varphi_o^{-1}\mathcal G$ of a sheaf $\mathcal G$
on $N_o$ exhibits some relatively subtle features and it is easier to define
directly in terms of stalks: given $x\in M_o$, one has
$\bigl(\varphi_o^{-1}\cG\bigr)_x\cong\cG_{\varphi_o(x)}$.
If $\cG$ is a sheaf of $\cO_N$-modules, then $\varphi_o^{-1}\cG$ is only
a sheaf of $\varphi_o^{-1}\cO_N$-modules.
In this situation, the {\em inverse image sheaf} is defined as the sheaf
of $\cO_M$-modules by the formula
 $
\varphi^*\cG=\varphi_o^{-1}\cG\otimes_{\varphi_o^{-1}\cO_N}\cO_M
 $,
where the left action of $\varphi_o^{-1}\cO_N$ on $\cO_M$ is defined
by the map $\varphi_o^{-1}\cO_N\to\varphi_o^{-1}(\varphi_o)_*\cO_M\twoheadrightarrow\cO_M$ induced by $\varphi^*:\cO_N\to(\varphi_o)_*\cO_M$. If $\mathcal G$ is locally free, then the sheaf $ \varphi^{*}\mathcal G$
is locally free as well.

There is a natural adjunction correspondence between morphisms of sheaves of modules:
 $
\Hom_{\cO_M}(\varphi^*\cG,\cF)_{\bar0}\cong
\Hom_{\cO_N}(\cG,\varphi_*\cF)_{\bar0}
 $.

 \begin{definition}
A morphism $\cE_1\to\cE_2$ of locally free coherent sheaves $\cE_i$
of $\cO_{M_i}$-modules, $i=1,2$, is a pair
$(\psi,\psi_\flat)$, where $\psi_\flat:M_1\to M_2$ is a morphism of supermanifolds and $\psi:\cE_1\to\psi_\flat^{*}\cE_2$ is a morphism  of sheaves of $\cO_{M_1}$-modules.
 \end{definition}

Now Proposition \ref{prop:equivalence-vector-bundles} extends to morphisms
of vector bundles over different bases:
 \begin{theorem}\label{thm:equivalence-VB}
The category of the geometric vector bundles with morphisms of vector bundles
is equivalent to the category of algebraic vector bundles with morphisms of locally free coherent sheaves, provided the reduced manifolds of the bases of the bundles are connected.
 \end{theorem}

 \begin{proof}
Only the part of the proof on morphisms has to be modified and this
is based on the following observations.

\smallskip

(i)
The pullback of geometric vector bundles correspond to the inverse image of locally free coherent sheaves.

\smallskip

(ii)
A morphism of geometric vector bundles from $E_1$ to $E_2$ always factorizes through the pull-back bundle as
 $$\xymatrix{
E_1 \ar[r]^{\psi}\ar[d]^{\pi_1} & \psi_\flat^*E_2\ar[d]^{p_2}\ar[r]^{(\psi_\flat)^\sharp}& E_2\ar[d]^{\pi_2}\\
M_1 \ar[r]^{\one_{M_1}} & M_1\ar[r]^{\psi_\flat}& M_2
 }$$
Therefore it can be viewed as a morphism of supermanifolds
$\psi_\flat:M_1\to M_2$ paired with a morphism of geometric vector bundles
$\psi:E_1\to \psi_\flat^* E_2$ covering the identity $\one_{M_1}$. The claim then follows from the claim on morphisms of Proposition \ref{prop:equivalence-vector-bundles}.
 \end{proof}

Definitions \ref{def:geometric-VB} and \ref{def:algebraic-VB} can therefore be
used interchangeably, but care must be given to distinguish between morphisms
of vector bundles and sheaves. For instance the kernel and the direct image of
a morphism of vector bundles can only be interpreted as coherent sheaves in
general. Henceforth we will use the nomenclature vector bundle without specification.

\subsection{Principal bundles on supermanifolds}
Let $G=(G_o,\cO_G)$ be  a Lie supergroup.

 \begin{definition}[{\it Geometric approach}]\label{def:geometric-PB}
A {\em geometric principal bundle} with structure group $G$ is a fiber bundle $\pi:P\to M$ with
typical fiber $G$ such that the
transition morphisms take values in the Lie supergroup
 acting on itself by left multiplication: $\widetilde\varphi_{ij}:\cU_{ij}\to G\subset \Aut(G)$.
 \end{definition}

The right action of $G$ on a local trivialization $\cU_{i}\times G$ is given by $\one_{\cU_i}\times m:\cU_{i}\times G\times G\to\cU_{i}\times G$, and
it extends to a well-defined action morphism of supermanifolds
$\alpha:P\times G\to P$ satisfying $\pi\circ\alpha=\pi\circ \pr_P$.
Moreover the $\varphi_{ij}$ are $G$-equivariant, i.e., we have the commutative diagram
 $$\xymatrix{
\cU_{ij}\times G\times G\;\; \ar[r]^{\;\;\;\;\;\;\varphi_{ij}\times\one_G}\ar[d]^{\one_{\cU_{ij}}\times m}
& \;\;G\times G \ar[d]^{m}\\
\cU_{ij}\times G \ar[r]^{\varphi_{ij}} & G
 }$$
and by the Yoneda lemma this is equivalent to the identity
$\varphi_{ij}=m\circ\big(\widetilde\varphi_{ij}\times\one_G\big)$.
\begin{definition}
The {\em fundamental vector field} $\zeta_X\in\mathfrak{X}(P)$ associated to $X\in \fg$ is the supervector field  defined by
$(\1_{P}\otimes X)\circ \alpha^*=\alpha^*\circ \zeta_X$. Equivalently, given any local trivialization $\pi^{-1}(U)\cong \cU\times G$, it is the left-invariant supervector field on $G$ corresponding to $X$.
\end{definition}

Let $\pi_1:P_1\to M_1$, $\pi_2:P_2\to M_2$ be geometric principal bundles
with structure groups $G_1$ and $G_2$, respectively, and let
$\gamma:G_1\to G_2$ be a homomorphism of Lie supergroups.
We define a {\em $\gamma$-morphism of principal bundles} to be
a fiber bundle morphism $\psi:P_1\to P_2$ that is $\gamma$-equivariant.
More concretely, this is expressed as the commutative diagram
 $$\xymatrix{
P_1\times G_1 \ar[r]^{\psi\times\gamma}\ar[d]^{\alpha_1} & P_2\times G_2 \ar[d]^{\alpha_2}\\
P_1 \ar[r]^{\psi} & P_2
 }$$
Equivalently a local component $\varphi:\cU_1\times G_1\to G_2$
of a $\gamma$-morphism has the following form
$\varphi=m_2\circ\big(g\times\gamma\big):\cU_1\times G_1\to G_2$,
for some morphism $g:\cU_1\to G_2$.

If $G=G_1=G_2$ and $\gamma=\one_{G}$, we simply say that $\psi$ is a morphism of $G$-principal bundles.

 \begin{example}\label{ExRed}
Let $\gamma:G_1\hookrightarrow G_2$ be a subsupergroup.
A $\gamma$-morphism $\psi:P_1\to P_2$ of principal bundles over the same base
$M$ with $\psi_\flat=\one_M$ is called a {\em reduction} of the
structure group.
 \end{example}

Proposition \ref{prop:fiber-bundles-cocycles} and Corollary
\ref{cor:fiber-bundles-pull-back} specialize straightforwardly
for principal bundles.

\medskip

Now we consider the algebraic approach. Restricting the functor of points
of $G$ to superdomains of $M$, we get a sheaf of classical groups  $
\cG_M:\cU_o\mapsto\cG_M(\cU_o):=G[\cU]
 $ over $M_o$.

 \begin{definition}
A {\em sheaf of right $\cG_M$-sets} is a sheaf of sets $\mathcal P$ on $M_o$
on which the sheaf $\cG_M$ acts on the right:
$\forall$ open subset $\cU_o\subset M_o$ we have an action
$\alpha_{\cU_o}^\cP:\cP(\cU_o)\times\cG_M(\cU_o)\to\cP(\cU_o)$
and these actions are compatible with restrictions to open subsets $\cV_o\subset\cU_o$.
 \end{definition}

Let $\cP$ be a sheaf of right $\cG_M$-sets and $\cQ$ a sheaf of right
$\cH_M$-sets on $M_o$, and let $\gamma:G\to H$ be a homomorphism of Lie supergroups. A {\em $\gamma$-morphism} $\psi:\cP\to\cQ$
associates morphisms of sets $\psi_{\cU_o}:\cP(\cU_o)\to\cQ(\cU_o)$
compatible with restrictions
to open subsets $\cV_o\subset\cU_o$ and that are $\gamma$-equivariant:
 $$\xymatrix{
\cP(\cU_o)\times \cG_M(\cU_o)\;\; \ar[r]^{\small{\psi_{\cU_o}\times\gamma[\cU]}}
\ar[d]^{\alpha_{\cU_o}^\cP} &
\;\;\;\cQ(\cU_o)\times\cH_M(\cU_o) \ar[d]^{\alpha_{\cU_o}^\cQ}\\
\cP(\cU_o) \ar[r]^{\psi_{\cU_o}} & \cQ(\cU_o)
 }$$

 \begin{definition}[{\it Algebraic approach}]\hskip-2pt%
An {\em algebraic principal bundle} over $M$ with structure group $G$
is a sheaf $\mathcal P$ of right $\cG_M$-sets that is {\it locally
simply transitive} in the following sense: $\forall x\in M_o$ $\exists$
open neighborhood $\cU_o$ for which
$\cG_M(\cU_o)$ acts simply transitively on $\mathcal P(\cU_o)$.
\end{definition}


Let $\varphi:M\to N$ be a morphism of supermanifolds and
$\mathcal P$ a sheaf on $N_o$ of right $\cH_N$-sets.
Similar to the case of algebraic vector bundles we note that
the inverse image sheaf $\varphi_o^{-1}\cP$ is only a sheaf
of $\varphi_o^{-1}\cH_N$-sets.
In this situation, the {\em inverse image sheaf} is defined as the sheaf
of $\cH_M$-sets by the formula
 $
\varphi^*\cP=\varphi_o^{-1}\cP\times_{\varphi_o^{-1}\cH_N}\cH_M
 $,
where the left action of $\varphi_o^{-1}\cH_N$ on $\cH_M$ is defined via $\varphi_o^{-1}\cH_N\to\varphi_o^{-1}(\varphi_o)_*\cH_M\twoheadrightarrow\cH_M$. 
If $\mathcal P$ is locally simply transitive, then
$\varphi^*\cP$ is locally simply transitive as well.

\begin{definition}
A {\em $\gamma$-morphism of algebraic principal bundle} $\cP_1\to \cP_2$ is a
pair $(\psi,\psi_\flat)$ where $\psi_\flat:M_1\to M_2$ is a morphism of
supermanifolds and $\psi:\cP_1\to\psi_\flat^*\cP_2$ a $\gamma$-morphism.
\end{definition}

To link geometric principal bundles with algebraic principal bundles, it is sufficient to consider even sections as defined in
Definition \ref{def:section}.
For any morphism $g:\cU\to G$ and any even section $\sigma\in\Gamma_P(\cU)_{\bar0}$
we define another section by $\sigma\cdot g=\alpha\circ\big(\sigma, g\big)$,
so the assignment $\cU_o\mapsto\Gamma_P(\cU)_{\bar 0}$ is sheaf of
right $\cG_M$-sets.
By the Yoneda lemma, one easily sees the following:
 \begin{lemma}
For any $\sigma,\tau\in \Gamma_P(\cU)_{\bar 0}$ there exists a unique morphism $g:\cU\to G$ such that $\tau=\sigma\cdot g$.
 \end{lemma}

In other words $\cG_M(\cU_o)=G[\cU]$ acts simply transitively on $\Gamma_P(\cU)_{\bar 0}$ if it is nonempty. Since this condition is always satisfied for small superdomains $\cU\subset M$ we conclude the following.
 \begin{proposition}
The sheaf of even sections of a geometric principal bundle $\pi:P\to M$ is an algebraic principal bundle.
 \end{proposition}

Conversely, given an algebraic principal bundle $\mathcal P$,
there is an open cover $\big\{\cU_i\big\}_{i\in I}$ of $M$ such that
$\cG_M\big((\cU_i)_o\big)$ acts simply transitively on $\mathcal P\big((\cU_i)_o\big)$
for all $i\in I$. In other words, we have an identification
$t_i:\mathcal P\big((\cU_i)_o\big)\to \cG_M\big((\cU_i)_o\big)$ of right $\cG_M\big((\cU_i)_o\big)$-sets for all $i\in I$.
On intersections $\cU_{ij}$ we get isomorphisms $
\varphi_{ij}=t_i\circ t_j^{-1}:\cG_M\big((\cU_{ij})_o\big)\to\cG_M\big((\cU_{ij})_o\big)
 $
of simply transitive right $\cG_M\big((\cU_{ij})_o\big)$-sets.
These are left multiplication by some
$g_{ij}\in\cG_M\big((\cU_{ij})_o\big)=G[\cU_{ij}]$,
whence the morphisms $\widetilde\varphi_{ij}=g_{ij}:\cU_{ij}\to G$
that satisfy the cocycle conditions.

Therefore we obtain a geometric principal bundle $\pi:P\to M$.

 \begin{theorem}\label{thm:equivalence-PB}
The category of geometric principal bundles with $\gamma$-morphisms
is equivalent to the category of algebraic principal bundles with $\gamma$-morphisms, for homomorphisms of supergroups $\gamma$.
 \end{theorem}

 \begin{proof}
We already proved the correspondence between objects.
The correspondence between morphisms is similar to that of the proof
of Theorem \ref{thm:equivalence-VB}: 

\smallskip

(i)
The pullback of a geometric principal bundle correspond to the inverse image sheaf.

\smallskip

(ii)
A $\gamma$-morphism of geometric principal bundles from $P_1$ to $P_2$ always factorizes through the pull-back bundle as
 $$\xymatrix{
P_1 \ar[r]^{\psi}\ar[d]^{\pi_1} & \psi_\flat^*P_2\ar[d]^{p_2}\ar[r]^{(\psi_\flat)^\sharp}& P_2\ar[d]^{\pi_2}\\
M_1 \ar[r]^{\one_{M_1}} & M_1\ar[r]^{\psi_\flat}& M_2
 }$$
where $(\psi_\flat)^\sharp: \psi_\flat^*P_2\to P_2$ is a morphism of $G_2$-principal bundles.
Therefore it can be viewed as a morphism of supermanifolds $\psi_\flat:M_1\to M_2$ paired with a $\gamma$-morphism of geometric principal bundles $\psi:P_1\to \psi_\flat^* P_2$ covering the identity $\one_{M_1}:M_1\to M_1$. One then concludes as in Theorem \ref{thm:equivalence-VB}.
 \end{proof}

\subsection{Subbundles}\label{subbundles}

Let $F'\subset F$ be a subsupermanifold.

 \begin{definition}\hskip-2pt%
Let $\pi:E\to M$ be a fiber bundle with typical fiber $F$.
A {\em subbundle} with typical fiber $F'$ is a subsupermanifold $E'\subset E$
such that for some open cover $\{\cU_i:i\in I\}$ of $M$:
\begin{enumerate}
	\item the
transition morphisms $\widetilde\varphi_{ij}:\cU_{ij}\to\Aut(F)$ of $E$ factor
through $\Aut(F,F')$,
	\item $\pi:E'\to M$ is itself a fiber bundle
with typical fiber $F'$, whose transition morphisms are obtained by
postcomposing with the restriction map $\Aut(F,F')\to\Aut(F')$.
\end{enumerate}
\end{definition}

This has a clear specification for vector and principal bundles:
 \begin{itemize}
\item For vector subbundles the typical fibers $V'\subset V$ are supervector spaces
   and we consider $\GL(V)\supset\GL(V,V')\to\GL(V')$ instead of general automorphisms;
\item For principal subbundles we have instead a Lie subsupergroup $G'\subset G$.
 \end{itemize}
These are geometric subbundles. There are also algebraic subbundles,
of which we specify only the vector version, leaving algebraic
principal subbundles to the reader.

 \begin{definition}
An {\em algebraic vector subbundle} is subsheaf $\cE'$ of $\cE$ that
is locally a direct factor, i.e.\ $\forall x\in M_o$ $\exists$ a
neighborhood $\cU_o$ and a subsheaf $\cE''\subset\cE|_{\cU_o}$
such that $\cE_y=\cE'_y\oplus\cE''_y$ $\forall y\in\cU_o$.
\end{definition}

An algebraic vector subbundle is itself an algebraic vector bundle.
Indeed (cf.\ \cite[\S 4.7]{V}) by Nakayama's lemma $\forall x\in\cU_o$
the stalk $\cE'_x$ is a free module over the local ring $\cO_{M,x}$ of dimension
$\dim_{\cO_{M,x}}(\cE'_x)=\dim_\R(\cE'_x/\mu_x\cE'_x)$,
where $\mu_x\subset\cO_{M,x}$ is the maximal ideal. Since $\cE'$ is locally a direct factor, we have:
 \begin{itemize}
\item[$(i)$] the fiber $E'_x\cong\cE'_x/\mu_x\cE'_x$ embeds into the fiber
    $E_x\cong\cE_x/\mu_x\cE_x$;
\item[$(ii)$] the rank $\dim_\R E'_x=\dim_{\cO_{M,x}}(\cE'_x)$ is locally constant, thus constant.
 \end{itemize}
Consequently $\cE'$ is a a locally free coherent sheaf, proving our claim.

\smallskip

Note that a vector subbundle is not the same as an injective morphism of
vector bundles. Indeed, the latter may not be locally a direct factor.
Similarly to the previous subsections one can prove the equivalence of algebraic
and geometric definitions of subbundles.

We can also define associated bundles. Let $\pi:P\to M$ be a geometric principal bundle determined by transition morphisms $g_{ij}:\cU_{ij}\to G$. For any representation
$\rho:G\to \GL(V)$ the associated geometric vector bundle $E=P\times_{\rho} V$ is defined through the transition morphisms $\rho_{ij}=\rho\circ g_{ij}:\cU_{ij}\to \GL(V)$. Its sections are in bijective correspondence with $G$-equivariant morphisms
$f:P\to V\oplus \Pi V$.

\section{Prolongation Structures}\label{S3}

We begin by revising prolongations of $G$-structures on supermanifolds
using our setup and then will pass to the filtered version.

\subsection{Frame bundles}
\label{sec:frame-bundles}
Let $V_M= M\times V\to M$ be the trivial vector bundle over $M$ with the fiber
$V=\R^{m|n}$, where $\dim (M)=(m|n)$, and $\mathcal V_M$ is the associated locally free sheaf on $M_o$.
The frame bundle $\pi:Fr_M\to M$ is defined via
the geometric-algebraic correspondence as the sheaf $\mathcal{F}r_M:M_o\supset\mathcal{U}_o\mapsto\mathcal{F}r_M(\mathcal{U}_o)$
of $\mathcal{A}_M$-linear isomorphisms from $\mathcal{V}_M$ to $\mathcal{T}M$:
\begin{equation}
\label{eq:frame-sheaf}
\mathcal{F}r_M(\mathcal{U}_o)=\big\{\text{$\cA|_{\cU_o}$-linear isomorphism}\;F:\mathcal{V}_M|_{\cU_o}\to\mathcal{T}M|_{\cU_o}\big\}\,.
\end{equation}
We prefer to think of $\mathcal{V}_M$ and $\mathcal{T}M$ as locally free sheaves of {\it left} $\cA_M$-modules and define the linearity of $F$ accordingly; such a linear isomorphism still yields, via the sign rule, a local isomorphism $V_M\to TM$ of vector bundles covering the identity of $M$.

Setting $\mathcal{T}_M^{m|n}=\mathcal{T}M^{\oplus m}\bigoplus
\Pi\mathcal{T}M^{\oplus n}$,
we have an embedding of sheaves $\mathcal{F}r_M\hookrightarrow(\mathcal{T}_M^{m|n})_{\bar 0}$ whose image, by the
Nakayama lemma, consists of $(m|n)$-tuples of even and odd supervector fields such that their reductions to $M_o$ give a basis of $T_xM=(T_xM)_{\bar 0}\oplus (T_xM)_{\bar 1}$ at each point $x\in M_o$.
In other words, by passing to the reduced  bundle $Fr_M|_{M_o}\to M_o$ we get
the usual frame bundle for classical $\Z_2$-graded vector bundles over $M_o$.
More concretely, $F\in\mathcal{F}r_M(\cU_o)$ is
a frame $(X_1,\dots,X_m\,|\,Y_1,\dots,Y_n)$ with $X_i\in\mathfrak{X}(\cU)_{\bar0}$ and $Y_j\in\mathfrak{X}(\cU)_{\bar1}$ for $1\leq i\leq m$,
$1\leq j\leq n$, considered as an isomorphism
\begin{equation}
\label{eq:sign-I}
F:\mathcal{V}_M|_{\mathcal{U}_0}\to\mathcal{T}M|_{\mathcal{U}_0}\;,\qquad
(a_i|b_j)\mapsto\sum_{i=1}^ma_iX_i+\sum_{j=1}^nb_jY_j\,.
\end{equation}
The sheaf of groups $\mathcal{GL}_M:\cU_o\mapsto\GL(V)[\cU]$ acts naturally
on the right on the set of frame fields via
\begin{equation}
\label{eq:sign-II}
\big(X_1,\ldots,X_m\,|\,Y_1,\dots,Y_n\big)\cdot g=\Bigl(\sum_{i=1}^m a_1^i X_i
-\sum_{\alpha=1}^n c_1^\alpha Y_\alpha,\,\ldots\,\big|\,\ldots\,,
\sum_{i=1}^m b_n^i X_i+ \sum_{\alpha=1}^n d_n^\alpha Y_\alpha\Bigr).
\end{equation}
where $g\in \GL(V)[\cU]$ is parametrized as in \S\ref{eq:Liesupergroup-functor}.
This action is locally simply transitive,
hence $\mathcal{F}r_M$ is an algebraic principal bundle over
$M$ with structure group $GL(V)$. (The minus sign in \eqref{eq:sign-II} follows from the sign rule, so that \eqref{eq:sign-II} is indeed an action.)

Higher order frame bundles $Fr^k_M\to M$ ($k\geq 1$, with $Fr^1_M=Fr_M$)
can be introduced via the description of jet superbundles of \cite{HR-MM}, which does not rely neither on (topological) points or the functor of points.
Let $J^k(\mathbb R^{m|n}, M)$ be the vector bundle of jets from $\mathbb R^{m|n}$ to $M$,
which is defined as a supermanifold of homomorphisms of appropriate algebras in \cite[\S 6]{HR-MM}. This is a geometric vector bundle over
the product $\mathbb R^{m|n}\times M$. The open subsupermanifold $J^k_{inv}(\mathbb R^{m|n}, M)$ of $J^k(\mathbb R^{m|n}, M)$ of invertible jets is easily defined via local coordinates and its pull-back to $M$ by the natural injection $M\cong \{0\}\times M\hookrightarrow \mathbb R^{m|n}\times M$ is the higher order frame bundle $Fr^k_M\to M$.
The corresponding sheaf is obtained via the geometric-algebraic correspondence.
Below we adapt a different approach (which is though similar in spirit)
to introduce higher frame bundles in the non-holonomic situation.


\subsection{$G$-structures on supermanifolds}
\label{subsec:Gstru}
In geometric language, a $G$-structure for $G\subset\mathop{GL}(V)$ is a reduction of the
frame bundle as in Example \ref{ExRed}; following \S\ref{subbundles}, this
corresponds to a subbundle $F_G\subset Fr_M$.
In algebraic language this is a subsheaf ${\cF}_G\subset\mathcal{F}r_M$ on which the subsheaf $\cG_M\subset\mathcal{GL}_M$
acts locally simply transitively from the right.

The soldering form $\vartheta\in\Omega^1(F_G,V)$ is given by
 $$
\vartheta(\xi)=F^{-1}(\pi_*\xi)\;,
 $$
where $\xi\in\mathfrak X(F_G)$, $\pi=(\pi_o,\pi^*): F_G\to M$ is the natural projection and $F$ a local field of frames. More precisely, the R.H.S.\ is a short-hand for the operation detailed in the following.
\begin{lemma}
\label{lemma:soldering-G-structure}
The soldering form is a well-defined even $G$-equivariant horizontal form on $F_G$.
\end{lemma}
\begin{proof}
We work by localizing at a point $p\in(F_G)_o$, that is, we consider $\xi\in(\mathcal T F_G)_p$ and its image via the push-forward $\Xi=\pi_*\xi=\xi\circ\pi^*:(\cA_{M})_{\pi_o(p)}\to(\cA_{F_G})_p$. The push-forward $\Xi$ is a $\pi$-superderivation, i.e., it satisfies
$\Xi(fg)=\Xi(f)\pi^*(g)+(-1)^{|\Xi||f|}\pi^*(f)\Xi(g)$ for all $f,g\in (\cA_{M})_{\pi_o(p)}$.

Now, the sheaf of $\pi$-superderivations is isomorphic to the inverse image sheaf
$$\pi^*\mathcal T M=\mathcal A_{F_G}\otimes_{\pi_o^{-1}\cA_M}\pi_o^{-1}\mathcal T M$$
whose stalk at $p$ is
$(\mathcal A_{F_G})_p\otimes_{(\cA_M)_{\pi_o(p)}}\mathcal T M_{\pi_o(p)}$,
so we may express $\Xi$ as follows:
\begin{equation}
\Xi=\sum_{i=1}^{q}f^i(\pi^*\circ X_i)\;,
\label{eq:pi-derivation}
\end{equation}
with $f^i\in (\mathcal A_{F_G})_p$ and $X_i\in\mathcal T M_{\pi_o(p)}$. We then have
\begin{equation}
\label{eq:final-definition-soldering}
\begin{aligned}
\vartheta(\xi)&=F^{-1}(\Xi)=\sum_{i=1}^{q}f^i (F^{-1}X_i)_p\;,
\end{aligned}
\end{equation}
where we identified $X_i$ with the associated $G$-equivariant morphism
$F^{-1}X_i:F_G\to V\oplus \Pi V$, i.e., with an element of $V\otimes \cA(F_G)\cong \cA(F_G)\otimes V$.
Therefore \eqref{eq:final-definition-soldering} is an element of $(\mathcal A_{F_G})_p\otimes V$, as expected, and it is easy to see that it does not depend on the fixed expression \eqref{eq:pi-derivation}.

The other claims are obvious --  e.g., $G$-equivariancy can be checked as in the classical case thanks to Yoneda lemma.
%
\end{proof}
The following exact sequence defines the first prolongation of $\g\subset\mathfrak{gl}(V)$:
 \begin{equation}
\label{eq:first-prol-algebra}
0\to \g^{(1)}\longrightarrow \g\otimes V^*
\stackrel{\delta}\longrightarrow V\otimes\Lambda^2V^* \to0\;,
 \end{equation}
with $\delta:V\otimes V^*\otimes V^*\to V\otimes \Lambda^2V^*$ the Spencer skew-symmetrization operator (in the super-sense),
that is $\delta(w\otimes\alpha\otimes\beta)=w\otimes(\alpha\otimes\beta-(-1)^{|\alpha|\,|\beta|}\beta\otimes\alpha)$, and
$\g^{(1)}=\mathop{Ker}(\delta)=\g\otimes V^*\cap V \otimes S^2 V^*$.
If $\g_{F_G}= F_G\times \g\to F_G$ is the trivial vector bundle over $F_G$ with fiber the Lie superalgebra
$\g$ and, by abuse of notation, we denote the corresponding locally free sheaf on $(F_G)_o$ with the same symbol, then we have the exact sequence of sheaves
 \begin{equation}
\label{eq:first-prol-geometry}
0\to \g^{(1)}_{F_G}\longrightarrow \g_{F_G}\otimes \cV_{F_G}^*
\stackrel{\delta}\longrightarrow \cV_{F_G}\otimes\Lambda^2\cV_{F_G}^* \to0\;.
 \end{equation}
Let $\pi_*:\mathcal{T}F_G\to \pi^*\mathcal{T}M$ be the differential of $\pi:F_G\to M$, where
$\pi^*\mathcal T M=\mathcal A_{F_G}\otimes_{\pi_o^{-1}\cA_M}\pi_o^{-1}\mathcal T M$ is the sheaf of $\pi$-superderivations.
\begin{definition}
\hfill\par
\begin{enumerate}	
\item A {\em horizontal distribution} is a subsheaf $\mathcal{H}\subset \mathcal{T}F_G$ on $(F_G)_o$ of $\cA_{F_G}$-modules
that is complementary to the subsheaf of $\cA_{F_G}$-modules $\mathop{Ker}(\pi_*)\subset\mathcal{T}F_G$.
	\item A {\em normalization} is a supervector space $N\subset V\otimes\Lambda^2V^*$ that is complementary to $\Im(\delta)$ in \eqref{eq:first-prol-algebra}. A similar terminology is used for the associated subsheaf $\mathcal{N}\subset \mathcal{V}_{F_G}\otimes\Lambda^2\mathcal{V}_{F_G}^*$of $\cA_{F_G}$-modules on $(F_G)_o$.
\end{enumerate}
\end{definition}
Since $\mathop{Ker}(\pi_*)$ and $\mathop{Im}(\delta)$ are locally free sheaves, all normalizations and horizontal distributions are as well, see \S\ref{subbundles}. Any horizontal distribution gives an isomorphism $\mathcal H\cong \pi^*\mathcal{T}M$.

As
$\pi_o^{-1}\mathcal TM\subset\pi^*\mathcal{T}M$ naturally via $X\mapsto \pi^*\circ X$,
we have a morphism $\phi_\mathcal{H}:\pi_o^{-1}\mathcal{T}M\to\mathcal{T}F_G$: the horizontal lift $X\mapsto\phi_{\mathcal H}X\in\Gamma(\mathcal{H})$
is the right-inverse to the projection $\pi_*$.
The torsion 
of the horizontal distribution $\mathcal H$ is then defined by
 $$
T_{\mathcal{H}}(X_1,X_2)=d\vartheta(\phi_\mathcal{H}X_1,\phi_\mathcal{H}X_2)
 $$
for all $X_1,X_2\in \pi_o^{-1}\mathcal T M$, where the differential $d\vartheta\in\Omega^2(F_G,V)$ can be computed  by
the Cartan formula. 
In other words, we have a morphism of sheaves from
$\Lambda^2\pi_o^{-1}\mathcal{T}M$ to $\mathcal V_{F_G}$, which clearly extends to $\Lambda^2\pi^*\mathcal{T}M$. Since $\mathcal H\cong \pi^*\mathcal{T}M$ and the soldering form $\vartheta\in\Omega^1(F_G,V)$ induces an isomorphism of sheaves $\vartheta|_{\cH}:\cH\to \mathcal V_{F_G}$, the torsion can be in turn identified with a global even section of the trivial vector bundle
over $F_G$ with the fiber $V\otimes\Lambda^2V^*$.

Let $Fr_0=F_G$, $\mathcal{F}r_0=\mathcal{F}_G$, and define $\mathcal{F}r_1:(F_G)_o\supset \cV_o\mapsto \mathcal{F}r_1(\cV_o)$ to be the sheaf on $(Fr_0)_o$
given by
\begin{equation}
\label{eq:definition-prolongation}
\mathcal{F}r_1(\cV_o)=\big\{\cH(\cV_o)\mid \cH=\text{horiz. distrib. contained in}\;\mathcal{T}Fr_0|_{\cV_o}\;\text{such that}\;T_{\mathcal{H}}\in\mathcal{N}|_{\cV_o}\big\}\;,
\end{equation}
for any open subset $\cV_o$ of $(Fr_0)_o$.
The sheaf of Abelian groups $\mathcal{G}^{(1)}_{F_G}:\cV_o\mapsto\g^{(1)}[\cV]$ on $(Fr_0)_o$ acts
simply-transitively on $\mathcal Fr_1$ from the right, so this gives an affine bundle $Fr_1\to Fr_0$ by the geometric-algebraic correspondence.

Further prolongations follow the same scheme (literally as in \cite{St})
and yield the tower of prolongations
 \begin{equation}\label{GstrPB}
M\leftarrow Fr_0\leftarrow Fr_1\leftarrow Fr_2\leftarrow\dots.
 \end{equation}
A $G$-structure $F_G$ is called of {\em finite type} if this
tower stabilizes.

 \begin{rk}\label{LieEq}
Alternatively, a geometric structure can be defined via its Lie equations
\cite{K2}: instead of frames one considers the supermanifold $J^1(V,M)$ of $1$-jets of maps $V\to M$,
and the defining equation is a subsupermanifold $\mathcal{E}_1\subset J^1(V,M)$ which is in bijective correspondence with $F_G$ through
$V$-translations. Prolongations are defined as differential ideals
$\mathcal{E}_k=\{D_\sigma f_i=0 : |\sigma|<k\}$, where the $f_i$ are defining equations of $\mathcal{E}_1$, and $D_\sigma$ are iterated total derivatives.
 \end{rk}

Symmetries of $G$-structures may be introduced via automorphism supergroups as in \S\ref{sec:supergroups}, but a more
concrete description is in terms of super Harish-Chandra pairs. To this, we recall that
the differential $\varphi_*:\mathcal{T}M\to(\varphi_o)_*^{-1}\mathcal{T}M$ of
any automorphism $\varphi=(\varphi_o,\varphi^*)\in\Aut(M)_{\bar0}$
induces an isomorphism
$\varphi_*:\mathcal{F}r_M\to(\varphi_o)_*^{-1}\mathcal{F}r_M$
and we note that the Lie superalgebra
$\fg\otimes \cO(\cU_o)$ acts from the right on
$\mathcal{T}_M^{m|n}(\cU)$.

 \begin{definition}
\hfill
\begin{itemize}
	\item[$(i)$] An {\em automorphism} of ${\cF}_G$ is a
$\varphi\in\Aut(M)_{\bar0}$ such that
 $
\varphi_*({\cF}_G)\subset(\varphi_o)_*^{-1}{\cF}_G
 $;
\item[$(ii)$] An {\em infinitesimal automorphism} of ${\cF}_G$ on a superdomain
$\cU\subset M$ is a supervector field $X\in\mathfrak X(\cU)$ such that
 $
\mathcal L_{X}\big({\cF}_G(\cU_o)\big)\subset{\cF}_G(\cU_o)\cdot \big(\fg\otimes \cO_M(\cU_o)\big)\subset\mathcal{T}_M^{m|n}(\cU_o)
 $.
\end{itemize}
\end{definition}

The symmetries of $G$-structures are majorized by the tower of principal
bundles \eqref{GstrPB} by the classical construction of Sternberg \cite{St} (see also \cite{Gu}),
extended to the supercase in \cite{O}:
it is proven there that automorphisms of a finite type $G$-structure $F_G$ on a supermanifold $M$
form a Lie supergroup $\Aut(M,F_G)$. We will generalize this in what follows.
\subsection{Superdistributions and algebraic prolongations}
\label{sec:superdistributions}
A distribution on a supermanifold $M$ is a graded
$\cO_M$-subsheaf $\cD=\cD_{\bar0}\oplus\cD_{\bar1}$
of the tangent sheaf $\mathcal{T}M$ that is locally a direct factor. As explained in \S\ref{subbundles}, any such sheaf is locally free, so we may consider the associated
vector bundle $D$ over $M$. The latter induces a reduced subbundle $D|_{M_o}\subset TM|_{M_o}$, but as usual
with evaluations, $D|_{M_o}$ does not determine $\cD$.
We focus here on the algebraic perspective.

The weak derived flag of $\cD$ is defined as follows:
 \begin{equation}
\label{eq:weakderivedflag}
\cD^{1}=\cD \subset \cD^{2}\subset\cdots\subset\cD^i\subset\cdots\;,\qquad
\cD^i = [\cD, \cD^{i-1}]\;,
 \end{equation}
where each term is a graded $\cA_M$-subsheaf of $\mathcal TM$.
We assume the {\em bracket-generating} property $\cD^\mu=\mathcal TM$
for some $\mu>0$, and also that $\cD$ is {\em regular}, i.e.,
all subsheaves $\cD^i$ are locally direct factors in $\mathcal TM$.
\begin{example}\label{Ex35}
For many examples of (strongly) regular superdistributions, see \cite{KST}. We give here a superdistribution that is {\it not} regular.
It is a superextension of the Hilbert--Cartan equation depending on two odd variables. (In \cite{KST}, we discussed a more general extension
with $G(3)$-symmetry.)

Consider the supermanifold $\R^{5|2}$ with coordinates $(x,u,p,q,z\,|\,\theta,\nu)$, endowed with the following superdistribution of rank $(2|1)$:
 $$
\cD=\langle D_x=\p_x+p\p_u+ q\p_p+q^2\p_z,\,\p_q\,|\,D_\theta=\p_\theta+q\p_\nu+\theta\p_p+2\nu\p_z
\rangle.
 $$
We directly compute
 \begin{gather*}
[\p_q,D_x]=\p_p+2q\p_z\;,\;[\p_q,D_\theta]=\p_\nu,\\
[D_x,D_\theta]=-\theta\p_u\;,\;
[D_\theta,D_\theta]=2(\p_p+2q\p_z)\;,
\end{gather*}
so that
 $$
\cD^{2}=\langle D_x,\,\p_q,\,\p_p+2q\p_z\,|\,D_\theta,\, \p_\nu,\, \theta\p_u
\rangle.
 $$
The latter is clearly not a superdistribution, due to the presence of the supervector field $\theta\p_u$.
\end{example}

In the case of a regular $\cD$ we get an increasing filtration $\cD^i$ of $\mathcal{T}M$ by superdistributions,
which is compatible with brackets of supervector fields:
for each superdomain $U\subset M$ we have
 $$
[\cD^i(\cU),\cD^j(\cU)]\subset\cD^{i+j}(\cU)\quad \forall\ i,j>0\;,
 $$
as follows easily by induction from the Jacobi superidentities.
Note that the bracket is only $\R$-linear and it
satisfies the Leibniz superidentity as a module over $\cO_M$.
Clearly, we also have the increasing filtration of $TM|_{M_o}$ given by the classical $\mathbb Z_2$-graded vector bundles $D^i|_{M_o}$.

Setting $\opp{gr}(\mathcal T M)_{-i}=\cD^i/\cD^{i-1}$ for any $i>0$, we get a locally free sheaf  of $\cA_M$-modules $\opp{gr}(\mathcal T M)=\bigoplus_{i<0}\opp{gr}(\mathcal T M)_{i}$ over $M_o$. It has a natural structure of a sheaf of
negatively-graded Lie superalgebras over $\cA_M$:
if $v_k\in\cD^k$ with associated quotient $\bar v_k\in\cD^k/\cD^{k-1}$ for any $k>0$, then
 $
[\bar v_i,\bar v_j]\in \cD^{i+j}/\cD^{i+j-1}
$
and
$
[\bar v_i,f\bar v_j]=f[\bar v_i,\bar v_j]$ for all $f\in\cO_M$.

In particular, the bracket on
$\opp{gr}(\mathcal T M)$ is $\cO_M$-linear and thus descends to a Lie superalgebra bracket on the supervector space $\fm(x)=\bigoplus_{i<0}\g_{i}(x)$, $\g_{i}(x)=D^{-i}|_x/D^{-i-1}|_x$ for any $x\in M_o$. We shall set $\opp{gr}(TM|_{M_o})=\bigoplus_{x\in M_o}\fm(x)$, which is a classical vector bundle over $M_o$. (Indeed, this is nothing but the reduction of the vector bundle over $M$ associated to the sheaf $\opp{gr}(\mathcal T M)$.) Since supervector fields are not determined by their values at points of $M_o$, the reduction map $\opp{ev}:\opp{gr}(\mathcal T M)\to \opp{gr}(TM|_{M_o})$ can lose information.
However the entire information is recoverable in the case of
strongly regular distributions, whose correct generalization to the supercase is given in terms of the stalks:

\begin{definition}
Let $\cD$ be a regular distribution on a supermanifold $M=(M_o,\cO_M)$ that is
bracket-generating of depth $\mu$. Then $\cD$ is {\em strongly regular} if there exists a negatively-graded Lie superalgebra $\fm=\bigoplus_{-\mu\leq i<0}\g_{i}$ such that $\opp{gr}(\mathcal{T}_xM)\cong(\cO_M)_x\otimes \fm$ at any $x\in M_o$, as graded Lie superalgebras over $(\cO_M)_x$.
In this case, $\fm$ is called the {\it symbol} of $\cD$.
\end{definition}

Concretely, a regular superdistribution is strongly regular if it has a local basis of supervector fields adapted to the weak derived flag and whose brackets, after the appropriate quotients, are given by the structure constants of $\fm$ (which are real constants).

From now on, we assume all arising superdistributions to be strongly regular.
Note that by construction $\fm$ is {\it fundamental}, i.e., generated
by $\fm_{-1}$. We will also assume that $\fm$ is {\it non-degenerate},
i.e., $\g_{-1}$ contains no central elements of $\fm$ if $\mu>1$.
(Typically, one has $\mathfrak{z}(\fm)=\g_{-\mu}$.) The {\em Tanaka--Weisfeiler prolongation} of $\fm$
is the maximal $\mathbb Z$-graded Lie superalgebra $\g=\bigoplus_{i\in\mathbb Z}\g_i$ such that
 \begin{itemize}
\item[$(i)$] $\g_-=\fm$,
\item[$(ii)$] $\opp{Ker}\bigl(\opp{ad}(\g_{-1})\,|_{\g_i}\bigr)=0$ $
    \forall i\ge0$.
 \end{itemize}
It is denoted $\g=\opp{pr}(\fm)$.
The proof of the existence and uniqueness of $\opp{pr}(\fm)$
from \cite{T, W} extends verbatim to the Lie superalgebra case.
Concretely $\fg_0=\der_{gr}(\fm)$ and $\g_i$ for $i>0$ are defined
recursively by the condition (applies also for $i=0$)
 \begin{equation}\begin{split}\label{prolg}
\g_i=& \Bigl\{u:\bigoplus_{j>0}\g_{-j}\to\bigoplus_{j>0}\g_{i-j}\;\text{of $\mathbb Z$-degree}\;i
\ \text{(identified with $\opp{ad}_u=[u,\cdot]$) such that}\\
& \qquad [u,[v,w]]=[[u,v],w]+(-1)^{|u||v|}[v,[u,w]]\
\forall v,w\in\fm\Bigr\}.
 \end{split}\end{equation}
It is easy to verify that $\opp{pr}(\fm)=\bigoplus_{i\geq-\mu}\g_i$ is a Lie superalgebra.

There are several variations on this construction.
The most popular one is related to a reduction to a subalgebra
$\fg_0\subset\der_{gr}(\fm)$. Then $(i)$ in the definition
of the prolongation is changed to $\g_{\leq0}=\fm\oplus\g_0$
and $(ii)$ remains with the same formula but $\forall i>0$.
The resulting prolongation superalgebra is denoted by $\g=\opp{pr}(\fm,\g_0)$.
A more sophisticated reduction is as follows.
Assume we have already computed the prolongation to the level $\ell>0$ and
let $\g_\ell$ as $\g_0$-module be reducible:
$\g_\ell=\g_\ell'\oplus\g_\ell''$. Let also
$[\g_j,\g_{\ell-j}]\subset\g_\ell'$ for all $1\leq j\leq \ell-1$.
Then we can reduce
$\g_{-\mu}\oplus\dots\oplus\g_{\ell-1}\oplus\g_{\ell}$ to $\g_{-\mu}\oplus\dots\oplus\g_{\ell-1}\oplus\g_\ell'$
and prolong for $i>\ell$ by adapting the range of the map $u$ in \eqref{prolg}.
The result will be denoted by $\opp{pr}(\fm,\dots,\g_\ell',\dots)$,
where we list all reductions, or simply $\g=\oplus_{i=-\mu}^\infty\g_i$ if no confusion arises. An example of this higher order reduction
is projective geometry, cf.\ the classical case in \cite[Example 3]{K2},
which we will also discuss in the super-setting in \S\ref{Sproj}.

The generalized Spencer complex of a reduced prolongation algebra $\g=\oplus_{i=-\mu}^\infty\g_i$
is the Lie superalgebra cohomology complex
$\Lambda^\bullet\fm^*\otimes\g$ with the Chevalley--Eilenberg differential $\delta$:
 $$
H^{j}(\fm,\g)=H^\bullet\bigr(\Lambda^{j-1}\fm^*\otimes\g
\stackrel{\delta}\to\Lambda^j\fm^*\otimes\g
\stackrel{\delta}\to\Lambda^{j+1}\fm^*\otimes\g\bigl)\;.
 $$
It is naturally bi-graded $H^{\bullet}(\fm,\fg)=\oplus_d H^{d,\bullet}(\fm,\fg)$, where $d$ is the $\mathbb Z$-degree of a cochain, and it also admits a parity decomposition into even and odd parts as a supervector space.
It follows from definitions that $H^{i,1}(\fm,\g)=0$ if and only if $\g_{i}$ is the full prolongation of $\fm\oplus\g_{0}\oplus\cdots\oplus\g_{i-1}$,
therefore $H^{\geq 0,1}(\fm,\g)=\oplus_{i\geq 0}H^{i,1}(\fm,\g)$ encodes all possible reductions.

\subsection{Filtered geometric structures}\label{FGS}
Now we superize the notion of filtered geometric structure
as developed in \cite{T, Mo,K2}. Let $\cD$ be a strongly regular, fundamental, non-degenerate distribution on a supermanifold $M$.
The corresponding zero-order frame bundle is a principal bundle $\pi:Fr_0=\opp{Pr}_0(M,\cD)\to M$ defined via
the geometric-algebraic correspondence as the sheaf $\mathcal{F}r_0:M_o\supset\mathcal{U}_o\mapsto\mathcal{F}r_0(\mathcal{U}_o)$
of $\mathcal{A}_M$-linear Lie superalgebra isomorphisms from $\fm_M$ to $\opp{gr}(\mathcal T M)$:
\begin{equation}
\label{eq:frame-sheaf-graded}
\mathcal{F}r_0(\mathcal{U}_o)=\big\{\text{$\cA|_{\cU_o}$-linear Lie superalgebra isomorphism}\;F:\fm_M|_{\cU_o}\to\opp{gr}(\mathcal T M)|_{\cU_o}\big\}\,.
\end{equation}
Here we denoted by $\fm_M= M\times \fm\to M$ the trivial vector bundle over $M$ with the fiber
$\fm$ and, by abuse of notation, the associated locally free sheaf with the same symbol. The structure group of the bundle is
the Lie supergroup $G_0=\opp{Aut}_{gr}(\fm)$ which, by the Harish-Chandra construction, can be identified with the pair $\bigl(\opp{Aut}_{gr}(\fm)_{\bar0},\opp{der}_{gr}(\fm)\bigr)$
formed by the Lie group of degree zero automorphisms of $\fm$ and the Lie superalgebra of degree zero superderivations of $\fm$.
Since $\fm$ is fundamental, 
the structure group $G_0$ embeds into the Lie supergroup $\GL(\g_{-1})$ and
$\mathcal{F}r_0$ can be realized as a sheaf of special
$\mathcal{A}_M$-linear isomorphisms from $(\g_{-1})_M$ to $\cD$. 

More generally a first-order reduction is given by a $G_0$-reduction $F_0\subset\opp{Pr}_0(M,\cD)$ with structure group a Lie subsupergroup
$G_0\subset\opp{Aut}_{gr}(\fm)$,
which again can be thought of as an inclusion of
super Harish-Chandra pairs $\bigl((G_0)_{\bar0},\g_0\bigr)\subset \bigl(\opp{Aut}_{gr}(\fm)_{\bar0},\opp{der}_{gr}(\fm)\bigr)$.
(Often such reductions are given by order 1 invariants,
e.g., tensors or their spans. As first example, an $OSp(m|2n)$-reduction
in the case $\cD=\mathcal{T}M$ corresponds to an even supermetric $q\in S^2\mathcal{T}^*M$.)

In the next section we will construct higher order frame bundles
$Fr_i=\opp{Pr}_i(M,\cD)$, which fit into a tower of
principal bundles with projections
 \begin{equation}
\label{tower}
M\leftarrow Fr_0\leftarrow Fr_1\leftarrow Fr_2\leftarrow\dots\;,
 \end{equation}
where the principal bundle $Fr_{i}\rightarrow Fr_{i-1}$ has Abelian structure group $\g_{i}$, for all $i>0$.

The bottom projections have the structure of fiber bundles over $M$:
$Fr_1\to M$ with fiber $G^1=G_0\times\g_1$,
$Fr_2\to M$ with fiber $G^2=G^1\times\g_2$, etc,
but in general these are not principal bundles.
Similiar to the embedding $\mathcal Fr_M \subset\mathcal{T}_M^{m|n}$ described in \S\ref{sec:frame-bundles},
the higher order frame bundle $\mathcal{F}r_i$ succesively embeds into
a locally free sheaf $\widetilde{\mathcal{F}r}_i$ of
$\cA_{\mathcal{F}r_{i-1}}$-modules
 \begin{equation}
\begin{aligned}
\label{eq:locally-free-frames}
\widetilde{\mathcal{F}r}_i(\cU_o)&= \bigl\{\text{even and odd $\cA|_{\cU_o}$-linear maps}\; u:\fm_{\cU_o}\to
\mathcal{T}\mathcal{F}r_{i-1}|_{\cU_o}\bigr\}
\end{aligned}
 \end{equation}
for every superdomain $\cU\subset Fr_{i-1}$.
 \comm{
 \begin{equation*}
\begin{aligned}
\widetilde{\mathcal{F}r}_i(\cU_o)&= \bigl\{\text{even and odd $\cA|_{\cU_o}$-linear maps}\; u:\bigoplus_{j>0}(\g_{-j})_M|_{\cU_o}\to
\bigoplus_{j>0}\cD^j/\cD^{j-i-1}|_{\cU_o}\bigr\}.
\end{aligned}
 \end{equation*}
 }
This corresponds to a vector bundle, whose fiber can be further reduced
but it is not relevant here.

For any first-order reduction $F_0\subset\opp{Pr}_0(M,\cD)$
we denote the prolongation bundles by
$$F_0^{(i)}=\opp{Pr}_i(M,\cD,F_0)\subset Fr_i\;,$$ for all $i>0$.
They also fit into a tower of principal bundles analogous to \eqref{tower}.
For higher-order reductions, we restrict to a subbundle
$F_i\subset F_0^{(i)}$ for some $i>0$, and the geometric object $q$ responsible for this reduction
will have higher order. (E.g., a projective superstructure is given
by an equivalence class of superconnections. The associated Lie equations for symmetry, cf. Remark \ref{LieEq},
are of second order.) Further reductions can be imposed in a similar way on the prolongations $F_j^{(i)}$.
In the rest of the paper, in order not to overload
notations, we will mostly concentrate on pure prolongations or
first order reductions. However, the results apply in the general situation.

 \begin{definition}
A {\em filtered geometric structure} $(M,\cD,F)$ on a supermanifold $M$ consists of a
strongly regular, fundamental, non-degenerate distribution $\cD$ on $M$
and possibly some reductions $F$ of the tower \eqref{tower}. If $F$ are encoded
 by a tensorial or higher-order structure $q$, we will also use the notation $(M,\cD,q)$.
\end{definition}

\subsection{Geometric Prolongation}

Now we shall construct the higher (super) frame bundles partially following
the revision by Zelenko \cite{Z} of the constructions by
Sternberg \cite{St} and Tanaka \cite{T}
 (beware: our notations differ from theirs).
Our approach is novel in the following: we construct
the tower of bundles $\mathcal{F}_\ell$, $\ell\geq 0$,
and the frames $\varphi_{\mathcal H_\ell}$ on them,
using the entire Spencer differential (instead of a reduced one)
and recognize the choices of complements as the space of $0$- and $1$-cochains therein (with freedom being co-boundaries).

\newcommand\cT{\mathcal{T}}

\subsubsection{First prolongation}
\label{S:first-prolong}

Thanks to \S\ref{FGS}, we assume that the bundle $\pi_0 : F_0 \to M$ is already constructed.  Via pullback by $d\pi_0$, the filtration on $\cT M$ induces a filtration on $\cT F_0$:
 \begin{align}
\label{eq:filtrationTF_0}
 \vcenter{\vbox{
 \xymatrixcolsep{0.5pc}
 \xymatrix{
 \cT F_0 = & \cT^{-\mu} F_0 \ar[d]
 & \supset & \ldots
 & \supset & \cT^{-1} F_0 \ar[d]
 & \supset & \cT^0 F_0 := \ker(d\pi_0) \\
\cT M = &  \cT^{-\mu} M
 & \supset & \ldots
 & \supset &  \cT^{-1} M}}}
 \end{align}
where we also set $\cT^k F_0 = \mathcal T F_0$ for all $k<-\mu$, $\cT^k F_0 = 0_{F_0}$ for all $k> 0$ and, for simplicity, omit the inverse image symbol $\pi_0^*$ in front of each of the sheaves on $M$ in the bottom row of \eqref{eq:filtrationTF_0}. Via $d\pi_0$, we then have $$d\pi_0:\gr_-(\cT F_0) \stackrel{\cong}{\longrightarrow} \pi_0^* \gr(\cT M)$$ as sheaves of negatively $\mathbb{Z}$-graded Lie superalgebras on $F_0$, with $\pi_0^* \gr(\cT M)$ referring to the inverse image sheaf.  There is the canonical ($\cO_{F_0}$-linear, even) {\em vertical 0-trivialization}
\begin{equation}
\label{eq:vertical0frame}
\gamma_0 : \cT^0 F_0 \stackrel{\cong}{\to} \cO_{F_0}\otimes\g_0
\end{equation}
given on fundamental vector fields by $\gamma_0(\zeta_X) = X$ for all $X \in \fg_0$.

Let $\cU \subset M$ be a superdomain and consider a section of $F_0$ on $\cU$, which is identified with an $\cA_M|_{\cU_o}$-linear isomorphism $\varphi_0 = \{ \varphi^i_0 \}_{i < 0}:\fm_M|_{\cU_o} \to \gr(\cT M)|_{\cU_o}$ of zero $\mathbb Z$-degree, cf. \eqref{eq:frame-sheaf-graded}. We also call it a {\em horizontal 0-frame}. Working with stalks of inverse image sheaves as in the proof of Lemma \ref{lemma:soldering-G-structure}, one easily checks that the following is well-defined.

\begin{definition}
The tuple $\vartheta_0= \{ \vartheta^i_0 \}_{i < 0} \in \Hom(\gr_-(\cT F_0),\fm_{F_0})$ of morphisms of $\cA_{F_0}$-modules
defined by $\vartheta_0 = (\varphi_0)^{-1} \circ d\pi_0$ is called the {\em soldering form} of $F_0$.
\end{definition}
Concretely, one may compute $\vartheta_0$ by restricting to a superdomain
$\cU\times G_0\cong\pi_0^{-1}(\cU)\subset F_0$ trivialized by a fixed frame
$\varphi_0$ and recalling that all other frames are obtained by the action of $G_0$:
$\varphi_0\cdot g:=\alpha\circ\big(\varphi_0, g\big)$ for any morphism $g:\cU\to G_0$, with $\alpha:F_0\times G_0\to F_0$ the right action. The soldering form is $G_0$-equivariant. We also note that
\begin{equation}
\label{eq:solderingcomponent}
\vartheta^i_0\in\Hom(\mathcal T^i F_0/\mathcal T^{i+1} F_0,(\fm_i)_{F_0})
\end{equation}
for $i<0$, is invertible,
and that $\vartheta_0=\{ \vartheta^i_0 \}_{i < 0}$ is an isomorphism of sheaves of $\mathbb{Z}$-graded Lie superalgebras over $\cA_{F_0}$.
In particular, using $\vartheta_0$ and $\gamma_0$, we obtain a full frame of $\gr(\cT F_0)$. (We caution the reader that this does not identify $\gr(\cT F_0)$ with $\cA_{F_0}\otimes (\fm\oplus\fg_0)$ as Lie superalgebras, as the bracket on
 $\gr(\cT F_0)$ is not $\cA_{F_0}$-linear if at least one entry is a vertical supervector field.)


 For each $k\in\mathbb Z$, we set $\cT^k_0 := \cT^k F_0$. Let $i<0$  and consider the following exact sequence of sheaves on $F_0$:
 \begin{equation}\label{3seq}
0\to
\mathcal{T}_0^{i+1}/\mathcal{T}_0^{i+2}\stackrel{\iota_0^i}\longrightarrow
\mathcal{T}_0^i/\mathcal{T}_0^{i+2}\stackrel{\jmath_0^i}\longrightarrow
\mathcal{T}_0^i/\mathcal{T}_0^{i+1}\to
0.
 \end{equation}
Because all $\mathcal{T}_0^i$ are superdistributions, the image of $\iota_0^i$ in
\eqref{3seq} is a direct factor, i.e., there exists a complementary
subsheaf
$\mathcal{H}^i_0\subset\mathcal{T}_0^i/\mathcal{T}_0^{i+2}$ so that we have a splitting
 \begin{align}
 \mathcal{T}_0^i/\mathcal{T}_0^{i+2}=\mathcal{H}^i_0\oplus
\mathcal{T}_0^{i+1}/\mathcal{T}_0^{i+2}.
 \end{align}
The sequence \eqref{3seq} makes sense for $i=0$ as well, in which case it simply says that $\mathcal{H}^0_0:=\mathcal{T}_0^0$.
By the splitting lemma, $\cH^i_0$ is the kernel of a left-inverse $h^i_0$ to $\iota_0^i$
or, equivalently, the image of the right-inverse
$k^i_0 := (\opp{id}-\,\iota_0^i\circ h^i_0)\circ(\jmath_0^i)^{-1}$ to $\jmath_0^i$. We will write $h_0 = \{ h^i_0 \}_{i \leq 0}$
and $k_0 = \{ k^i_0 \}_{i \leq 0}$.

Set now $\gr^{[1]}(\cT F_0)= \bigoplus_{i \leq 0} \cT^i_0 / \cT^{i+2}_0$. Given a fixed choice of complements $\cH_0 = \{ \cH^i_0 \}_{i \leq 0}$
as above, we define a {\em 1-frame}
 \begin{align}
\label{eq:horizontal-1-frames}
 \varphi_{\cH_0} : (\fg_{\leq 0})_{F_0} \to \gr^{[1]}(\cT F_0)
 \end{align}
as the map of zero $\mathbb Z$-degree  with the components $\varphi_{\cH_0}^i:(\fg_i)_{F_0}\rightarrow \mathcal{T}_0^i/\mathcal{T}_0^{i+2}$ determined by the soldering form via the equation $\varphi_{\cH_0}^i\circ\vartheta_0^i= k^i_0$ for all $i<0$, and
 $\varphi_{\cH_0}^i\circ\gamma_0= k^i_0=\1_{\cT^0_0}$ for $i=0$.
Note that $\varphi_{\cH_0}$  is an isomorphism with image $\operatorname{im}(\varphi_{\cH_0}^i)=\operatorname{im}(k^i_0)=\cH^i_0$ for all $i\leq 0$.

In terms of the maps $h_0 = \{ h^i_0 \}_{i \leq 0}$ determined by $\cH_0$, we define the {\em 1st structure function}
$$c_{\cH_0} \in(\Lambda^2\fm_{F_0}^*\otimes_{\cO_{F_0}}\fm_{F_0})=\cO_{F_0}\otimes(\Lambda^2\fm^*\otimes\fm)$$
on the entries $v_k\in\g_k$, $k<0$, by
 \begin{align}
c_{\cH_0}(v_i,v_j)=\vartheta_0 \Bigl(h_0\bigl(\bigl[
\varphi_{\cH_0}(v_i),\varphi_{\cH_0}(v_j)\bigr]
\,\opp{mod}\mathcal{T}_0^{i+j+2}\bigr)\Bigr)
 \end{align}
and then extend by $\cA_{F_0}$-linearity to $(\g_k)_{F_0}=\cA_{F_0}\otimes \fg_k$.
 Since the filtration on $\cT F_0$ is respected by the Lie bracket, then the input of $\vartheta_0$ above lies in $\gr_{i+j+1}(\cT F_0) = \cT^{i+j+1}_0 / \cT^{i+j+2}_0$, which is mapped by $\vartheta_0$ to $\fg_{i+j+1}$.  In particular, $c_{\cH_0}$ is well-defined, it has even parity and {\em $\mathbb Z$-degree 1}, i.e., it maps $\fg_i \otimes \fg_j$ to $\fg_{i+j+1}$.
We let $\Lambda^2 \fm^*\otimes\fg=\bigoplus_{k\in\mathbb Z}(\Lambda^2 \fm^*\otimes\fg)_k$ be the natural decomposition of
$\Lambda^2 \fm^*\otimes\fg$ into $\mathbb Z$-graded components, so that $c_{\cH_0}\in \cO_{F_0}\otimes(\Lambda^2 \fm^*\otimes\fg)_1$.
The space $\fm^*\otimes\g$  has an analogous decomposition and clearly
$(\fm^*\otimes\g\bigr)_1\subset \fm^*\otimes(\fm\oplus\g_0)$.

Let us take another complement $\widetilde{\mathcal{H}_0}=\{ \widetilde\cH^i_0 \}_{i \leq 0}$ and the $1$-frame $ \varphi_{\widetilde\cH_0}$. By construction, for any $v_i\in(\g_i)_{F_0}$ with $i<0$, we have that $\varphi_{\widetilde\cH_0}(v_i)-\varphi_{\cH_0}(v_i)$
is an element of $\mathcal{T}_0^{i+1}/\mathcal{T}_0^{i+2}$, hence
\begin{align}
\label{eq:firstdifferencehorizontalframes}
\vartheta^{i+1}_0\big(\varphi_{\widetilde\cH_0}(v_i)-\varphi_{\cH_0}(v_i)\big)&=\psi(v_i)\;\;\text{for}\;\;i<-1\;,\\
\label{eq:firstdifferencehorizontalframesII}
\gamma_0\big(\varphi_{\widetilde\cH_0}(v_i)-\varphi_{\cH_0}(v_i)\big)&=\psi(v_i)\;\;\text{for}\;\;i=-1\;,
\end{align}
for some morphism $\psi:\fm_{F_0}\rightarrow(\fm\oplus\g_0)_{F_0}$ of sheaves of $\cA_{F_0}$-modules. It is clear that $\psi$ has $\mathbb Z$-degree $1$, in other words, it is an element of even parity of $\cO_{F_0}\otimes\bigl(\fm^*\otimes\g\bigr)_1$. Conversely,
given any such $\psi$, there is a unique complement $\widetilde{\mathcal{H}_0}=\{ \widetilde\cH^i_0 \}_{i \leq 0}$ for which \eqref{eq:firstdifferencehorizontalframes}-\eqref{eq:firstdifferencehorizontalframesII} hold.
\begin{lem}\label{Ldel}
Under a change of the complement,
the structure function transforms as
 $
c_{\widetilde{\mathcal{H}_0}}=c_{\mathcal{H}_0}+\delta\psi
 $,
where $\delta$ is the Chevalley--Eilenberg differential from $C^{1,1}(\fm,\fg)_{F_0}$ to $C^{1,2}(\fm,\fg)_{F_0}$.
 \end{lem}

 \begin{proof}
One directly infers from \eqref{eq:firstdifferencehorizontalframes} 
that $\vartheta^{k+1}_0\circ\widetilde{h_0}=\vartheta^{k+1}_0\circ h_0-\psi\circ\vartheta^k_0\circ\jmath^k_0$ for all $k<-1$. Suppressing upper indices for simplicity and denoting by $\Psi:\fm_{F_0}\rightarrow\gr(\cT F_0)$ the morphism obtained composing $\psi$ with the identifications \eqref{eq:vertical0frame}-\eqref{eq:solderingcomponent}, we get for all $v_l\in\fg_l$, $l\leq -1$:
 \begin{align*}
c_{\widetilde{\mathcal{H}_0}}(v_i,v_j)
&=\vartheta_0
\Bigl(\widetilde h
_0\bigl(\bigl[
\varphi_{\mathcal{\widetilde H}_0}(v_i),
\varphi_{\mathcal{\widetilde H}_0}(v_j)\bigr]
\,\opp{mod}\mathcal{T}_0^{i+j+2}\bigr)\Bigr)\\
&=
\vartheta_0
\Bigl(h
_0\bigl(\bigl[
\varphi_{\mathcal{\widetilde H}_0}(v_i),
\varphi_{\mathcal{\widetilde H}_0}(v_j)\bigr]
\,\opp{mod}\mathcal{T}_0^{i+j+2}\bigr)\Bigr)\\
&\;\;\;\;
-\psi\circ\vartheta_0\circ\jmath_0
\Bigl(\bigl[
\varphi_{\mathcal{\widetilde H}_0}(v_i),
\varphi_{\mathcal{\widetilde H}_0}(v_j)\bigr]
\,\opp{mod}\mathcal{T}_0^{i+j+2}\Bigr)\\
&=\vartheta_0
\Bigl(h
_0\bigl(\bigl[
\varphi_{\mathcal{H}_0}(v_i)+\Psi(v_i),
\varphi_{\mathcal{H}_0}(v_j)+\Psi(v_j)\bigr]
\,\opp{mod}\mathcal{T}_0^{i+j+2}\bigr)\Bigr)\\
&\;\;\;\;
-\psi\circ\vartheta_0\circ\jmath_0
\Bigl(\bigl[
\varphi_{\mathcal{H}_0}(v_i)+\Psi(v_i),
\varphi_{\mathcal{H}_0}(v_j)+\Psi(v_j)\bigr]
\,\opp{mod}\mathcal{T}_0^{i+j+2}\Bigr)\\
&=
c_{\mathcal{H}_0}(v_i,v_j)
+\vartheta_0\Bigl(\bigl[
\varphi_{\mathcal{H}_0}(v_i),
\Psi(v_j)\bigr]
\,\opp{mod}\mathcal{T}_0^{i+j+2}\Bigr)
+\vartheta_0\Bigl(\bigl[
\Psi(v_i),
\varphi_{\mathcal{H}_0}(v_j)\bigr]
\,\opp{mod}\mathcal{T}_0^{i+j+2}\Bigr)\\
&\;\;\;\;
-\psi\circ\vartheta_0
\Bigl(\bigl[
\varphi_{\mathcal{H}_0}(v_i),
\varphi_{\mathcal{H}_0}(v_j)\bigr]
\,\opp{mod}\mathcal{T}_0^{i+j+1}\Bigr)\\
&=
c_{\mathcal{H}_0}(v_i,v_j)
+\underbrace{[v_i,\psi(v_j)]-(-1)^{|v_i|\,|v_j|}
[v_j,\psi(v_i)]
-\psi([v_i,v_j])}_{\delta\psi(v_i,v_j)}\;,
 \end{align*}
where the last equality follows from the definition of structure function and the fact that the soldering form $\vartheta_0$ is a $G_0$-equivariant morphism of Lie superalgebras.
\end{proof}
 This gives the following method to restrict the $\mathcal{H}_0$'s.
Take a complement $N_1\subset(\Lambda^2\fm^*\otimes\g)_1$
to $\delta(\fm^*\otimes\g\bigr)_1$ and denote the corresponding sheaf over $F_0$ by $\mathcal{N}_1=\cA_{F_0}\otimes N_1$. Then we
define the sheaf $\opp{Pr}_1(M,\cD,F_0)$ over $F_0$ by
\begin{equation}
\label{eq:definition-prolongation-Tanaka}
\opp{Pr}_1(M,\cD,F_0)(\cV_o)=\big\{\cH_0(\cV_o)\mid \cH_0=\{ \cH^i_0 \}_{i \leq 0}\;\text{on $\cV_0$}\;\text{such that}\;
c_{\mathcal{H}_0}\in \mathcal{N}_1|_{\cV_o}
\big\}\;,
\end{equation}
equivalently the collection of the associated $1$-frames \eqref{eq:horizontal-1-frames}. By \eqref{eq:firstdifferencehorizontalframes}-\eqref{eq:firstdifferencehorizontalframesII} and Lemma \ref{Ldel}, this is a
principal bundle $\pi_{1}:\opp{Pr}_1(M,\cD,F_0)\to F_0$ over $F_0$ with Abelian structure group $G_1=\exp(\fg_1)$ consisting of all
$\psi\in \bigl(\fm^*\otimes\g\bigr)_1$ in the kernel of the Spencer operator $\delta$, i.e., of all elements of the first prolongation $\fg_1=\fg_0^{(1)}$.

The affine bundle $\opp{Pr}_1(M,\cD,F_0)$ may have a further reduction resulting in the first frame bundle
$F_1\subseteq\opp{Pr}_1(M,\cD,F_0)$.
By \eqref{eq:horizontal-1-frames},
a section $\varphi_1$ of $F_1$ over $\cV\subset F_0$ can be equivalently thought
as an element $\varphi_{1} : (\fg_{\leq 0})_{F_0} \to \gr^{[1]}(\cT F_0)$ such that
$\mathcal{H}_0=\opp{Im}(\varphi_1)$.

\smallskip

 \subsubsection{Higher frame bundles}
\label{subsec:hfb}
The higher frame bundles are constructed similarly.
We will not specify structure reductions anymore, denoting
(reduced or non-reduced) frame bundles by the same symbol $F_i$.

The construction is inductive.  For $\ell \geq 1$, suppose that we have constructed:
\begin{enumerate}
\item the affine bundle $\pi_{\ell} : F_\ell \to F_{\ell-1}$ with Abelian structure group $G_\ell$ with associated
Lie superalgebra $\fg_\ell = \fg_{\ell-1}^{(1)}$;
\item a decreasing filtration
\begin{equation}
\label{eq:dec-filtration}
\cT F_{\ell-1}=\cT^{-\mu} F_{\ell-1} \supset \ldots \supset \cT^{\ell-1} F_{\ell-1}=\ker(d\pi_{\ell-1})
\end{equation}
on $F_{\ell-1}$ with associated soldering form and vertical $(\ell-1)$-trivialization
\begin{equation}
\label{eq:solderingandvertical-l-1}
\begin{aligned}
\vartheta_{\ell-1}&= \{ \vartheta^i_{\ell-1} \}_{i < \ell-1} \in \Hom(\gr_{<\ell-1}(\cT F_{\ell-1}),(\fg_{<\ell-1})_{F_{\ell-1}})\;,\\
\gamma_{\ell-1}&: \cT^{\ell-1} F_{\ell-1} \stackrel{\cong}{\to} \cO_{F_{\ell-1}}\otimes\g_{\ell-1}\;.
\end{aligned}
\end{equation}
Henceforth, we write $\cT^i_j := \cT^i F_j$, with the understanding that  $\cT^i F_j = \mathcal T F_j$ for $i<-\mu$ and $\cT^i F_j = 0_{F_j}$ for $i> j$;
\item an $\ell$-frame, which is an injective morphism of sheaves of $\cA_{F_{\ell-1}}$-modules
\begin{align}
 \varphi_\ell: (\fg_{\leq \ell-1})_{F_{\ell-1}} \to \gr^{[\ell]}(\cT F_{\ell-1}),
 \end{align}
 where $\gr^{[\ell]}(\cT F_{\ell-1}) := \bigoplus_{i \leq \ell-1} \cT^i_{\ell-1} / \cT^{i+\ell+1}_{\ell-1}$.  (Equivalently, we have a section $\varphi_\ell$ of $\pi_{\ell} : F_\ell \to F_{\ell-1}$ over a superdomain $\cV\subset F_{\ell-1}$.)  The $\ell$-frame  $\varphi_\ell = \{ \varphi_\ell^i \}_{i \leq\ell-1}$ selects horizontal subspaces $\cH_{\ell-1} = \{ \cH^i_{\ell-1} \}_{i \leq \ell-1} = \opp{Im}(\varphi_\ell)$, with
 \begin{align}
 \cH_{\ell-1}^i \subset \cT^i_{\ell-1} / \cT^{i+\ell+1}_{\ell-1}
 \end{align}
for $i<0$, and
 \begin{align}
 \cH_{\ell-1}^i \subset \cT^i_{\ell-1}
 \end{align}
for $0\leq i\leq\ell-1$, with $\cH_{\ell-1}^{\ell-1}=\cT^{\ell-1}_{\ell-1}$. The component $\varphi^{\ell-1}_\ell$ is the identification of $(\fg_{\ell-1})_{F_{\ell-1}}$ with $\cT^{\ell-1}_{\ell-1}$ given by the vertical $(\ell-1)$-trivialization and the component $\varphi^{i}_\ell$ does not vary upon the action of the structure group $G_\ell$, for any $i\geq 0$.

Since the framework of supermanifolds does not allow to work at a fixed point, what we will really need is
the pull-back bundle $\pi_\ell^* F_\ell\to F_\ell$ with its canonical section. In other words $\varphi_\ell: (\fg_{\leq\ell-1})_{F_{\ell}} \to \pi_\ell^*\gr^{[\ell]}(\cT F_{\ell-1})$ and the sheaf $\cH_{\ell-1}^i$ is a subsheaf of $\pi^*_\ell\big(\cT^i_{\ell-1} / \cT^{i+\ell+1}_{\ell-1}\big)$ and $\pi^*_\ell\cT^i_{\ell-1}$ for, respectively, negative and non-negative indices.
\end{enumerate}
In this subsection, we construct the new horizontal subspaces $\cH_\ell = \{ \mathcal{H}^i_\ell \}_{i \leq\ell}$, which includes the construction for the non-negative indices $0\leq i \leq \ell$.

Via pullback by $d\pi_{\ell}$, the filtration \eqref{eq:dec-filtration} on $\cT F_{\ell-1}$ induces a filtration on $\cT F_\ell$:
\begin{align*}\xymatrixcolsep{0.5pc}
 \xymatrix{
 \cT F_\ell = & \cT^{-\mu}_\ell \ar[d]
 & \supset & \ldots
 & \supset & \cT^{\ell-1}_\ell \ar[d]
 & \supset & \cT^\ell_\ell = \ker(d\pi_{\ell})\\
 \cT F_{\ell-1} = & \cT^{-\mu}_{\ell-1}
 & \supset & \ldots
 & \supset & \cT^{\ell-1}_{\ell-1}}
 \end{align*}
where, as usual, we omit inverse image sheaf symbol $\pi_{\ell}^*$ for the sheaves in the bottom row.
It is important to note for later use that the filtration on $\cT F_\ell$ is respected by the Lie bracket only for {\it non-positive} filtration indices
(because of the Leibniz rule). For instance, the vertical subbundle $\cT^\ell_\ell$ is integrable and it also preserves all $\cT^i_\ell$ for $-\mu\leq i\leq\ell-1$, since the latter bundle is induced via pull-back. Similarly one has
\begin{equation}
\label{Liebracket-positive-smaller}
[\cT^{m}_\ell,\cT_\ell^{n}]\subset \cT^{n}_\ell
\end{equation}
for all $0\leq m\leq \ell$ and $-\mu\leq n\leq m$.

Note the isomorphism $\gr_{<\ell}(\cT F_\ell) \stackrel{\cong}{\to} \pi_{\ell}^*\gr(\cT F_{\ell-1})$ as sheaves
of $\cA_{F_\ell}$-modules.
The {\em soldering form} $\vartheta_\ell= \{ \vartheta^i_\ell \}_{i < \ell} \in \Hom(\gr_{<\ell}(\cT F_\ell), (\fg_{<\ell})_{F_{\ell}})$ on $F_\ell$ is defined by composing this isomorphism with the soldering form and the vertical trivialization on $F_{\ell-1}$, i.e., it is the pull-back via $\pi_\ell$ of the forms \eqref{eq:solderingandvertical-l-1}. We also have a canonical {\em vertical $\ell$-trivialization}
$\gamma_\ell: \cT^\ell_\ell\to \cO_{F_\ell}\otimes\g_\ell$.

 Consider the following two exact sequences
of sheaves over $F_\ell$ (the sequence over
$F_{\ell-1}$ lifts via the inverse image operation, 
the notation
of which we suppress again), with $i<0$:
\[
\xymatrix{
0 \ar[r] & \mathcal{T}_\ell^{i+\ell+1}/\mathcal{T}_\ell^{i+\ell+2}
\ar@/^/[r]^>>>{\iota^i_\ell} \ar@{.>}[d]|{0} &
\mathcal{T}_\ell^i/\mathcal{T}_\ell^{i+\ell+2}
\ar@/^/[r]^>>>{\jmath^i_\ell}\ar@{-->}[l]^{}
\ar[d]_{b^i_\ell} &
\mathcal{T}_\ell^i/\mathcal{T}_\ell^{i+\ell+1} \ar[r]\ar@{-->}[l]
\ar[ld]_{a^i_\ell} & 0 \\
0 \ar[r] & \mathcal{T}_{\ell-1}^{i+\ell}/\mathcal{T}_{\ell-1}^{i+\ell+1}
\ar@/^/[r]^>>>{\iota^i_{\ell-1}} &
\mathcal{T}_{\ell-1}^i/\mathcal{T}_{\ell-1}^{i+\ell+1}
\ar@/^/[r]^>>>{\jmath^i_{\ell-1}}\ar@{-->}[l]^{
}  &
\mathcal{T}_{\ell-1}^i/\mathcal{T}_{\ell-1}^{i+\ell} \ar[r]\ar@{-->}[l] & 0 \\
}
 \]
For all $i<0$, the differential $d\pi_{\ell}$ induces the map
$a^i_\ell : \cT^i_\ell / \cT^{i+\ell+1}_\ell \to \pi_\ell^*\big(\cT^i_{\ell-1} / \cT^{i+\ell+1}_{\ell-1}\big)$, which is an isomorphism. We then define the middle vertical map by $b^i_\ell := a^i_\ell \circ \jmath^i_\ell$. As in \S\ref{S:first-prolong}, we will later see in Lemma \ref{lem:splitting} that these sequences split, with dashed lines indicating left-inverses $h^i_j$ to $\iota^i_j$ and right-inverses $k^i_j := (\opp{id}-\,\iota_j^i\circ h^i_j)\circ(\jmath_j^i)^{-1}$ to $j^i_j$. We consider $\cH^i_\ell$ satisfying
 \begin{align}
\label{eq:choose-complement}
 \cT^i_\ell / \cT^{i+\ell+2}_\ell \supset (b^i_\ell)^{-1} \big(\cH^i_{\ell-1}\big) =\cH^i_\ell \oplus \opp{Im}(\iota^i_\ell)\;,
 \end{align}
so that the restriction of $b^i_\ell$ to $\cH^i_\ell$ defines an isomorphism $\cH^i_\ell\stackrel{\cong}{\to}\cH^i_{\ell-1}$. (We recall for reader's convenience that by definition
$\cH_{\ell-1}^i \subset \pi^*_\ell\big(\cT^i_{\ell-1} / \cT^{i+\ell+1}_{\ell-1}\big)$  and that $\opp{Ker}(b^i_\ell)=\opp{Ker}(\jmath^i_\ell)=\opp{Im}(\iota^i_\ell)$.)

For all $0\leq i\leq \ell-1$, we let $b_\ell^i:
\mathcal{T}_\ell^i\to\pi_{\ell}^*\mathcal{T}_{\ell-1}^i$
be the projection induced by the differential, whose kernel
is $\mathcal{T}_\ell^\ell$.
We then have the following exact sequences
 \[
\xymatrix@C+1pc{
0 \ar[r] & \mathcal{T}_\ell^{\ell}
\ar@/^/[r]^>>>{\iota^i_\ell} \ar@{.>}[d]|{0} &
\mathcal{T}_\ell^i
\ar@/^/[r]^>>>{\jmath^i_\ell}\ar@{-->}[l]^{}
\ar[d]_{b^i_\ell} &
\mathcal{T}_\ell^i/\mathcal{T}_\ell^{\ell} \ar[r]\ar@{-->}[l]
\ar[ld]_{a^i_\ell} & 0 \\
0 \ar[r] & \mathcal{T}_{\ell-1}^{\ell-1}
\ar@/^/[r]^>>>{\iota^i_{\ell-1}} &
\mathcal{T}_{\ell-1}^i
\ar@/^/[r]^>>>{\jmath^i_{\ell-1}}\ar@{-->}[l]^{
}  &
\mathcal{T}_{\ell-1}^i/\mathcal{T}_{\ell-1}^{\ell-1} \ar[r]\ar@{-->}[l] & 0 \\
}
 \]
with $a^i_\ell$ the isomorphism induced by $b^i_\ell$ on the quotient.
We choose a complement $\cH^i_\ell$ to $\mathcal{T}_\ell^\ell$ in
$(b_\ell^i)^{-1}(\cH^i_{\ell-1})$ for every $0\leq i\leq\ell-1$ as before in \eqref{eq:choose-complement}, namely
 \begin{align}
\label{eq:choose-complement+}
\cT^i_\ell \supset (b^i_\ell)^{-1} \big(\cH^i_{\ell-1}\big) =\cH^i_\ell \oplus \cT^\ell_\ell\;,
 \end{align}
and set $\cH_{\ell}^{\ell}=\cT^{\ell}_{\ell}$. Dashed lines indicate
the respective inverses $h^i_j$ and $k^i_j$ to $\iota^i_j$ and $j^i_j$.

Note that $\cH^i_\ell\stackrel{\cong}{\to}\cH^i_{\ell-1}$ via $b^i_\ell$ for all $i\leq \ell-1$, the inverse of which we denote by $(b^i_\ell)^{-1}$.
We set
$$\gr^{[\ell+1]}(\cT F_\ell) := \bigoplus_{i\leq\ell} \cT^i_\ell / \cT^{i+\ell+2}_\ell$$ and define an $(\ell+1)$-frame
$
 \varphi_{\ell+1} : (\fg_{\leq\ell})_{F_\ell} \to \gr^{[\ell+1]}(\cT F_\ell)
$ by $\varphi^i_{\ell+1} := (b^i_\ell)^{-1}\circ \varphi^i_{\ell}$ for $i\leq \ell-1$ and using the principal bundle structure
via $\varphi^i_{\ell+1}:=(\gamma_\ell)^{-1}$ for $i=\ell$. We also note that $\cH_\ell =\{\cH^i_\ell\}_{i\leq\ell}=\opp{Im}(\varphi_{\ell+1})$ and that each component
$$\varphi^i_{\ell+1}:\cA_{F_\ell}\otimes\g_i\to\mathcal{T}_{\ell}^i/\mathcal{T}_{\ell}^{i+\ell+2}$$
is an embedding that projects to an isomorphism $\cA_{F_\ell}\otimes\g_i\cong
\mathcal{T}_{\ell}^i/\mathcal{T}_{\ell}^{i+1}$. (Because $\mathcal{T}^k_\ell=0_{F_\ell}$ for $k>\ell$, there is no truncation for $i\geq -1$, that is $\varphi^{i}_{\ell+1}(v)$ is a vector field on $F_\ell$ for any $v\in\g_i$ with $i\geq-1$.)
\begin{lemma}
\label{lem:complements}
Given $\cH_\ell = \{ \mathcal{H}^i_\ell \}_{i \leq\ell}$, we have
\begin{itemize}
	\item[$(i)$] $\cT^i_\ell/\cT^{i+\ell+2}_\ell=\cT^{i+1}_\ell/\cT^{i+\ell+2}_\ell\oplus\cH_\ell^i$ for all $i\leq \ell$,
	\item[$(ii)$] $\cT^{i+s}_\ell/\cT^{i+\ell+2}_\ell=\cT^{i+s+1}_\ell/\cT^{i+\ell+2}_\ell\oplus \pi^{s}_\ell(\cH^{i+s}_\ell)$ for all $0\leq s\leq \ell+1$ and $i+s\leq \ell$,
\end{itemize}
where $\pi_\ell^s:\cT^{i+s}_\ell/\cT^{i+s+\ell+2}_\ell\to\cT^{i+s}_\ell/\cT^{i+\ell+2}_\ell$ is the natural projection.
\end{lemma}
We omit the proof by induction of $(i)$ for the sake of brevity. Claim $(ii)$ follows from $(i)$ considering $i+s$ instead of $i$ and taking the quotient by $\cT^{i+\ell+2}_{\ell}/\cT^{i+s+\ell+2}_\ell$. We note that $\pi^{s}_\ell(\cH^{i+s}_\ell)\cong \cH^{i+s}_\ell$ in $(ii)$.
The following result is then a straightforward consequence of $(ii)$.
 \begin{proposition}
\label{lem:splitting}
Given $\cH_\ell = \{ \mathcal{H}^i_\ell \}_{i \leq\ell}$, we have
\begin{equation}
\label{eq:splitting-at-level-l}
\begin{aligned}
\cT^i_\ell/\cT^{i+\ell+2}_\ell&=\bigoplus_{0\leq s\leq\ell} \pi_\ell^{s}(\cH^{i+s}_\ell)\oplus \cT^{i+\ell+1}_\ell/\cT^{i+\ell+2}_\ell\\
&\cong \bigoplus_{0\leq s\leq\ell} \cH^{i+s}_\ell\oplus \cT^{i+\ell+1}_\ell/\cT^{i+\ell+2}_\ell
\end{aligned}
\end{equation}
for all $i\leq\ell-1$, where $\pi_\ell^s:\cT^{i+s}_\ell/\cT^{i+s+\ell+2}_\ell\to\cT^{i+s}_\ell/\cT^{i+\ell+2}_\ell$ is the natural projection.
In particular
\begin{equation}
\opp{Ker}(h^i_\ell) = \opp{Im}(k^i_\ell)\cong
\begin{cases}
\bigoplus_{i\leq s\leq \ell+i} \cH^{s}_\ell&\text{if}\;\;i<0\;,\\
\bigoplus_{i\leq s\leq\ell-1} \cH^{s}_\ell&\text{if}\;\;0\leq i\leq \ell-1\;,\\
\end{cases}
\end{equation}
are the complements to $\cT^{i+\ell+1}_\ell/\cT^{i+\ell+2}_\ell$ and $\cT^\ell_\ell$, respectively.
\end{proposition}

%


Let us take another complement $\widetilde{\mathcal{H}_\ell}=\{ \widetilde\cH^i_\ell \}_{i \leq \ell}$ constructed as before and the associated $(\ell+1)$-frame $ \widetilde\varphi_{\ell+1}$. By construction, for any $v_i\in(\g_i)_{F_\ell}$, we have that $\widetilde\varphi_{\ell+1}(v_i)-\varphi_{\ell+1}(v_i)$
is an element of $\mathcal{T}_\ell^{i+\ell+1}/\mathcal{T}_\ell^{i+\ell+2}$ if $i<0$
and $\cT^\ell_\ell$ if $0\leq i\leq \ell-1$. Hence
\begin{align}
\label{eq:firstdifferencehorizontalframes-l}
\vartheta^{i+\ell+1}_\ell\big(\widetilde\varphi_{\ell+1}(v_i)-\varphi_{\ell+1}(v_i)\big)&=\psi(v_i)\;\;\text{for}\;\;i<-1\;,\\
\label{eq:firstdifferencehorizontalframesII-l}
\gamma_\ell\big(\widetilde\varphi_{\ell+1}(v_i)-\varphi_{\ell+1}(v_i)\big)&=\psi(v_i)\;\;\text{for}\;\;-1\leq i\leq \ell-1\;,
\end{align}
for some morphism $\psi:(\fg_{\leq\ell-1})_{F_\ell}\rightarrow(\fg_{\leq\ell})_{F_\ell}$ of sheaves of $\cA_{F_\ell}$-modules. It is clear that
the components
\begin{align}
\psi_{-}&:\fm_{F_\ell}\to\fg_{F_\ell}\;,\\
\psi_{+}&:(\fg_0\oplus\cdots\oplus\fg_{\ell-1})_{F_\ell}\to(\fg_{\ell})_{F_\ell}
\end{align}
are elements of even parity, with the first component having $\mathbb Z$-degree $(\ell+1)$. In other words $\psi_{-}$ is an even element of $\cO_{F_\ell}\otimes\bigl(\fm^*\otimes\g\bigr)_{\ell+1}$ and $\psi_{+}$ of
$\cO_{F_\ell}\otimes\bigl((\fg_{\leq \ell-1}^+)^*\otimes\g_{\ell}\bigr)$.
Conversely,
given any such $$\psi=\psi_{-}+\psi_{+}$$ there is a unique complement $\widetilde{\mathcal{H}_\ell}=\{ \widetilde\cH^i_\ell \}_{i \leq \ell}$ with the required properties.
\subsubsection{Normalization conditions}
\label{subsec:normcond}
In this section, we detail the normalization conditions to be enforced on the $(\ell+1)$-frames.
Since the Lie bracket is compatible with the filtration on $\cT F_\ell$ only for non-positive filtration indices,
we first need to collect some finer properties satisfied by the frames.
\begin{lemma}
\label{lemma:specialbracketsI}
Let $\zeta\in\cT^{k}_{\ell}$ with $0\leq k\leq \ell$ and $\varphi_{\ell+1} : (\fg_{\leq\ell})_{F_\ell} \to \gr^{[\ell+1]}(\cT F_\ell)$ be an $(\ell+1)$-frame. Then, for $v_i\in\fg_i$, $i<0$:
\begin{equation}
\begin{aligned}
\;[\zeta,\varphi_{\ell+1}(v_i)]&\in
\begin{cases}
\cT^{k-1}_{\ell}\quad&\text{if}\;i=-1;\\
\cT^{i+k}_{\ell}/\cT^{k}_{\ell}\quad&\text{if}\;k-\ell-2< i\leq-2;\\
\cT^{i+k}_{\ell}/\cT^{i+\ell+2}_{\ell}\quad&\text{if}\; -\mu\leq i\leq k-\ell-2.\\
\end{cases}
\end{aligned}
\end{equation}
\end{lemma}
Given a choice of complements $\cH_\ell=\{\cH^i_\ell\}_{i\leq\ell}$, we define the $(\ell+1)$-th {\em horizontal structure function}
$c^{-}_{\mathcal{H}_\ell}\in
\cO_{F_\ell} \otimes
(\Lambda^2\fm^*\otimes\g_{< \ell})$ on the entries $v_k\in\g_k$, $k<0$, by
 \begin{align}
\label{eq:horizoontalstructurefunction!}
c^-_{\mathcal{H}_\ell}(v_i,v_j)=\vartheta^{i+j+\ell+1}_\ell \Bigl(h^{i+j}_\ell
\bigl(\bigl[\varphi_{\ell+1}
(v_i),\varphi_{\ell+1}(v_j)\bigr]
\,\opp{mod}\mathcal{T}^{i+j+\ell+2}_\ell
\bigr)\Bigr)
 \end{align}
and extending by $\cA_{F_\ell}$-linearity to the entries from $(\g_k)_{F_\ell}=\cA_{F_\ell}\otimes \fg_k$.
Evidently $[\cT^i_\ell,\cT^j_\ell]\subset\cT^{i+j}_\ell$ as $i,j<0$. However,  the Lie bracket is compatible with the filtration on $\cT F_\ell$ only for the non-positive filtration indices, so the fact that \eqref{eq:horizoontalstructurefunction!} is well-defined deserves an additional explanation: we show that the input of $\vartheta_\ell\circ h_\ell$ above is a well-defined element in $ \cT^{i+j}_\ell / \cT^{i+j+\ell+2}_\ell$.
\begin{lemma}
The horizontal structure function $c^{-}_{\mathcal{H}_\ell}$ is well-defined.
\end{lemma}
\begin{proof}
Recall $i,j<0$. If both $i+\ell+2$ and $j+\ell+2$ are non-positive, the claim follows immediately from the general properties of the Lie bracket. Otherwise we may assume, say, $i+\ell+2>0$, $j\leq i$.
Now
$[\cT^{i+\ell+2}_\ell,\cT_\ell^{j+\ell+2}]\subset \cT^{j+\ell+2}_\ell\subset \cT^{i+j+\ell+2}_\ell$ by \eqref{Liebracket-positive-smaller} and
$[\cT^{i+\ell+2}_\ell,\varphi_{\ell+1}(v_j)]\equiv 0\!\!\mod \cT^{i+j+\ell+2}_\ell$
by Lemma \ref{lemma:specialbracketsI}, so we are left to deal with $[\cT^{j+\ell+2}_\ell,\varphi_{\ell+1}(v_i)]$.

If $j+\ell+2\leq 0$, then $[\cT^{j+\ell+2}_\ell,\varphi_{\ell+1}(v_i)]\equiv 0\!\!\mod \cT^{i+j+\ell+2}_\ell$
by the general property of the Lie bracket, and the same result follows from Lemma \ref{lemma:specialbracketsI} if $j+\ell+2> 0$.
\end{proof}


Note that $c^-_{\cH_\ell}$ has {\em $\mathbb Z$-degree $(\ell+1)$}, i.e., it is an element of $C^{\ell+1,2}(\fm,\fg)_{F_\ell}$. As in Lemma \ref{Ldel}:

 \begin{lem}
Under a change of complement,
the structure function transforms as
 $
\widetilde {c^-_{{\mathcal{H}_\ell}}}=c^-_{\mathcal{H}_\ell}+\delta\psi_{-}
 $,
where $\delta$ is the Chevalley--Eilenberg differential from $C^{\ell+1,1}(\fm,\fg)_{F_\ell}$ to $C^{\ell+1,2}(\fm,\fg)_{F_\ell}$.
\end{lem}
We know from Proposition \ref{lem:splitting} that for $0\leq k\leq \ell-1$, a complement of $\mathcal{T}_\ell^\ell$
in $\cT^k_\ell$ is $\oplus_{s=k}^{\ell-1}\cH^s_\ell$, so there is a projection 
$\pr_s^{k}$
from $\mathcal{T}^k_\ell$ to $\cH^{s}_\ell\cong \cA_{F_\ell}\otimes\g_s$ for any $k\leq s\leq \ell-1$, where the last isomorphism is given by the soldering form.
The analogous projection from $\mathcal{T}^k_\ell$ to $\cT^k_\ell/\cT^{k+\ell+2}_\ell$ and then to $\cH^s_\ell$ for
any $k\leq s\leq\ell+k$ is defined for all $k<0$. We note that $\opp{Ker}(\pr_s^{k})\supset \cT^{s+1}_\ell$.

Again we need some finer properties of the frames.
\begin{lemma}
\label{lemma:specialbracketsII}
Let $\varphi_{\ell+1} : (\fg_{\leq\ell})_{F_\ell} \to \gr^{[\ell+1]}(\cT F_\ell)$ be an $(\ell+1)$-frame and $i\geq 0$. We then have $[\zeta,\varphi_{\ell+1}(v_i)]\in \cT^{k}_{\ell}$ for all
$\zeta\in\cT^{k}_{\ell}$ with $i< k\leq \ell$.
\end{lemma}
Note that the claim of Lemma \ref{lemma:specialbracketsII} is automatically satisfied also for $k\leq i$, due to \eqref{Liebracket-positive-smaller}.
Let 
$\g_{\leq \ell-1}^+=\fg_0\oplus\cdots\oplus\fg_{\ell-1}$ as before.
The $(\ell+1)$-th {\em vertical structure function}
 $$
c_{\mathcal{H}_\ell}^+\in
\cO_{F_\ell}\otimes(\g_{\leq \ell-1}^+)^*\otimes(\fm^*\otimes\g)_{\ell}
\subset\cO_{F_\ell}\otimes\opp{Hom}(\fm\otimes\g_{\leq \ell-1}^+,\g_{\leq \ell-1})
 $$
is defined 
as the $\cA_{F_\ell}$-linear extension of the following formula
 $$
c^+_{\mathcal{H}_\ell}(v_i,v_j)=\vartheta_\ell^{i+\ell}
\Bigl(\pr^{i+j}
_{i+\ell}\bigl[
\varphi_{\ell+1}(v_i),\varphi_{\ell+1}(v_j)
\bigr]\Bigr)
$$
where $v_i\in\g_i$ with $i<0$, and $v_j\in\g_j$ with $0\le j\leq\ell-1$.
By Lemma \ref{lemma:specialbracketsI} one sees that the input of $\vartheta_\ell^{i+\ell}\circ\pr^{i+j}_{i+\ell}$ is in $\cT^{i+j}_\ell$ with some ambiguity, which in this case lies in $[\cT^{i+\ell+2}_\ell,\varphi_{\ell+1}(v_j)]$. By Lemma \ref{lemma:specialbracketsII} and \eqref{Liebracket-positive-smaller},
the input lies then in $\cT_\ell^{i+j}/\cT_\ell^{i+\ell+2}$, so that
$c^+_{\mathcal{H}_\ell}$ is well-defined.

\begin{lem}
Changing complement, the structure function transforms as
 $
\widetilde{c^+_{{\mathcal{H}_\ell}}}=c^+_{\mathcal{H}_\ell}+\bar\delta\psi_{+},
 $
where
\begin{equation}
\label{eq:tensorproductdifferential}
\bar\delta=\delta\otimes\opp{id}:
C^{\ell,0}(\fm,\fg)_{F_\ell}\otimes(\g_{\leq \ell-1}^+)_{F_\ell}^* \to C^{\ell,1}(\fm,\fg)_{F_\ell}\otimes(\g_{\leq \ell-1}^+)_{F_\ell}^*
\end{equation}
is the tensor product of the Chevalley--Eilenberg differential from $C^{\ell,0}(\fm,\fg)_{F_\ell}$ to $C^{\ell,1}(\fm,\fg)_{F_\ell}$ with the identity of $(\g_{\leq \ell-1}^+)_{F_\ell}^*$.
\end{lem}

 \begin{proof}

If $0\leq k\leq \ell-1$, it is clear from Proposition \ref{lem:splitting} and \eqref{eq:firstdifferencehorizontalframesII-l} that $\widetilde{\pr_s^{k}}(Y^k)\equiv\pr_s^{k}(Y^k)\!\!\mod \cT^{\ell}_\ell$, for all $Y^k\in\cT^{k}_\ell$. Similarly $\widetilde{\pr_s^{k}}(Y^k)\equiv\pr_s^{k}(Y^k)\!\!\mod \mathcal{T}_\ell^{s+\ell+1}/\mathcal{T}_\ell^{s+\ell+2}$ if $k<0$.
Since $\opp{Ker}(\vartheta_\ell^{i+\ell})=\cT^{i+\ell+1}_\ell\supset\cT^{\ell}_\ell+\cT_\ell^{i+2\ell+1}$, one directly infers that $\vartheta_\ell^{i+\ell}\circ\widetilde{\pr^{i+j}_{i+\ell}}=\vartheta_\ell^{i+\ell}\circ\pr^{i+j}_{i+\ell}$.

Denoting by $\Psi:(\fg_{\leq \ell-1})_{F_\ell}\rightarrow\gr(\cT F_\ell)$ the morphism obtained by composing $\psi$ with the inverses of the soldering form and vertical $\ell$-trivialization, we then have
\begin{equation}
\begin{aligned}
\bigl[\Psi(v_i),
\varphi_{\ell+1}(v_j)\bigr]&\equiv 0\mod \cT^{i+\ell+1}_\ell\\
\bigl[\Psi(v_i), \Psi(v_j)\bigr]&\equiv 0\mod \cT^{i+\ell+1}_\ell
\end{aligned}
\end{equation}
by Lemma \ref{lemma:specialbracketsII} and since bracketing with $\cT^\ell_\ell$ preserves all other bundles in the filtration.
On the other hand $[\varphi_{\ell+1}(v_i),\Psi(v_j)\bigr]\in\cT_\ell^{i+\ell}$ by Lemma \ref{lemma:specialbracketsI}, up to elements in $\cT^{i+\ell+1}_\ell$.

Since $\cT^{i+\ell+1}_\ell\subset\opp{Ker}(\pr_{i+\ell}^{i+j})$, we then have
\begin{equation}
\label{eq:vertical-structurefunction-changecomplement}
\begin{aligned}
\widetilde{c^+_{{\mathcal{H}_\ell}}}(v_i,v_j)&=
\vartheta_\ell^{i+\ell}
\Bigl(\pr^{i+j}
_{i+\ell}\bigl[
\varphi_{\ell+1}(v_i)+\Psi(v_i),
\varphi_{\ell+1}(v_j)+\Psi(v_j)\bigr]\Bigr)\\
&=c^+_{\mathcal{H}_\ell}(v_i,v_j)+
\vartheta_\ell^{i+\ell}
\Bigl(\pr^{i+j}
_{i+\ell}\bigl[
\varphi_{\ell+1}(v_i),\Psi(v_j)\bigr]\Bigr)\\
&=c^+_{\mathcal{H}_\ell}(v_i,v_j)+\bar\delta\psi(v_i,v_j)\;,
\end{aligned}
\end{equation}
where
$\bar\delta\psi(v_i,v_j)=[v_i,\psi(v_j)]$.
 \end{proof}
Choose complements
$N_{\ell+1}^-\subset C^{\ell+1,2}(\fm,\fg)$ to $\delta C^{\ell+1,1}(\fm,\fg)$,
$N^+_{\ell+1}\subset C^{\ell,1}(\fm,\fg)$ to $\delta C^{\ell,0}(\fm,\fg)$, and consider the sheaves $\mathcal{N}^-_{\ell+1}=\cA_{F_\ell}\otimes N^{-}_{\ell+1}$ and $\mathcal{N}^+_{\ell+1}=\cA_{F_\ell}\otimes N^{+}_{\ell+1}\otimes(\g_{\leq \ell-1}^+)^*$ over $F_\ell$. We then require that
\begin{equation}\label{cHell}
c^\pm_{\mathcal{H}_\ell}\in \mathcal{N}^\pm_{\ell+1}\;,
 \end{equation}
and define the sheaf $\opp{Pr}_{\ell+1}(M,\cD,F_0)$ over $F_\ell$ by
\begin{equation}
\label{eq:definition-prolongation-Tanaka-l+1}
\opp{Pr}_{\ell+1}(M,\cD,F_0)(\cV_o)=\big\{\cH_\ell(\cV_o)\mid \cH_\ell=\{ \cH^i_\ell \}_{i \leq \ell}\;\text{on $\cV_0$}\;\text{such that}\;
c^\pm_{\mathcal{H}_\ell}\in \mathcal{N}^\pm_{\ell+1}|_{\cV_o}
\big\}\;,
\end{equation}
equivalently the collection of the associated $(\ell+1)$-frames. Since the Chevalley--Eilenberg differential
$0\to\g\stackrel{\delta}\rightarrow\fm^*\otimes\g$ is injective on $\fg_{\geq 0}$, hence on $\fg_{\ell}$, also the differential \eqref{eq:tensorproductdifferential} is injective and \eqref{eq:definition-prolongation-Tanaka-l+1} is a
principal bundle $\pi_{\ell+1}:\opp{Pr}_{\ell+1}(M,\cD,F_0)\to F_\ell$ over $F_\ell$. It has Abelian structure group $G_{\ell+1}=\exp(\fg_{\ell+1})$ consisting of all elements of $C^{\ell+1,1}(\fm,\fg)$ in the kernel of the Spencer operator $\delta$, i.e., of all elements of the prolongation $\fg_{\ell+1}$.
%
\subsection{Canonical parallelism and comparison with Zelenko's approach}
\label{subsec:canpar}

 \begin{theorem}
\label{thm:absolute-parallelism}
Let $(M,\cD,q)$ be a filtered structure of finite type, i.e.,
the Tanaka prolongation stabilizes:
$\g=\g_{-\mu}\oplus\dots\oplus\g_0\oplus\dots\oplus\g_d$
(with $\g_d\neq0$, but $\g_{d+1}=0$).
Then there exists a fiber bundle $\pi:P\to M$ of $dim P=\dim\g$
and an absolute parallelism $\Phi\in\Omega^1_{\bar 0}(P,\g)$,
which is natural in the sense that any equivalence transformation
$f:M\to M'$ lifts to a unique map $F:P\to P'$
preserving the parallelisms $\Phi$ and $\Phi'$.
 \end{theorem}

 \begin{proof}
Let $P=F_d$ and consider the structure $\Phi_d$
consisting of $\varphi_{d+1}(v_i)$, which are supervector fields
 for $i\ge-1$ and truncated supervector fields for $i<0$.
Let us prolong further, but first note that  the map
$F_k\to F_{k-1}$ is a principal bundle with trivial fiber $\g_k=0$ for any $k\geq d+1$, hence a diffeomorphism. Take $j=d+\mu-1$.
Then $\varphi_{j+1}|_{\fg_i}:\cO_{F_{j}}\otimes\g_i\to\mathcal{T}_{j}^i/\mathcal{T}_{j}^{i+j+2}
=\mathcal{T}_{j}^i$ because
$\mathcal{T}_{j}^{i+j+2}\cong\mathcal{T}_d^{i+j+2}=0$.
Thus we get the required non-truncated supervector fields.

The frames $\varphi_{j+1}:\cO_{F_{j}}\otimes\fg\to\mathcal T_j$ give
an absolute parallelism $\Phi$ on $P:=F_{j}\cong F_d$. Indeed a basis on $\g$, respecting the parity and $\mathbb Z$-grading, gives a basis of supervector fields on $P$.
 \end{proof}

 \begin{rk}
\label{rem:equivariancy}
In general, the fiber bundle $\pi:P\to M$ is not principal,
so the parallelism $\Phi$ lacks equivariancy and it is not
a Cartan superconnection, but it suffices for dimensional bounds.
If the normalizations can be chosen invariantly w.r.t. the Lie supergroup $G_0\ltimes\exp(\g_1\oplus\dots\oplus\g_d)$, then we expect that $\pi : P \to M$ is a principal bundle and a Cartan superconnection exists.  (We have verified that this is true in the $d=0$ case.)
In this case, the step of our construction involving vertical structure
functions and normalizations would not be required:
one may simply take the fundamental vector fields of the principal action.
 \end{rk}
An essential difference with the argument in \cite{Z} is that
this reference uses the reduced differential
$\partial:\fm^*\otimes\g\to\mathfrak{A}:=(\g_{-1}^*\wedge\fm^*)\otimes\g$
while we are using the Spencer differential\footnote{In a private discussion
with BK, Igor Zelenko confirmed that he was also considering this approach.}
$\delta:\fm^*\otimes\g\to\Lambda^2\fm^*\otimes\g$.
They are related through the restriction map
$p:\Lambda^2\fm^*\otimes\g\to\mathfrak{A}$,
$\partial=p\circ\delta$, or equivalently
$\partial\alpha=\delta\alpha|_{\g_{-1}\wedge\fm}$.
Our normalizations agree as follows.

 \begin{lem}
The map $p$ is injective when restricted to the kernel of $\delta$ on $\Lambda^2\fm^*\otimes\g$. In particular, it is injective when restricted to
$\delta(\fm^*\otimes\g)$ and $\opp{Ker}(\partial)=\opp{Ker}(\delta)$.
\end{lem}
\begin{proof}
We need to show that if $\omega\in\Lambda^2\fm^*\otimes\g$ is $\delta$-closed and
$p(\omega)=0$, i.e., $\omega(\g_{-1},\cdot)=0$, then $\omega=0$.
Let $u,v\in\g_{-1}$, $w\in \fm$. Then
(up to signs, which are not essential here)
 \begin{align*}
0=\delta(\omega)(u,v,w) &=
[u,\omega(v,w)]-[v,\omega(u,w)]+[w,\omega(u,v)]\\
&\;\;\;\;-\omega([u,v],w)+\omega([u,w],v)-\omega([v,w],u)\\
&=-\omega([u,v],w)\;.
 \end{align*}
Thus $\omega(\Pi_2,\cdot)=0$ for $\Pi_2=\g_{-2}\oplus\g_{-1}$.
Applying the above formula for $v\in\Pi_2$ we now obtain
$\omega(\Pi_3,\cdot)=0$ for $\Pi_3=\g_{-3}\oplus\g_{-2}\oplus\g_{-1}$ and iterating yields our first claim.
The rest is clear.
\end{proof}
 Since $p$ is also the projection to $\mathfrak{A}$ along the $\mathbb Z$-graded complement $\mathfrak{B}=\oplus_{i,j\leq-2}(\g_i^*\wedge\g_j^*)\otimes\g$  in $\Lambda^2\fm^*\otimes\g$, the following result follows straightforwardly.
 \begin{cor}

If $Z$ is a complement to
$\opp{Im}(\partial)$ in $\mathfrak{A}$ then $N=Z\oplus\mathfrak{B}$ is a complement to
$\opp{Im}(\delta)$ in $\Lambda^2\fm^*\otimes\g$.

 \end{cor}

Consequently, any complement $Z$ is obtained as $Z=p(N)$
for some complement $N$.
However, it is not  true that any $N$ is of the indicated type $N=Z\oplus\mathfrak{B}$,
yet for our purposes this distinction plays no role.
In concrete cases, when the prolongation $\g_{\ge0}$ acts
completely reducibly this may turn important in finding an invariant complement.

\section{Dimension bounds for symmetry}\label{S4}
\subsection{The automorphism supergroup and dimension bounds}
\subsubsection{Basic definitions}\label{S41sub}
Symmetries of filtered supergeometries are defined as follows.

 \begin{definition}
\hfill
\begin{itemize}
\item[$(i)$] An {\em automorphism} $\varphi\in\Aut(M,\cD,F)_{\bar0}$ of a filtered structure
is an element $\varphi\in\Aut(M)_{\bar0}$ such that:
\begin{itemize}
	\item it preserves the distribution:
 $
\varphi_*(\cD)\subset(\varphi_o)_*^{-1}\cD
 $, so it induces an isomorphism of $\opp{Pr}_0(M,\cD)$ (which, by abuse of notation, we simply denote by $\varphi_*$);
	\item in the case of a
first order reduction $F=F_0\subset \opp{Pr}_0(M,\cD)$, we also require
that $\varphi_*(\mathcal{F}_0)\subset(\varphi_o)_*^{-1}\mathcal{F}_0$;
	\item then it induces an isomorphism of $F_0^{(i)}$ and,
if there are higher order reductions, we also require
that it preserves them: $
\varphi_*(\mathcal{F}_i)\subset(\varphi_o)_*^{-1}\mathcal{F}_i
 $.
\end{itemize}

\item[$(ii)$] An {\em infinitesimal automorphism} $X\in\mathfrak{inf}(M,\cD,F)$
on $M$ is a supervector field $X\in\mathfrak X(M)$ such that
 $
\mathcal L_{X}(\cD)\subset\cD
 $,
and it successively preserves the structure reductions, namely:
 $$
\mathcal L_{X}\big(\mathcal{F}_i(\cV_o)\big)\subset\mathcal{F}_i(\cV_o)\cdot
\big(\g_i\otimes\cO_{F_{{i-1}}}(\cV)\big)\subset\widetilde{\mathcal{F}r}_i(\cV_o)\;,
 $$ for any open subset $\cV_o\subset (F_{i-1})_o$. (See \eqref{eq:locally-free-frames} for the definition of $\widetilde{\mathcal{F}r}_i$.)
\item[(iii)] An infinitesimal automorphism $X\in\mathfrak{inf}(M,\cD,F)$ is {\it complete} if it is so as a supervector field, i.e.,
its local flow (in the sense of \cite{GW, MSV}) has maximal flow domain $\R^{1|1}\times M$. We denote the collection of all complete
infinitesimal automorphisms by $\mathfrak{aut}(M,\cD,F)$.
\end{itemize}
\end{definition}
We recall that a supervector field on $M=(M_o,\cA_M)$ is complete if and only if the associated vector field on $M_o$ is complete \cite{GW, MSV}. The automorphism supergroup
is in this way defined as the super Harish-Chandra pair
$\Aut(M,\cD,F):=\bigl(\Aut(M,\cD,F)_{\bar0},\mathfrak{aut}(M,\cD,F)\bigr)$.

Assume the filtered structure is of finite type.
By the naturality of the constructions, any automorphism of
the filtered structure on $M$ lifts to a symmetry
of the bundle $\pi:P\to M$ constructed in Theorem \ref{thm:absolute-parallelism} and it preserves
the absolute parallelism $\Phi$ on it. Likewise, any
infinitesimal symmetry of the filtered structure on $M$
lifts to an infinitesimal symmetry of the bundle $\pi:P\to M$
preserving the absolute parallelism $\Phi$. In particular, the dimension of
the symmetry superalgebra $\mathfrak{s}$ is bounded by the dimension of the infinitesimal symmetries of $\Phi$.
To prove $\dim\mathfrak{s}\leq\dim P$ (in the strong sense) and complete the proof of Theorems \ref{T1} and \ref{T2}, it is sufficient to establish a bound on the dimension of the symmetries of the absolute parallelism.

\subsubsection{Dimension of the symmetry superalgebra}\label{S41}

By fixing a basis of $\fg$, the absolute parallelism $\Phi$ corresponds to
a coframe field $\{\omega^\beta\}$ on $P$, where the index $\b$ run over both the even and odd indices.
Let $\{e_\alpha\}$ be the dual frame, i.e.,
$\langle e_\a,\omega^\b\rangle=(-1)^{|\a||\b|}\omega^\b(e_\a)
=\delta^\beta_\alpha$.
Here $\a$ runs through both even and odd indices as well. The following result was originally established in \cite[Lemma 13]{O}. We give here a simplified proof that does not use the concept of flow for supermanifolds.
 \begin{lem}\label{StZ}
Let $\{e_\a\}$ be a frame on a supermanifold $P=(P_o,\cA_P)$ with connected reduced manifold.  Fix a point $x \in P_0$.  Then any infinitesimal automorphism
$\upsilon\in\mathfrak{X}(P)$ of the frame is determined by its value at $x$.
 \end{lem}

 \begin{proof}
The statement is equivalent to the claim that $\opp{ev}_x(\upsilon)=0$
implies $\upsilon=0$.

Consider the ideal $\mathcal{J}=(\cO_P)_{\bar1}^2\oplus (\cO_P)_{\bar1}$
of $\cA_P$ generated by nilpotents. Then for any $k>0$ the map
$$\mathcal{J}^k/\mathcal{J}^{k+1}\to\cT^* P\otimes\mathcal{J}^{k-1}/\mathcal{J}^k\;,\quad f\,\opp{mod}\mathcal{J}^{k+1}\mapsto\sum_{\text{odd}\;\alpha} \omega^\a\otimes e_\a(f)\,\opp{mod}\mathcal{J}^k\quad(f\in\mathcal{J}^k)$$
is injective. In other words, if $f\not\in\mathcal{J}^{k+1}$, then there exists an odd $\a$ such that $e_\a(f)\not\in\mathcal{J}^k$.
Now if $\upsilon$ is in $\mathcal{J}^k\otimes \cT P$ but not in $\mathcal{J}^{k+1}\otimes \cT P$ for some $k> 0$, then the Lie equation
$L_\upsilon e_\a=0$ cannot hold for all odd $\alpha$ because there exists one for which it is
wrong already modulo $\mathcal{J}^k\otimes \cT P$. This tells us that the evaluation map
$\opp{ev}:\mathfrak{X}(P)\to \Gamma(TP|_{P_o})$ is injective on the symmetries.

Set $\widetilde\upsilon=\opp{ev}\upsilon\in\Gamma(TP|_{P_o})$. Taking the Lie equation $L_ \upsilon e_\a=0$ modulo $\mathcal{J}\otimes\cT P$ for an even $\a$ we get a pair of reduced Lie equations
\begin{equation}
\label{eq:Lieequations-reduced}
\begin{aligned}
\opp{ev}\big(L_{\upsilon_{\bar 0}} e_\a \big)&=0\;,\\
\opp{ev}\big(L_{e_\a} \upsilon_{\bar 1})&=0\;,
\end{aligned}
\end{equation}
which depend only on $\widetilde\upsilon=\widetilde\upsilon_{\bar 0}+\widetilde\upsilon_{\bar 1}$. More precisely
$\widetilde\upsilon_{\bar{0}}$ is a classical vector field on $P_o$
and the first reduced Lie equation is $L_{\widetilde\upsilon_{\bar 0}} \widetilde e_\a=0$, where $\{\widetilde e_\a=\opp{ev}e_{\a}\}_{\a\;\text{even}}$ is the induced absolute  parallelism on $P_o$. The infinitesimal version of Lemma 1 from the proof of
\cite[Thm. 3.2]{Ko} applies: the set of critical points of
$\widetilde\upsilon_{\bar 0}$ is simultaneously closed and open, so $\widetilde\upsilon_{\bar 0}$ is determined by
its value at $x\in P_o$. On the other hand, $\widetilde\upsilon_{\bar{1}}$ is a section of the bundle
$(TP|_{P_o})_{\bar1}$ with the natural flat connection defined by
$$\nabla_{f^\a\widetilde e_\a}\widetilde\upsilon_{\bar{1}}:=f^\a\opp{ev}\big(L_{e_\alpha}\upsilon_{\bar{1}}\big)\;,$$
where $f^\a\in C^{\infty}_{P_o}$ for any even $\a$. Hence the value
of a parallel section at one point determines the section everywhere. In summary, the map
$\upsilon\mapsto\opp{ev}_x\upsilon$ is injective.
 \end{proof}

Let us now observe how the dimension of the solution space is constrained.  The structure equations 
 $$
[e_\alpha,e_\beta]=c_{\alpha\beta}^\gamma e_\gamma\quad\Leftrightarrow\quad
d\omega^\gamma=
-\frac12(-1)^{|\a||\b|}
 \big(\omega^\alpha\wedge\omega^\beta\big)c_{\alpha\beta}^\gamma
 $$
involve structure superfunctions $c_{\alpha\beta}^\gamma\in\cO_P$
of parity $|\alpha|+|\beta|+|\gamma|$.
An infinitesimal symmetry is a supervector field $\upsilon=a^\delta e_\delta\in\mathfrak{X}(P)$ such that $L_\upsilon e_\alpha=0$ for all $\alpha$. Equivalently
it must preserve the coframe, so we get
 $$
0=L_\upsilon\omega^\gamma=
d\imath_\upsilon\omega^\gamma+\imath_\upsilon d\omega^\gamma=
da^\gamma -\frac{1}{2}a^\delta(-1)^{|\a||\b|}
\imath_{e_\delta}\big(\omega^\alpha\wedge\omega^\beta\big)c_{\alpha\beta}^\gamma\;,
 $$
for all $\gamma$, which we rewrite as
 \begin{equation}
\label{Frob}
da^\gamma=(-1)^{|\b||\upsilon|}(\omega^\b)a^\a c_{\a\b}^\gamma
.
 \end{equation}
This is a complete PDE on the superfunctions $a^\gamma$'s
and the dimension bound $\dim P$ is achieved if and only if the compatibility conditions $d^2a^\gamma=0$
holds. We can see this explicitly in local coordinates $x^\a$ on $P$.
Let $e_\a=\varkappa_\a^\delta\p_\delta$, where $\p_\b=\frac{\p}{\p x^\b}$,
with the dual coframe $\omega^\b=(dx^\delta)\widetilde\varkappa_\delta^\beta$, where
$\varkappa_\alpha^\delta\widetilde\varkappa^\b_\delta=\delta^\beta_\alpha$.
Then formula \eqref{Frob} becomes
\begin{equation*}
\begin{aligned}
\p_\epsilon a^\gamma&= (-1)^{|\upsilon||\b|}\widetilde\varkappa^\b_\epsilon a^{\a}c^\gamma_{\a\b}\\
&=(-1)^{|\upsilon||\epsilon|+|\alpha||\beta|+|\alpha||\epsilon|}a^{\a}\widetilde\varkappa^\b_\epsilon c^\gamma_{\a\b}\\
&=a^\alpha\sigma^\gamma_{\epsilon\alpha}\;,
\end{aligned}
\end{equation*}
where we denoted $\sigma^\gamma_{\epsilon\a}=(-1)^{|\upsilon||\epsilon|+|\alpha||\beta|+|\alpha||\epsilon|}\widetilde\varkappa^\b_\epsilon c^\gamma_{\a\b}$. The compatibility conditions given by the vanishing of the supercommutator of $\p_{\epsilon'}$ and $\p_{\epsilon''}$ on $a^\gamma$ are
\begin{multline}
\label{eq:compatibility}
a^\a\Bigl(
(-1)^{|\a||\epsilon'|+|\upsilon||\epsilon'|}\p_{\epsilon'}\sigma^\gamma_{\epsilon''\a}-
(-1)^{|\a||\epsilon''|+|\upsilon||\epsilon''|+|\epsilon'||\epsilon''|}\p_{\epsilon''}\sigma^\gamma_{\epsilon'\a}\\
+\sigma^\nu_{\epsilon'\a}\sigma^\gamma_{\epsilon''\nu}
-(-1)^{|\epsilon'||\epsilon''|}\sigma^\nu_{\epsilon''\a}\sigma^\gamma_{\epsilon'\nu}
\Bigr)=0.
\end{multline}
If the parenthetical expression vanishes for all indices, then any
initial value for the $a^\a$'s produces a unique solution
$\upsilon$, and the dimension of the solution space
is $\dim P$.
If not, then we have to differentiate the L.H.S.\ of \eqref{eq:compatibility}, substitute
\eqref{Frob} and study the $0$-th order linear equations on the $a^\a$'s.
When the system stabilizes, the corank of the resulting matrix
(i.e., the matrix size minus the size of the largest invertible minor),
gives the dimension of the solution space.

\subsubsection{Dimension of the automorphism supergroup}\label{S42}

By the results in \cite[\S 4.2]{O}, the {\it group} of automorphisms $\Aut(\Phi)_{\bar 0}$ of an absolute parallelism $\Phi$ on a supermanifold $P=(P_o,\cA_P)$ with connected reduced manifold $P_o$ is a (finite-dimensional) Lie group. At first,
if we denote by $\Phi_{\bar 0}$ the induced absolute parallelism on $P_o$ (in the notation of \S\ref{S41}, this is $\Phi_{\bar0}=\{\opp{ev}e_{\a}\}_{\a\;\text{even}}$), then the classical argument of \cite{Ko} proves that $\Aut(\Phi_{\bar 0})$ is a Lie group:
$\opp{Aut}(\Phi_{\bar 0})$ is mapped to $P_o$ as the orbit through $x\in P_o$ and the stabilizer of a classical absolute parallelism
at any point is trivial. Then, the forgetful map $\Aut(\Phi)_{\bar 0}\to \Aut(\Phi_{\bar 0})$ is injective with closed image,
cf. \cite[Lemmas 10 and 11]{O}.

It follows from this and Lemma \ref{StZ} that the automorphism supergroup $(\Aut(\Phi)_{\bar 0}, \mathfrak{aut}(\Phi))$ is a finite-dimensional super Harish-Chandra pair, in other words a Lie supergroup. Here $\mathfrak{aut}(\Phi)$ is the Lie superalgebra of complete infinitesimal automorphisms.
We remark that for a pair of complete supervector fields
neither their linear combination nor their commutator are complete. (This holds also in the classical case.)
However the set of
complete supervector fields that are infinitesimal automorphisms of an absolute parallelism
form a supervector space, and moreover a Lie superalgebra.
This is because the sum and Lie bracket of two infinitesimal automorphisms of $\Phi_{\bar 0}$ is still complete by the classical result of \cite{Ko}
and a supervector field $\upsilon\in\mathfrak{X}(P)$
is complete if and only if the associated vector field $\opp{ev}(\upsilon_{\bar 0})\in\mathfrak{X}(P_o)$
on $P_o$ is so \cite{GW, MSV}. This shows that the ``representability issue'' of
\cite[Thm 15]{O} 
can be amended: completeness of the infinitesimal automorphisms (i.e., the requirement that $\mathfrak{inf}(\Phi)=\mathfrak{aut}(\Phi)$) is {\it not} an obstruction for the representability of the automorphism supergroup.



By the construction of the absolute parallelism $\Phi$ on $P$,
it is not difficult to see that the automorphism supergroup
$\opp{Aut}(M,\cD,q)=(\opp{Aut}(M,\cD,q)_{\bar 0},\mathfrak{aut}(M,\cD,q))$
of a non-holonomic geometric structure $(\cD,q)$ on $M$ or, more generally,
of a filtered structure, is a closed subsupergroup of
$\Aut(\Phi)=(\Aut(\Phi)_{\bar 0}, \mathfrak{aut}(\Phi))$.
Therefore it is a Lie supergroup, whose dimension is bounded by $\dim P=\dim\fg$,
and this finishes the proof of Theorem \ref{T2}.

A more careful analysis shows that
$\opp{Aut}(M,\cD,q)$ is a discrete quotient of $\Aut(\Phi)$.
Indeed, by Theorem \ref{thm:absolute-parallelism},
automorphisms in the base lift to the frame bundle.
On the other hand, for any $k$, automorphisms of the frame bundle $F_k$
preserve the fundamental fields from $\g_k$ and therefore they project
to automorphisms of a cover of the frame bundle $F_{k-1}$, namely to
the quotient of $F_k$ by the connected component of unity in
the structure group $G_k$.
We apply this for $k$ descending from $d$ to $0$ and conclude the claim.
If the structure groups $G_k$ are connected (this is usually the case for
$k>0$, if no higher order reductions are imposed, because the fibers
are affine), then we have the equality $\opp{Aut}(M,\cD,q)=\Aut(\Phi)$.

 \begin{rk}
In the case $M_o$ has finitely many connected components, say
$n\in\N$, one easily modifies the above arguments to get
the following dimension bound:
 $$
\dim\opp{Aut}(M,\cD,q)\leq\dim\fs\leq n\cdot\dim\g.
 $$
Indeed, enumerate the components $1,\dots,n$ of $M_o$ and let $\sigma\in S_n$
encodes a (possibly trivial) permutation of components.  
Then all automorphisms
are parametrized as follows: no more than
$\dim\g$ parameters for maps of the 1st component to that number
$\sigma(1)$, no more than $\dim\g$ parameters for maps of the 2nd
component to that number $\sigma(2)$, etc.
 \end{rk}

\subsection{Structure of the symmetry superalgebra and maximally symmetric spaces}\label{S43}

We now discuss the following statement, which is not primary for
the purposes of this paper.
We therefore will only sketch the proof,
referring the reader for details to the original papers.

Let $\g=\g_{-\mu}\oplus\dots\oplus\g_0\oplus\dots\oplus\g_d$
be the Tanaka algebra associated to the filtered structure $(M,\cD,q)$.
Its natural (decreasing) filtration is given by subspaces
$\g^i=\g_i\oplus\dots\oplus\g_d$, $i\ge-\mu$, which
for $\mu=1$ is the so-called filtration by stabilizers and for
$\mu>1$ is the weighted (or Weisfeiler) filtration.

 \begin{theorem}
\label{thm:filt-def}
The symmetry algebra $\fs$ of $(M,\cD,q)$ embeds into $\g$ as a filtered
subspace $\iota:\fs\to\g$ in such a way that the corresponding graded
map $\opp{gr}(\iota):\opp{gr}(\fs)\to\g$ is an injection of Lie algebras.
 \end{theorem}

 \begin{proof}
Fix a point $x\in M_o$
and consider the weighted filtration of the stalk
$\cT M_x$ that refines the filtration by the maximal ideal in
$(\cA_M)_x$ using the weighted filtration induced by the distribution $\cD$.
This generalizes to the super-setting the second filtration on
the symmetry Lie algebra sheaf from \cite{K1} and gives the required
embedding $\iota:\fs\to\g$.

Alternatively, consider the Lie equation governing infinitesimal
symmetries of $(M,\cD,q)$ as a subsupermanifold embedded into
the space of weighted super-jets. This provides the solution space $\fs$
of the equation with the desired filtration, see \cite{K2}.
 \end{proof}

This theorem serves as a base to obtain submaximally symmetric models
via filtered deformations of large graded subalgebras of $\g$,
see \cite{KT} for applications in the classical case and \cite{KST}
for examples in the super case.

\begin{rem}
Spaces $(M,\cD,q)$ with
$\dim\opp{Aut}(M,\cD,q)=\dim\mathfrak{aut}(M,\cD,q)=\dim\g$
as well as spaces with $\dim\mathfrak{inf}(M,\cD,q)=\dim\g$
are non-unique but there are always two cases when
the maximal symmetry dimension is attained.

The so-called \emph{flat model} is the homogeneous space $G/H$,
with $G$ a Lie supergroup with $\opp{Lie}(G)=\g$ and $H$ its closed subsupergroup with $\opp{Lie}(H)=\g^0$.
(One can impose simply-connectedness of $G/H$
 though this is not necessary.)
The filtration $\g^i$ on $\g$ induces a left-invariant filtration
$\cF^i$ on $G$ and therefore the distribution $\cD=\cF^{-1}/\cF^0$ on $G/H$
with the desired derived flag. In addition, all the reductions are invariant
w.r.t. $G$, hence the induced filtered structure is invariant.
If $q$ encodes the filtered structure, then
$\mathfrak{inf}(G/H,\cD,q)=\g$ and
$\Aut(G/H,\cD,q)$ coincides with the supergroup $G$ or
its discrete factor.

The so-called \emph{standard model}
is obtained through a left-invariant structure $(\cD,q)$,
or more generally a filtered structure,
on the nilpotent Lie supergroup $M=\exp(\fm)$. This usually
does not have the maximal automorphism group,
but it is locally isomorphic to the flat model and hence
$\mathfrak{inf}(M,\cD,q)=\g$.
Complete description of other models with maximal symmetry dimension
can be obtained via the technique of filtered deformations of $\fs=\g$.
\end{rem}

\section{Examples and applications}\label{S5}

Here we demonstrate how our dimensional bounds work.
We emphasise that all our main results are applicable to both real smooth and complex analytic cases,
so some examples will be stated over $\R$ and some over $\bbC$.

\subsection{Holonomic structures}\label{S51}
Let us first illustrate the symmetry bounds with some particular
geometric structures on a supermanifold $M=(M_o,\cA_M)$ of $\dim M=(m|n)$
in the holonomic case $\cD=\cT M$ (thus $\fm=\g_{-1}$ in this subsection).

\subsubsection{Affine superconnections}
An affine superconnection
is an even map $\nabla:\mathfrak{X}(M)\otimes_\R\mathfrak{X}(M)\to\mathfrak{X}(M)$,
$(X,Y)\mapsto\nabla_XY$, which is $\cO_M$-linear in $X$ and satisfies $\nabla_X(fY)=(-1)^{|f||X|}f\nabla_XY+X(f)Y$. In local coordinates it is given via the Christoffel symbols
$\nabla_{\p_\a}\p_\b=\Gamma_{\a\b}^\gamma\p_\gamma$, where
$|\Gamma_{\a\b}^\gamma|=|\a|+|\b|+|\gamma|$.
From the viewpoint of $G$-structures, an affine superconnection is
a reduction of the second order, i.e., $F_0=F r_M\cong F_1$
or equivalently $\g_0=\mathfrak{gl}(V)$ and $\g_1=0$.
A choice of $\cH$ as in \S\ref{subsec:Gstru} that is equivariant under $GL(V)$ is equivalent to the choice of a connection $1$-form $\omega\in\Omega^1(F r_M,\mathfrak{gl}(V))_{\bar 0}$, $\cH=\opp{Ker}(\omega)$, which in turn is equivalent to $\nabla$.

Thus for the symmetry algebra of $\nabla$ we have:
 $$
\dim\mathfrak{s}\leq\dim\g_{-1}+\dim\g_0=(m+n^2+m^2\,|\,n+2mn).
 $$
The Lie algebra of the even part $\Aut(M,\nabla)_{\bar0}$ of the Lie supergroup of affine transformations $\Aut(M,\nabla)$ consists of supervector fields that are complete. Therefore, its dimension might be smaller than $\dim\mathfrak{s}$ in general.

\subsubsection{Super-Riemannian structures}
A super-Riemannian structure on $M$ is given by a nondegenerate even supersymmetric $\cA_M$-bilinear form $q$ on $\cT M$.  (In the real case, the even part of $q$ can have any signature.) It is a $G_0$-structure with
$G_0=\opp{OSp}(m|n)$, $n\in2\Z$. For $\g_0=\opp{Lie}(G_0)$
it is known that $\g_1=\g_0^{(1)}=0$.  The argument
straightforwardly generalizes the classical one \cite{St}, see \cite{O},
which corresponds to the analog of the Levi-Civita connection \cite{G}.
Thus $(M,q)$ determines an affine structure.

The Lie superalgebra of
Killing supervector fields satisfies
 $$
\dim\mathfrak{s}\leq\dim\g_{-1}+\dim\g_0=
\Bigl(\tbinom{m+1}2+\tbinom{n+1}2\,|\,n+mn\Bigr).
 $$
The above remark about completeness for affine structures
applies to super-Riemannian structures
and the isometry supergroup as well.

\subsubsection{Almost super-symplectic structures}
An almost super-symplectic structure on $M$ is given by a nondegenerate
even super-skewsymmetric bilinear form $q$ on $TM$. It is a
$G_0$-structure with $G_0=\opp{SpO}(m|n)$, $m\in2\Z$.
In this case $\g_0=\opp{Lie}(G_0)=\mathfrak{spo}(m|n)\cong
\mathfrak{osp}(n|m)$ but as representations on $V=\R^{m|n}\cong\Pi\R^{n|m}$
these Lie superalgebras are quite different, cf.\ Remark \ref{Pi-rep}.
In particular $\g\subset\mathfrak{gl}(V)$ is of infinite type
unless $V$ is purely odd.

 \begin{lem}
We have: $\g_i=\g_0^{(i)}=S^{i+2}V^*$ (in the super-sense),
which is nonzero $\forall i\ge0$ if $m>0$.
 \end{lem}

The proof of this claim mimics the proof of the classical computation for
almost symplectic structure \cite{St} and will be omitted. We note that
$(S^iV^*)_{\bar0}=\oplus_{j=0}^{\max(i,n)}
S^{i-j}V^*_{\bar0}\otimes\Lambda^jV^*_{\bar1}$, where the symmetric and exterior powers in the R.H.S.\ are meant in the classical sense.

Also $\cT M^*\cong \cT M$ via $q\in\Omega^2(M)$ and, provided $dq=0$, the local symmetries are all of the form $q^{-1}dH$ for $H\in\cO_M$. Thus in this case we may have
$\dim\mathfrak{s}=\infty$ and $\Aut(M,q)$
is not necessarily a Lie supergroup. This may happen even when $dq\neq0$.

In the case $M$ is purely odd $(m=0)$, we have:
 $$
\dim\fs\leq\sum_{i=-1}^{n-2}\dim\g_i=\sum_{k=1}^n\dim\Lambda^k\Pi(V^*)=2^n-1.
 $$

\subsubsection{Periplectic-related structures}

Let $\mathfrak{P}$ be non-degenerate bilinear form on $V = \R^{n|n}$ that is odd, i.e.\ $\mathfrak{P}(V_{\bar{0}},V_{\bar{0}}) = \mathfrak{P}(V_{\bar{1}},V_{\bar{1}}) = 0$.  When $\mathfrak{P}$ is supersymmetric, i.e., $\mathfrak{P}(x,y) = (-1)^{|x||y|} \mathfrak{P}(y,x)$ for all pure parity $x,y$, we define the {\sl periplectic} Lie superalgebra by
 \begin{align} \label{E:peri}
 \fpe(n) := \{ X \in \fgl(n|n) : X^{{\rm st}}\mathfrak{P} + (-1)^{|X|}\mathfrak{P}X = 0 \},
 \end{align}
with $\fgl(n|n)$-inherited $\Z_2$-grading.
Explicitly, taking
$\mathfrak{P} = \begin{psm} 0 & \id_n\\
 \id_n & 0 \end{psm}$ yields
  \begin{align} \label{E:peri-matrix}
 \fpe(n) &= \left\{ \begin{pmatrix}
 A & B\\ C & -A^\top
 \end{pmatrix} : A,B,C \in \fgl(n) : B = B^\top, \, C = -C^\top \right\}.
 \end{align}
Above $st$ and $\top$ are the supertranspose and the usual transpose, respectively. We also define some related Lie superalgebras:
 \begin{itemize}
 \item {\sl special periplectic} $\fspe(n) := \fpe(n) \cap \fsl(n|n)$.  This is simple for $n \geq 3$.
 \item {\sl conformal (special) periplectic} $\fcpe(n) := \bbC\id_{2n} \oplus \fpe(n)$ and $\fcspe(n) := \bbC\id_{2n} \oplus \fspe(n)$.
 \item $\fspe_{a,b}(n) := \langle a\tau + bz \rangle \ltimes \fspe(n)$, where $a,b \in \bbC$, $\tau = \operatorname{diag}(\id_n,-\id_n)$, and $z = \id_{2n}$.
 \end{itemize}


%

Note that $\fspe_{a,b}(n)$ depends only on $[a:b]\in\bbC P^1$, and
$\fspe_{1,0}(n)=\fpe(n)$, $\fspe_{0,1}(n)=\fcspe(n)$; we treat separately
$\fspe_{0,0}(n)=\fspe(n)$. We have:
 $$
\dim\fspe_{a,b}(n)=(n^2|n^2),\ \dim\fspe(n)=(n^2-1|n^2),\
\dim\fcpe(n)=(n^2+1|n^2).
 $$

 \begin{theorem}
Consider an irreducible $G_0$-structure $F_0$ on a supermanifold $M$ of dimension $(n|n)$, i.e., $G_0=\exp\g_0$ acts irreducibly on $\fg_{-1}$,
where $\g_0$ is one of the Lie superalgebras above.
We get the following symmetry dimension bounds for $\fs=\mathfrak{inf}(M,F_0)$:
 \begin{enumerate}
 \item If $\fg_0 = \fpe(n),\ \fspe(n),\ \fcspe(n)$, or $\fspe_{a,b}(n)$, where $a,b \in \bbC^\times$ with
     $b\neq na$, then \linebreak
     $\dim(\fs) \leq (n|n) + \dim(\fg_0)$.
 \item If $\fg_0 = \fcpe(n)$, then $\dim(\fs) \leq \dim(\fpe(n+1)) = ((n+1)^2|(n+1)^2)$.
 \item If $\fg_0 = \fspe_{1,n}(n)$, then $\dim(\fs) \leq \dim(\fspe(n+1)) = (n^2+2n|(n+1)^2)$.
 \end{enumerate}
 \end{theorem}

 \begin{proof} These results follow from applying our Theorem \ref{T1} 
to the prolongation results due to Poletaeva \cite[Thm.1.2]{P}: namely, she proved that $\pr(\fg_{-1},\fg_0) = \fg_{-1} \op \fg_0$ for (1), while $\pr(\fg_{-1},\fg_0) \cong \fpe(n+1)$ for (2), and $\pr(\fg_{-1},\fg_0) \cong \fspe(n+1)$ for (3).  We remark that for (2) and (3), the prolongation height is 2, which differs from its depth being 1.  See \cite[Lemma 1.1]{P} for details on this $\Z$-grading.
 \end{proof}

 Consider now the case when $\mathfrak{P}$ is skew-supersymmetric, i.e.\ $\mathfrak{P}(x,y) = - (-1)^{|x||y|} \mathfrak{P}(y,x)$ for all pure parity $x,y$.  The same formula as in \eqref{E:peri} defines the {\sl skew-periplectic} Lie superalgebra $\fpe^{sk}(n)$ and taking $\mathfrak{P} = \begin{psm}
 0 & \id_n\\
 -\id_n & 0
 \end{psm}$ gives
 \begin{align} \label{E:sk-peri}
 \fpe^{sk}(n) &= \left\{ \begin{pmatrix}
 A & C\\ B & -A^\top
 \end{pmatrix} : A,B,C \in \fgl(n) : B = B^\top, \, C = -C^\top \right\}.
 \end{align}
We may define analogous related Lie superalgebras as above. Clearly the parity change functor $V \to \Pi(V)$ on $V=\mathbb R^{n|n}$ induces an isomorphism $\fpe(n) \cong \fpe^{sk}(n)$ as Lie superalgebras, but the differing representation on $V$ is crucial as the following result shows:

 \begin{prop} Let $\fg_{-1} = V = \R^{n|n}$.  If $\fg_0 \supset \fspe^{sk}(n)$, then $\fg = \pr(\fg_{-1},\fg_0)$ has  infinite odd part.
 \end{prop}

 \begin{proof}
Focusing on the ``$B$-part'' of \eqref{E:sk-peri},
we see that $\fg_0$ contains a rank 1 odd element $x \otimes \omega$, where $x \in V_{\bar{1}}$ and $\omega \in V_{\bar{0}}^*$.  Considering the odd elements $\phi_k = x \otimes \omega^{k+1}$ for all $k > 0$, we inductively observe that $\phi_k \in \fg_k$ for all $k > 0$.
 \end{proof}

\subsubsection{Projective superstructures}\label{Sproj}
Classical projective structures are defined as equivalence classes of  affine connections for which geodesics differ by
a reparametrization. It is well-known that every class contains a torsion-free connection.
We here omit the discussion of what a supergeodesic is since this is not uniform in
the literature (\cite{G,LRT}) and simply follow \cite{LRT} in adapting the classical interpretation of projective
equivalence for the torsion-free connections: two torsion-free affine superconnections $\nabla$ and $\nabla'$ are equivalent if and only if
$\nabla-\nabla'=\opp{Id}\circ\omega\in\Gamma(S^2\cT ^*M\otimes \cT M)$ for an even 1-form
$\omega\in\Omega^1(M)$. (The symmetric power is meant in the super-sense.)
This is a higher order reduction of the frame bundle.
Namely, using the $\mathfrak{gl}(V)$-equivariant splitting
$\g_1=S^2V^*\otimes V=V^*\oplus(S^2V^*\otimes V)_{\text{tf}}=\g_1'\oplus\g_1''$
 (trace and trace-free parts),
the principal bundle $F_1\to F_0=F r_M$ is
reduced to the (Abelian) structure group $\g_1'$.

After this the geometric structure is prolonged.
The obtained projective structure has symmetry the entire diffeomorphism
group in the case of (even) line.

 \begin{prop}
The projective structure is of finite type for $\dim V=(m|n)\neq(1|0)$.
 \end{prop}

 \begin{proof}
In dimension $(0|1)$ we have $\fg_1=0$, so it is clear. Otherwise we claim that $\g_i=(\g_1')^{(i-1)}=0$ for all $i>1$.
Indeed 
the Spencer complex in $\mathbb Z$-degree $2$ is given by
$$
0\to V^*\otimes V^*\to \Lambda^2V^*\otimes\mathfrak{gl}(V)\to \Lambda^3V^*\otimes V\to0
$$
and its first cohomology group vanishes, i.e., $H^{2,1}(V,V\oplus\mathfrak{gl}(V)\oplus\fg'_1)=0$.
Indeed, we have
 $$
(\delta A)(u,v,u)=
A(u,v)u+(-1)^{|u||v|}A(u,u)v-(-1)^{|u||v|}A(v,u)u-(-1)^{|u||v|+|u|}A(v,u)u,
 $$
for $u,v\in V$, $A\in V^*\otimes V^*$. Taking $u,v$ independent
tells us that $A(u,u)=0$ for all $u\in V$. Considering the vanishing
of the remaining three terms gives $A=0$.
The  first cohomology groups then must vanish in higher $\mathbb Z$-gradings
too 
and $H^{\geq 2,1}(V,V\oplus\mathfrak{gl}(V)\oplus\fg'_1)=0$ is equivalent to the prolongation claim,
cf.\ \cite{T,KST}.
 \end{proof}

Consequently, the symmetry dimension of a projective structure on a
supermanifold $M$ with $\dim M\neq (1|0)$ is bounded by
 $$
\dim\mathfrak{s}\leq\dim V+\dim\mathfrak{gl}(V)+\dim\g'_1=\bigl(2m+n^2+m^2\,|\,2n+2mn\bigr).
 $$
Now we will consider examples of filtered structures in the nonholonomic case $\cD\varsubsetneq \cT M$.

\subsection{$G(3)$-supergeometries}\label{S52}

In \cite{KST}, we studied two types of $G(3)$-supergeometries,
where $G(3)$ is the exceptional simple Lie supergroup of dimension $(17|14)$.
\subsubsection{$G(3)$-contact supergeometry}

Consider a contact distribution $\cC$ of rank $(4|4)$ on a supermanifold $M$ of dimension $(5|4)$.
The induced conformally super-symplectic
structure on $\cC$ reduces the structure group to $CSpO(4|4)$ and it is still of infinite
type. A cone structure on $\cC$ is given by a field of supervarieties
in projectivized contact spaces. Namely for $x\in M_o$ the
projective superspace $\mathbb{P}\cC|_x$ contains a distinguished
subvariety $\cV|_x$ of dimension $(1|2)$ that is isomorphic to the
unique irreducible flag manifold of the simple Lie supergroup $OSp(3|2)$, namely $\cV|_x\cong OSp(3|2)/P_1^{\rm II}$,
where $\mathfrak{p}_1^{\rm II}$ is the parabolic subalgebra
 \raisebox{-0.09in}{
 $\begin{tikzpicture}
 \draw (1,0) -- (2,0);
 \node[draw,circle,inner sep=2pt,fill=mygray] at (1,0) {};
 \node[draw,circle,inner sep=2pt,fill=black] at (2,0) {};
 \node[below] at (1,-0.1) {$\times$};
 \end{tikzpicture}$
 }.
We call this subvariety the $(1|2)$-twisted cubic,
because its underlying classical manifold is a rational normal curve
of degree 3, which is ``deformed'' in 2 odd dimensions.

This cone field reduces the structure group to $COSp(3|2)\subset CSpO(4|4)$,
and now this is of finite type: if $\g_0=\mathfrak{cosp}(3|2)$
and $\g=\g(3)$ is the Lie algebra of $G(3)$ then
$H^{d,1}(\fm,\g)=0$ if $d>0$ by \cite[Theorem 3.9]{KST}, so
the maximal prolongation is $\g=\opp{pr}(\fm,\g_0)$ (Corollary 3.10 loc.cit.).
Such a geometric structure arises on the generalized flag-supervariety
$G(3)/P_1^{\rm IV}$ with marked Dynkin diagram
 \raisebox{-0.09in}{
 $\begin{tikzpicture}
 \draw (0,0) -- (1,0);
 \draw (0,0.07) -- (1,0.07);
 \draw (0,-0.07) -- (1,-0.07);
 \draw (0.4,0.15) -- (0.6,0) -- (0.4,-0.15);
 \draw (1,0.05) -- (2,0.05);
 \draw (1,-0.05) -- (2,-0.05);
 \draw (1.6,0.15) -- (1.4,0) -- (1.6,-0.15);
 \node[draw,circle,inner sep=2pt,fill=white] at (0,0) {};
 \node[draw,circle,inner sep=2pt,fill=mygray] at (1,0) {};
 \node[draw,circle,inner sep=2pt,fill=black] at (2,0) {};
 \node[below] at (0,-0.1) {$\times$};
 \end{tikzpicture}
 $}
and in \cite[Theorem 4.9]{KST} we established
that the maximal symmetry dimension (in the strong sense) of
supergeometries $(M,\cC,\cV)$ as above is $(17|14)$,
{\it under the assumption that the geometry is locally homogeneous.}
Now as a direct corollary of Theorem \ref{T1} we derive that the assumption
of local homogeneity can be removed (we fulfill thus what is written in
footnote 5 at page 54 of loc.cit.):
 \begin{theorem}
The maximal symmetry dimension of a $G(3)$-contact supergeometry $(M,\cC,\cV)$ is equal to $(17|14)$.
 \end{theorem}

\subsubsection{Super Hilbert--Cartan geometries}
Another $G(3)$ supergeometry lives on  supermanifolds of dimensions $(5|6)$ and it
is given by a superdistribution with growth vector $(2|4,1|2,2|0)$.
The symbol of a (fundamental, nondegenerate) superdistribution with such
growth vector can be one of four types \cite[Theorem 5.1]{KST},
and just of two types if its even part is the standard symbol
as for the Hilbert--Cartan equation.
Moreover one of them is generic, hence rigid, and it is called SHC type symbol.
More explicitly, for a basis of $\fm$
(listed in the format $(\text{even}|\text{odd})$)
 $
\g_{-1}=\langle e_1,e_2\,|\,\theta_1',\theta_1'',\theta_2',\theta_2''\rangle$,
$\g_{-2}=\langle h\,|\,\varrho_1,\varrho_2\rangle$,
$\g_{-3}=\langle f_1,f_2\,|\,\cdot \rangle$,
the nontrivial commutator relations of the SHC type symbol are the following:
 \begin{multline*}
[e_1,e_2]=h,\ [e_1,h]=f_1,\ [e_2,h]=f_2,\
[\theta_1',\theta_2']=[\theta_1'',\theta_2'']=h,\\
[e_1,\theta_2']=[e_2,\theta_1'']=\varrho_1,\
[e_1,\theta_2'']=-[e_2,\theta_1']=\varrho_2,\\
[\theta_1',\varrho_1]=[\theta_1'',\varrho_2]=f_1,\
[\theta_2'',\varrho_1]=-[\theta_2',\varrho_2]=f_2.
 \end{multline*}

Such a superdistribution arises on the generalized flag-supermanifold
$G(3)/P_2^{\rm IV}$ with the marked Dynkin diagram
\raisebox{-0.09in}{$\begin{tikzpicture}
 \draw (0,0) -- (1,0);
 \draw (0,0.07) -- (1,0.07);
 \draw (0,-0.07) -- (1,-0.07);
 \draw (0.4,0.15) -- (0.6,0) -- (0.4,-0.15);
 \draw (1,0.05) -- (2,0.05);
 \draw (1,-0.05) -- (2,-0.05);
 \draw (1.6,0.15) -- (1.4,0) -- (1.6,-0.15);
 \node[draw,circle,inner sep=2pt,fill=white] at (0,0) {};
 \node[draw,circle,inner sep=2pt,fill=mygray] at (1,0) {};
 \node[draw,circle,inner sep=2pt,fill=black] at (2,0) {};
 \node[below] at (1,-0.1) {$\times$};
 \end{tikzpicture}$}\;.
For the grading corresponding to the parabolic
$P_2^{\rm IV}$ the Lie superalgebra $\g=\opp{Lie}(G(3))$ contains $\fm$
as the negative part. In \cite[Theorem 3.16]{KST} we established
$H^{d,1}(\fm,\g)=0$ for all $d\ge0$. Hence (Corollary 3.17 loc.cit.)
$\g$ is the Tanaka-Weisfeiler prolongation of $\fm$, i.e.,
$\g=\opp{pr}(\fm)$.

The methods of \cite{KST} allow to conclude that $(17|14)$ is the
maximal symmetry dimension for locally homogeneous distributions with
the SHC symbol. Using Theorem \ref{T1} of the present paper we
derive the result in full generality without the local homogeneity assumption.
 \begin{theorem}
The maximal symmetry dimension of a superdistribution with SHC symbol
is $(17|14)$.
 \end{theorem}


\subsection{Super-Poincar\'e structures}\label{S53}
Let $\mathbb V$ be a complex vector space with a non-degenerate symmetric bilinear form $(\cdot,\cdot)$ and $\mathbb S$ an
irreducible module over the associated Clifford algebra. A supertranslation algebra is a $\mathbb Z$-graded
Lie superalgebra $\fm=\fm_{-2}\oplus\fm_{-1}$, where $\fm_{-2}=\fm_{\bar 0}=\mathbb V$ and $\fm_{-1}=\fm_{\bar 1}=\mathbb S\oplus\cdots\oplus\mathbb{S}$ is the direct sum of an arbitrary number $N\geq 1$ of copies of $\mathbb S$, whose bracket $\Gamma:\fm_{-1}\otimes\fm_{-1}\rightarrow\fm_{-2}$ is of the form
$(\Gamma(s,t),v)=\mathfrak{B}(v\cdot s,t)$ for $v\in \mathbb V$, $s,t\in \fm_{-1}$. Here $\mathfrak{B}$ is a non-degenerate bilinear form on $\fm_{-1}$, which is admissible in the sense of  \cite{AC}. We note that $\Gamma$ is $\fso(\mathbb V)$-equivariant, so the semidirect sum $\fp=\fm\rtimes\mathfrak{so}(\mathbb V)$ is a Lie superalgebra, usually referred to as Poincar\'e superalgebra (complex, N-extended, in dimension $\dim\mathbb V$).

Real supermanifolds $M$ endowed with a strongly-regular odd distribution $\cD\subset\cT M$ whose complexified symbol is $\fm$ appear naturally in ``super-space'' approaches to supergravity  and rigid supersymmetric field theories (see \cite{SS1,SS2, FSCMP, FSATMP, dFS1, dFS2} and references therein).
The superdistribution $\cD$ has been called a super-Poincar\'e structure in \cite{AS} and the main result of that paper is the explicit description of the maximal transitive prolongation of $\fm$. Here we recall it for the reader's convenience:
\begin{theorem}
\label{thm:superPoincare}
{\normalsize If $\dim \mathbb V=1,2$, the prolongation $\fg$ of $\fm$ is infinite-dimensional. If $\dim \mathbb V\geq 3$, it is finite-dimensional and
$\fg_{\leq 0}=\fp\oplus\mathbb C Z\oplus\fh_0$ as the vector space direct sum of the Poincar\'e superalgebra, the grading element $Z$  and the algebra
$\fh_0=\{D\in \fg_0| [D,{\fm_{-2}}]=0\}$
 of the internal symmetries of $\fm_{-1}$.
If $\dim \mathbb V\geq 3$, then $\fg_{p}= 0$ for all $p\geq 1$ in all cases except those listed in Table  \ref{tabintroduction}.}

{{\small
\begin{table}[H]
 \begin{centering}
\begin{tabular}{|c|c|c|c|c|c|c|}
\hline
$\fg$ & $\dim\fg$ & Dynkin diagram & $\dim \mathbb V$ & $\dim \mathbb S$ & $N$ & $\fh_0$\\
\hline
\hline
&&&&&\\[-3mm]

$\mathfrak{osp}(1|4)$& $(10|4)$
&
\begin{tikzpicture}
\node[root]   (3) [label=]      {} ;
\node[broot]   (4) [right=of 3] {} [label=]     edge [-] (3);
\end{tikzpicture}&
$3$ & $2$ & $1$ & $0$
\\
\hline
&&&&&\\[-3mm]

$\begin{gathered}\mathfrak{osp}(2m+1|4)\\ m\geq1\end{gathered}$&
$(2m^2+m+10|8m+4)$
&
\begin{tikzpicture}
\node[root]   (3) [label=] {};
\node[root, fill=mygray]   (4) [right=of 3] {} [label=]  edge [-] (3);
\node[]   (5) [right=of 4] {$\;\cdots\,$} edge [-] (4);
\node[root]   (7) [right=of 5] {} [label=]  edge [rdoublearrow] (5);
\end{tikzpicture}&
$3$ & $2$ & $2m+1$ & $\mathfrak{so}(2m+1)$
\\
\hline
&&&&&\\[-3mm]

$\begin{gathered} \mathfrak{osp}(2|4) \\ \\ \end{gathered}$& $\begin{gathered} (11|8) \\ \\ \end{gathered}$
&
\begin{tikzpicture}
\node[root]   (3) [label=] {} ;
\node[root, fill=mygray]   (4) [above right=of 3] [label=] {} edge [-] (3);
\node[root, fill=mygray]   (5) [below right=of 3] [label=] {} edge [-] (3) edge [doublenoarrow] (4);
\end{tikzpicture}&
$\begin{gathered} 3\\ \\ \end{gathered}$ & $\begin{gathered} 2\\ \\\end{gathered}$ & $\begin{gathered} 2\\ \\\end{gathered}$ & $\begin{gathered} \mathfrak{so}(2)\\ \\ \end{gathered}$
\\
\hline
&&&&&\\[-3mm]

$\begin{gathered}\mathfrak{osp}(2m|4)\\ m\geq2 \\ \\ \end{gathered}$& $\begin{gathered} (2m^2-m+10|8m)\\ \\ \end{gathered}$
&
\begin{tikzpicture}
\node[root]   (3)[label=] {} ;
\node[root, fill=mygray]   (4) [right=of 3][label=] {} edge [-] (3);
\node[]   (6) [right=of 4] {$\;\cdots\quad$} edge [-] (4);
\node[root]   (7) [above right=of 6] [label=]{} edge [-] (6);
\node[root]   (8) [below right=of 6] [label=]{} edge [-] (6);
\end{tikzpicture}&
$\begin{gathered} 3\\ \\ \end{gathered}$ & $\begin{gathered} 2\\ \\\end{gathered}$ & $\begin{gathered} 2m\\ \\\end{gathered}$ & $\begin{gathered} \mathfrak{so}(2m)\\ \\ \end{gathered}$
\\
\hline
&&&&&\\[-3mm]

$\begin{gathered}\mathfrak{sl}(m+1|4)\\m\neq3\end{gathered}$& $(m^2+2m+16|8m+8)$
&
\begin{tikzpicture}
\node[root]   (1)       [label=]              {};
\node[root, fill=mygray] (2) [right=of 1] {}[label=] edge [-] (1);
\node[]   (3) [right=of 2] {$\;\cdots\,$} edge [-] (2);
\node[root, fill=mygray]   (4) [right=of 3] {}[label=] edge [-] (3);
\node[root]   (5) [right=of 4] {}[label=] edge [-] (4);
\end{tikzpicture}&
$4$ & $4$ & $m+1$ & $\mathfrak{gl}(m+1)$
\\
\hline
&&&&&\\[-3mm]

$\mathfrak{pgl}(4|4)$& $(30|32)$
&
\begin{tikzpicture}
\node[root]   (1)       [label=]              {};
\node[root, fill=mygray] (2) [right=of 1] {}[label=] edge [-] (1);
\node[]   (3) [right=of 2] {$\;\cdots\,$} edge [-] (2);
\node[root, fill=mygray]   (4) [right=of 3] {}[label=] edge [-] (3);
\node[root]   (5) [right=of 4] {}[label=] edge [-] (4);
\end{tikzpicture}&
$4$ & $4$ & $4$ & $\mathfrak{gl}(4)$
\\
\hline
&&&&&\\[-3mm]

$F(4)$& $(24|16)$
&
\begin{tikzpicture}
\node[root]   (1)        [label=]             {};
\node[root, fill=mygray] (2) [right=of 1] {} [label=]  edge [rtriplearrow] (1) edge [-] (1);
\node[root]   (3) [right=of 2] {} [label=] edge [-] (2);
\node[root]   (4) [right=of 3] {} [label=] edge [doublearrow] (3);
\end{tikzpicture}&
$5$ & $4$ & $2$ & $\mathfrak{sl}(2)$\\
\hline
\end{tabular}
 \end{centering}
\caption[]{Exceptional prolongations of super-Poincar\'e algebras.}
\label{tabintroduction}
\end{table}
\vskip-15pt
}}

\noindent
Here the simple roots of degree $1$ coincide with the odd simple roots, i.e., those associated to black and gray nodes on the Dynkin diagram.
\end{theorem}
Since the complexification of the prolongation of a real symbol is the prolongation of the complexified symbol, one may
combine Theorems \ref{T1} and \ref{thm:superPoincare} to get the bound on the dimension of the symmetry superalgebra of a super-Poincar\'e structure in dimension $\dim\mathbb V\geq 3$. For the exceptional cases with $\fg_1\neq 0$, it  is provided by the second column in  Table  \ref{tabintroduction}, in all other cases it is given by
 $$
 \dim\fs\leq
\Bigl(\tfrac{d(d+1)}{2}+1+\dim\fh_0\,|\,N 2^{[d/2]}\Bigr)\;,
 $$
where $d=\dim\mathbb V$ and square brackets refer to the integer part.
Furthermore, the subalgebra $\fh_0$ of the internal symmetries can be easily described on a case-by-case basis. It splits into the sum of its symmetric part $\fh_0^{s}$ and skew-symmetric part $\fh_0^{a}$ with respect to $\mathfrak B$
and the condition that elements of $\fh_0$ act as derivations of $\fm$ yields:
\begin{equation*}\label{eeeeaaaa}
\begin{aligned}
\fh_0^{a}&=\{D\in\mathfrak{gl}(\fm_{-1})\mid \mathfrak{B}(Ds,t)=-\mathfrak{B}(s,Dt),\     D(v\cdot s)=v\cdot Ds\ \forall v\in \mathbb V,\ s,t\in \fm_{-1}\}\,,\\
\fh_0^{s}&=\{D\in\mathfrak{gl}(\fm_{-1})\mid \mathfrak{B}(Ds,t)=\mathfrak{B}(s,Dt),\ D(v\cdot s)=-v\cdot Ds\ \forall v\in \mathbb V,\ s,t\in \fm_{-1}\}\,.
\end{aligned}
\end{equation*}
It is well-known that $\mathbb S$ is $\fso(\mathbb V)$-irreducible if $d$ is odd and the direct sum of two inequivalent $\fso(\mathbb V)$-irreducible submodules if $d$ is even. By $\fso(\mathbb V)$-equivariancy, a uniform (but not sharp) bound on $\dim\fh_0$ is thus given by $N^2$ if $d$ is odd and $2N^2$ if $d$ is even.

\smallskip

We conclude this subsection with the following direct consequence of \S\ref{subsec:canpar}.
Consider, for instance, the $4$- and $11$-dimensional vector spaces $V$ in Lorentzian signature. The real spinor module $S$ is an irreducible module for the Lorentz algebra $\fso(V)$ and it is of Clifford real type (i.e., $S\otimes\mathbb C=\mathbb S$). It follows from Theorem \ref{thm:superPoincare} that, if we reduce the structure algebra to $\fso(V)$, the prolongation of the real $N=1$ Poincar\'e superalgebra
$$\fp=\fp_{-2}+\fp_{-1}+\fp_0=V+S+\mathfrak{so}(V)$$ is just $\fp$. Theorem \ref{thm:absolute-parallelism} and Remark
\ref{rem:equivariancy} then imply that any super-Poincar\'e structure $\cD$ with reduced structure group $P_0=\Spin(V)$ has associated a Cartan superconnection on a $P_0$-principal bundle $\pi:P\to M$. This bridges from the ``super-space approach'' to the so-called ``rheonomic approach'' of supergravity and supersymmetric field theories (see, e.g., the nice reviews \cite{DAuria, Cas}).
We stress that in the rheonomic approach the axioms of a Cartan superconnection follows from a Lagrangian principle on the absolute parallelism $\Phi$, whereas our general construction affords the existence of the Cartan superconnection, from purely geometric arguments.

It would be interesting to study the normalization conditions on the Cartan superconnection in the cohomological spirit of \S\ref{subsec:normcond} and compare them with those traditionally obtained in the rheonomic approach via Lagrangian principles.

 \subsection{Odd Ordinary Differential Equations}

 \subsubsection{Review of some classical ODE}

 Classically, ODE are geometrically viewed as submanifolds of a jet space with the inherited structure (via pullback along the inclusion map).  This leads to formulating these as manifolds $M$ with a rank 2 distribution $C \subset TM$ (having specific symbol $\fm$) and a splitting into line fields $C = E \op V$:
 \begin{itemize}
 \item {\em 2nd order ODE $y'' = f(x,y,y')$ (up to point transformations)}: Introduce local coordinates $(x,y,p)$ on $M$ with $C = E \op  V = \langle \partial_x + p\partial_y + f\partial_p \rangle \op \langle \partial_p \rangle$.  Then $C$ has symbol $\fm = \fg_{-1} \op \fg_{-2} = \langle X,e_1 \rangle \op \langle e_2 \rangle$ with non-trivial bracket $[X,e_1] = e_2$.  ($C$ is a contact distribution.)
 \item {\em 3rd order ODE $y''' = g(x,y,y',y'')$ (up to contact transformations)}: Introduce local coordinates $(x,y,p,q)$ on $M^4$ with $C = E \op V = \langle \partial_x + p\partial_y + q\partial_p + g\partial_p \rangle \op \langle \partial_q \rangle$.  Then $C$ has symbol $\fm = \fg_{-1} \op \fg_{-2} \op \fg_{-3} = \langle X, e_1 \rangle \op \langle e_2 \rangle \op \langle e_3 \rangle$ with non-trivial brackets $[X,e_1] = e_2$, $[X,e_2] = e_3$.  ($C$ is an Engel distribution.)
 \end{itemize}
 The splitting indicates a reduction $\fg_0 \hookrightarrow \mathfrak{der}_{{\rm gr}}(\fm)$.  In both cases, $\dim(\fg_0) = 2$ with $\fg_0 \hookrightarrow \fgl(\fg_{-1})$ corresponding to scalings along the two distinguished directions in $\fg_{-1}$.

There are well-known $\mathbb{Z}$-gradings of $A_2 \cong \fsl_3$ and $B_2 \cong \fso_{2,3}$ for which the negative parts $\fg_-$ are the indicated symbol algebras above and the non-negative parts $\fp = \fg_{\geq 0}$ are the respective Borel subalgebras.  This implies inclusions of $\fsl_3$ and $\fso_{2,3}$ into the respective Tanaka prolongations $\operatorname{pr}(\fg_-,\fg_0)$.
One can show that these are, in fact, equalities by verifying that
$H^{+,1}(\fg_-,\fg) = 0$, as was done by Yamaguchi \cite{Y} using
Kostant's theorem. (Previously this was done by
Tresse, Cartan and Chern using geometric methods.)

 \subsubsection{2nd order odd ODE}

Consider a 2nd order {\sl odd} ODE $\xi'' = \fF(x,\xi,\xi')$, where $\xi$ is an odd function of the even variable $x$, and $\fF$ is an odd function.  As in the classical case, the space $M=\RR^{1|2}(x,\xi,\xi')$
is equipped with a distribution $C = E \op V \subset TM$, where $E = \langle \partial_x + \xi' \partial_{\xi} + \fF \partial_{\xi'} \rangle$ is {\sl even} and $V = \langle \partial_{\xi'} \rangle$ is {\sl odd}.  The symbol is $\fm = \fg_{-1} \op \fg_{-2} = (\langle X \rangle \op \langle \theta_1 \rangle) \op \langle \theta_2 \rangle$, for even $X$ and odd $\theta_1,\theta_2$ satisfying $[X,\theta_1] = \theta_2$.

Consider $\fg = \fsl(2|1)$, i.e.\ supertrace-free $3\times 3$ matrices, with $\mathbb{Z}_2$-grading induced from $\R^{2|1}$.  Write $\fh = \langle \sfZ_1, \sfZ_2 \rangle$, where $\sfZ_1 = \operatorname{diag}(0,-1,-1)$ and $\sfZ_2 = \operatorname{diag}(-1,-1,-2)$.  Defining $\epsilon_i \in \fh^*$ by $\epsilon_i( \operatorname{diag}(h_1,h_2,h_3)) = h_i$, we have $\{ \sfZ_1, \sfZ_2 \}$ being the dual basis to $\{ \alpha_1 := \epsilon_1 - \epsilon_2, \alpha_2 := \epsilon_2 - \epsilon_3 \}$.  Letting $E_{ij}$ be the matrix with a 1 in the $(i,j)$-position and $0$ elsewhere, we use $(\sfZ_1,\sfZ_2)$ to induce a bigrading on $\fg$, and let $\sfZ = \sfZ_1 + \sfZ_2$ be the induced grading on $\fg$.  In particular:
 \begin{align}
 \fg_- = \fg_{-1} \op \fg_{-2} = (\fg_{-1,0} \op \fg_{0,-1}) \op \fg_{-1,-1} = (\langle E_{21} \rangle \op \langle E_{32}) \rangle \op \langle E_{31} \rangle.
 \end{align}
Here $E_{21}$ is even, while $E_{31}, E_{32}$ are odd.  The only non-trivial bracket on $\fg_-$ is $[E_{32}, E_{21}] = E_{31}$.  We conclude that $\fsl(2|1)$ includes into $\operatorname{pr}(\fg_-,\fg_0)$.  From Table \ref{F:odd-2ODE}, we use the differentials $\delta_k : C^k(\fg_-,\fg) \to C^{k+1}(\fg_-,\fg)$ to conclude that $H^{+,1}(\fg_-,\fg) = 0$, whence that $\operatorname{pr}(\fg_-,\fg_0) \cong \fsl(2|1)$.

  \begin{table}[h]
\[
 \begin{array}{cc}
 \begin{array}{|c|l|c|c|c} \hline
 \mbox{Bi-grading} & \mbox{Basis} & \ker(\delta_1) \\ \hline\hline
 (2,2) & E_{31}^* \otimes E_{13} & \cdot \\ \hline

 (2,1) & E_{31}^* \otimes E_{12}, \quad E_{21}^* \otimes E_{13} & \cdot \\ \hline

 (2,0) & E_{21}^* \otimes E_{12} & \cdot \\ \hline
 (1,2) & E_{32}^* \otimes E_{13}, \quad E_{31}^* \otimes E_{23} & \cdot \\ \hline

 (1,1) & E_{21}^* \otimes E_{23}, \quad E_{32}^* \otimes E_{12}, \quad E_{31}^* \otimes Z_1, \quad E_{31}^* \otimes Z_2 & \delta_0(E_{13})\\ \hline

 (1,0) & E_{21}^* \otimes Z_1, \quad E_{21}^* \otimes Z_2, \quad E_{31}^* \otimes E_{32} & \delta_0(E_{12})\\ \hline
 (0,2) & E_{32}^* \otimes E_{23} & \cdot \\ \hline

 (0,1) & E_{32}^* \otimes Z_1, \quad E_{32}^* \otimes Z_2, \quad E_{31}^* \otimes E_{21} & \delta_0(E_{23})\\ \hline
 \end{array}
 \end{array}
 \]
 \caption{Confirmation of $H^{+,1}(\fg_-,\fg) = 0$ for 2nd order odd ODE}
 \label{F:odd-2ODE}
 \end{table}

When $\fF = 0$, i.e.\ $\xi'' = 0$, we have the prolongation $X$ of $\bS_f$ (satisfying $\cL_X E \subset E$ and $\cL_X V \subset V$), expressed in terms of a generating superfunction $f$ (see Appendix \ref{S:odd-prolong}).
 \begin{align}
 \begin{footnotesize}
 \begin{array}{|c|c|c|c|c|} \hline
 & \multicolumn{2}{c|}{\mbox{Even part}} & \multicolumn{2}{c|}{\mbox{Odd part}}\\ \hline
 \mbox{Grading} & f & \mbox{Prolongation of } \bS_f & f & \mbox{Prolongation of } \bS_f \\ \hline\hline
 +2 & \cdot & \cdot & x\xi \xi' &
 -x\xi\partial_x +\xi\xi'\partial_{\xi'} +(2\xi+x\xi')\xi''\partial_{\xi''}
 \\
 +1 & x\xi -x^2 \xi' & x^2 \partial_x + x\xi\partial_\xi + (\xi - x\xi') \partial_{\xi'} - 3x\xi'' \partial_{\xi''} & \xi' \xi& \xi\partial_x - \xi' \xi'' \partial_{\xi''} \\
 0 & \begin{array}{c} x\xi' \\ \xi \end{array} & \begin{array}{c}
  -x\partial_x + \xi' \partial_{\xi'} + 2\xi'' \partial_{\xi''} \\
  \xi\partial_\xi + \xi' \partial_{\xi'} + \xi'' \partial_{\xi''}
  \end{array} & \cdot & \cdot\\
 -1 & \xi' & -\partial_x & x & x\partial_\xi + \partial_{\xi'}\\
 -2 & \cdot & \cdot & 1 & \partial_\xi\\ \hline
 \end{array}
 \end{footnotesize}
 \end{align}
These symmetries are all projectable over $(x,\xi)$-space, i.e.\ they are (prolonged) {\sl point} symmetries.  (Equivalently, their generating functions are linear in $\xi'$.) This symmetry superalgebra is indeed $\fsl(2|1)$.  In stark contrast to the classical case, 2nd order odd ODE do not admit non-trivial deformations:
 \begin{prop}
Any 2nd order odd ODE $\xi'' = \fF(x,\xi,\xi')$ is locally equivalent to the trivial equation $\xi'' = 0$ via a point transformation,
and thus
has symmetry dimension $(4|4)$.
 \end{prop}

 \begin{proof}
Since $\fF$ and $\xi,\xi'$ are odd, then any 2nd order odd ODE must be of the form:
 \begin{align}\label{xiFF}
\xi'' = \fF_0(x) \xi + \fF_1(x) \xi'.
 \end{align}
Let $(\widetilde{x},\widetilde\xi) := (a(x),b(x)\xi)$,
$a'(x)\neq0\neq b(x)$, which induces
$\frac{d\widetilde\xi}{dx} = \frac{d\widetilde\xi}{d\widetilde{x}} a'
= b' \xi + b \xi'$ and
$\frac{d^2\widetilde\xi}{dx^2} = \frac{d^2\widetilde\xi}{d\widetilde{x}^2} (a')^2 + \frac{d\widetilde\xi}{d\widetilde{x}} a''
= b'' \xi + 2 b' \xi' + b \xi''$.  We find that $\frac{d^2\widetilde\xi}{d\widetilde{x}^2} = \widetilde\fF_0 \widetilde\xi + \widetilde\fF_1 \frac{d\widetilde\xi}{d\widetilde{x}}$, where
 \begin{align*}
\widetilde\fF_0 = \tfrac{(b''+b\fF_0)b-(2b'+b\fF_1)b'}{(a')^2b^2}, \quad
\widetilde\fF_1 = \tfrac{(2 b' + b \fF_1) a'  - ba''}{(a')^2b}.
 \end{align*}
This vanishes for solutions of the even 2nd order ODE system
 $$
a''=\Bigl(2\tfrac{b'}{b}+\fF_1\Bigr)a',\quad
b''=\tfrac{2}{b}b'^2+\fF_1b'-\fF_0b,
 $$
and this trivializes the odd ODE \eqref{xiFF}.
 \end{proof}


 \subsubsection{3rd order odd ODE}

Consider a 3rd order {\sl odd} ODE $\xi''' = \mathfrak{G}(x,\xi,\xi',\xi'')$, where $\xi$ is an odd function of the even variable $x$, and $\mathfrak{G}$ is an odd function.  As in the classical case, the space $M=\RR^{1|3}(x,\xi,\xi',\xi'')$
is equipped with a distribution $C = E \op V \subset TM$, where $E = \langle \partial_x + \xi' \partial_{\xi} + \xi'' \partial_{\xi'} + \mathfrak{G} \partial_{\xi''} \rangle$ is {\sl even} and $V = \langle \partial_{\xi''} \rangle$ is {\sl odd}.  The symbol is $\fm = \fg_{-1} \op \fg_{-2} \op \fg_{-3} = (\langle X \rangle \op \langle \theta_1 \rangle) \op \langle \theta_2 \rangle \op \langle \theta_3 \rangle$, for even $X$ and odd $\theta_1,\theta_2, \theta_3$ satisfying $[X,\theta_1] = \theta_2$, $[X,\theta_2] = \theta_3$.  Let us now compute the prolongation directly.

Since $\fg_{-1}$ has a splitting into distinguished lines, then $\fg_0 = \langle T_1, T_2 \rangle \hookrightarrow \mathfrak{der}_{gr}(\fm)$ is even with $T_1 = \operatorname{diag}(1,0,1,2)$ and $T_2 = \operatorname{diag}(0,1,1,1)$, expressed in the $\{ X, \theta_1, \theta_2, \theta_3 \}$ basis.  Interestingly, the height of the prolongation differs from the depth of $\fm$.

 \begin{prop}
 $\dim(\fg_1) = (1|0)$, $\dim(\fg_2) = (0|1)$, while $\dim(\fg_k) = (0|0)$ for all $k \geq 3$.
 \end{prop}

 \begin{proof}
  Let $A \in \fg_1$ be odd, so $AX = 0$ and $A\theta_1 = a_1T_1 + a_2T_2$.  We find $A=0$ from:
 \begin{align*}
 A \theta_2 = A[X,\theta_1] = -a_1X, \quad
 A \theta_3 = A[X,\theta_2] = 0, \quad
 0 = A[\theta_1,\theta_1] = 2a_2 \theta_1, \quad
 0 = A[\theta_2,\theta_2] = -2a_1 \theta_3.
 \end{align*}
 Now let $A \in \fg_1$ be even.  As a map $\fg_{-1} \to \fg_0$, we have $AX = a_1T_1 + a_2T_2$ and $A\theta_1 = 0$.  Then
 \begin{align*}
 A\theta_2 = A[X,\theta_1] = a_2\theta_1, \quad
 A\theta_3 = A[X,\theta_2] = (a_1+2a_2)\theta_2, \quad
 0 = A[X,\theta_3] = 3(a_1+a_2)\theta_3,
 \end{align*}
 so $a_2=-a_1$.  Taking $-a_2=a_1=1$ yields a specific (even) generator for $\fg_1$, which we henceforth label as $A$.  (Since all odd-odd brackets vanish, then $A$ is indeed a superderivation.)

Let $B \in \fg_2$ be even.  Write $BX = bA$, $B\theta_1 = 0$.  We find $B=0$ from:
 \begin{align*}
 B\theta_2 = B[X,\theta_1] = 0, \quad
 B\theta_3 = B[X,\theta_2] = -b\theta_1, \quad
 0 = B[X,\theta_3] = -2b\theta_2.
 \end{align*}
 Now let $B \in \fg_2$ be odd.  Write $BX=0$, $B\theta_1 = bA$.  We obtain:
 \begin{align*}
 B\theta_2 = B[X,\theta_1] = -b(T_1 - T_2), \quad
 B\theta_3 = B[X,\theta_2] = bX.
 \end{align*}
A direct check shows that all conditions resulting from
$B[X,\theta_3]=0=B[\theta_i,\theta_j]$
are satisfied, so $B$ is indeed a superderivation.  Take $b=1$ above yields a specific (odd) generator for $\fg_2$ that we henceforth label as $B$.

Let $C \in \fg_3$ be even.  Write $C X = 0$, $C \theta_1 = cB$.  We find $C=0$ from:
 \begin{align*}
 C\theta_2 = C[X,\theta_1] = 0, \quad 0 = C[\theta_1,\theta_2] = c(T_2 - T_1).
 \end{align*}
Now let $C \in \fg_3$ be odd.  Write $C X = cB$, $C \theta_1 = 0$.  We find $C=0$ from:
 \begin{align*}
 C\theta_2 = C[X,\theta_1] = [cB,\theta_1] = cA, \quad
 0 = C[\theta_2,\theta_2] = 2[C\theta_2,\theta_2] = 2c[A,\theta_2] = -2c\theta_1.
 \end{align*}
 \end{proof}

When $\mathfrak{G} = 0$, i.e.\ $\xi''' = 0$, we have the symmetries $X = \bS_f$ (satisfying $\cL_X E \subset E$ and $\cL_X V \subset V$), expressed in terms of a generating superfunction $f$ (see Appendix \ref{S:odd-prolong}).

 \begin{align}
 \begin{footnotesize}
 \begin{array}{|c|c|c|c|c|} \hline
 &  \multicolumn{2}{c|}{\mbox{Even part}} &  \multicolumn{2}{c|}{\mbox{Odd part}}\\ \hline
 \mbox{Grading} & f & \mbox{Prolongation of } \bS_f & f & \mbox{Prolongation of } \bS_f\\ \hline\hline
 +2 & \cdot &\cdot & \xi \xi' & -\xi\partial_x + \xi'\xi'' \partial_{\xi''} + 2 \xi' \xi''' \partial_{\xi'''} \\ \hline
 +1 & x\xi - \frac{x^2}{2} \xi' & \begin{array}{l} \tfrac{x^2}{2}\partial_x + x\xi\partial_\xi + \xi\partial_{\xi'} \\
 \quad + (\xi' - x\xi'') \partial_{\xi''} - 2x\xi''' \partial_{\xi'''} \end{array} & \cdot & \cdot\\ \hline
 0 & \begin{array}{c} \xi - x\xi' \\ \xi \end{array} &\begin{array}{l} x\partial_x + \xi \partial_\xi - \xi''\partial_{\xi''} - 2\xi''' \partial_{\xi'''} \\
  \xi\partial_\xi + \xi' \partial_{\xi'} + \xi'' \partial_{\xi''} + \xi''' \partial_{\xi'''}\end{array} & \cdot & \cdot\\ \hline
 -1 & \xi' & -\partial_x & \frac{x^2}{2} & \tfrac{x^2}{2}\partial_\xi + x\partial_{\xi'} + \partial_{\xi''}\\ \hline
 -2 & \cdot & \cdot & x & x\partial_\xi + \partial_{\xi'}\\ \hline
 -3 & \cdot & \cdot & 1 & \partial_\xi\\ \hline
 \end{array}
 \end{footnotesize}
 \end{align}
These symmetries are all projectable over $(x,\xi)$-space, i.e.\ they are (prolonged) {\sl point} symmetries.  Abstractly, the symmetry superalgebra $\fg$ has derived superalgebras
 \begin{align}
 \fg^{(1)} = \langle 0 \,\,\,|\,\,\, \xi\xi' \rangle \ltimes \fg^{(2)}, \quad
 \fg^{(2)} = \left\langle \xi', \xi - x\xi', x\xi - \tfrac{x^2}{2} \xi' \,\,\,|\,\,\, 1, x, \tfrac{x^2}{2} \right\rangle.
 \end{align}
We note that $\fg^{(2)} \cong \fsl(2,\R) \ltimes S^2 \R^2$, where the even part is $\fsl(2,\R)$ and odd part is $S^2 \R^2$ (abelian), with even-odd brackets given by the naturally induced representation.  Alternatively, $\fg^{(2)}$ is the Euclidean Lie algebra $\fe(3,\R) := \fso(3,\R) \ltimes \R^3$ regarded as a Lie superalgebra with even part $\fso(3,\R)$ and odd part $\R^3$ (abelian).

Consequently Theorem \ref{T2} implies:

 \begin{theorem}
Any 3rd order odd ODE $\xi''' = \fG(x,\xi,\xi',\xi'')$ has contact
symmetry superalgebra of dimension at most $(4|4)$ and this bound is sharp.
 \end{theorem}

In contrast to the 2nd order odd ODE case,
3rd order odd ODE are not in general contact-trivializable,
i.e.\ equivalent to $\xi''' = 0$.
In fact, 3rd order odd ODE have the form
 $$
\xi'''=a(x)\xi+b(x)\xi'+c(x)\xi''+d(x)\xi\xi'\xi''
 $$
and one can verify that the term $d(x)$ is a relative invariant.
(The even part of the contact supergroup is $(x,\xi) \mapsto (\alpha(x), \beta(x) \xi)$ and the verification is straightforward;
the odd part does not contribute.)  Consequently, general 3rd order odd ODE are not linearizable.

Below we exhibit two examples of
3rd order odd ODEs that are not contact-trivializable.  Both have solvable symmetry superalgebras.
Symmetries are given in terms of their generating superfunctions (see Appendix \ref{S:odd-prolong} for the Lagrange bracket).

 \begin{align}
 \begin{array}{|c|c|c|cccc} \hline
 \mbox{Odd ODE} & \xi''' = \xi'' & \xi''' = \xi\xi'\xi''\\ \hline
 \mbox{Sym dim} & (2|3) & (2|2)\\ \hline
 \mbox{Symmetries} & \begin{array}{ll}
 \mbox{even part}: & \xi',\, \xi\\
 \mbox{odd part}: & 1,\, x,\, e^x
 \end{array} &
 \begin{array}{ll}
 \mbox{even part}: & \xi',\, x\xi'\\
 \mbox{odd part}: & \xi\xi',\, h:= 3 + x\xi\xi'
 \end{array}\\ \hline
 \begin{tabular}{c}
 Lagrange\\
 brackets
 \end{tabular} & \begin{array}{c|ccccc}
 & \xi' & \xi & 1 & x & e^x\\ \hline
 \xi' & \cdot & \cdot & \cdot & -1 & -e^x\\
 \xi & \cdot & \cdot & -1 & -x & -e^x\\
 1 & \cdot & 1 & \cdot & \cdot & \cdot\\
 x & 1 & x & \cdot & \cdot & \cdot\\
 e^x & e^x & e^x & \cdot & \cdot & \cdot
 \end{array} &  \begin{array}{c|cccc}
 & \xi' & x\xi' & \xi\xi' & h\\ \hline
 \xi' & \cdot & -\xi' & \cdot & -\xi\xi'\\
 x\xi' & \xi' & \cdot & \xi\xi' & \cdot\\
 \xi\xi' & \cdot & -\xi\xi' & \cdot & 3\xi'\\
 h & \xi\xi' & \cdot & 3\xi' & 6x\xi'
 \end{array}\\ \hline
 \end{array}
 \end{align}
More explicitly, the symmetries as (prolonged) contact vector fields are:
 \begin{itemize}
 \item $\xi''' = \xi''$:
 \begin{align}
 \begin{split}
 & -\partial_x, \quad \xi\partial_\xi + \xi'\partial_{\xi'} + \xi''\partial_{\xi''} + \xi'''\partial_{\xi'''},\\
 & \partial_\xi, \quad x\partial_\xi + \partial_{\xi'}, \quad e^x(\partial_\xi + \partial_{\xi'} + \partial_{\xi''} + \partial_{\xi'''}).
 \end{split}
 \end{align}
 \item $\xi''' = \xi\xi'\xi''$:
 \begin{align}
 \begin{split}
 &-\partial_x, \quad -x\partial_x + \xi'\partial_{\xi'} + 2\xi'' \partial_{\xi''} + 3\xi''' \partial_{\xi'''}, \quad
 -\xi\partial_x + \xi' \xi'' \partial_{\xi''} + 2\xi' \xi''' \partial_{\xi'''},\\
 & -x\xi\partial_x + 3\partial_\xi + \xi\xi' \partial_{\xi'} + (2\xi + x\xi') \xi'' \partial_{\xi''} + (3\xi'\xi'' + 3\xi\xi''' + 2x\xi'\xi''') \partial_{\xi'''}.
 \end{split}
 \end{align}
 \end{itemize}

\section*{Acknowledgments}

The research leading to these results has received funding from the Norwegian Financial Mechanism 2014-2021 (project registration number 2019/34/H/ST1/00636)
and the Troms\o{} Research Foundation (project ``Pure Mathematics in Norway'').

\section*{Conflict of Interest}
The authors declare that they have no conflict of interest.

\appendix
 \section{Generating superfunctions and an odd prolongation formula}
 \label{S:odd-prolong}

In this section, we briefly discuss the jet spaces $J^r := J^r(\R^{p|0},\R^{0|1})$ associated with $p$ even variables $(x^i)$ and one {\em odd} (dependent) variable $\xi$.  Our interest is only local, so we introduce these spaces (with {\sl Cartan distribution} $\cC_r \subset TJ^r$) from a local perspective only:
 \begin{itemize}
 \item $r=0$: Coordinates $(x^i,\xi)$.  No local structure.
 \item $r=1$: Coordinates $(x^i,\xi,\xi_i)$.  Local structure: $\cC_1 = \langle \partial_{x^i} + \xi_i \partial_\xi,\, \partial_{\xi_i} \rangle$.
 \item $r=2$: Coordinates $(x^i,\xi,\xi_i,\xi_{ij} = \xi_{ji})$.  Local structure: $\cC_2 = \langle \partial_{x^i} + \xi_i \partial_\xi + \xi_{ij} \partial_{\xi_j},\, \partial_{\xi_{ij}} \rangle$.
 \end{itemize}
(Here, $\xi,\xi_i,\xi_{ij}$, etc. are odd.)  For $r \geq 3$, we proceed similarly.  For notational convenience, we take $J^\infty$ as the inverse limit and on it we introduce
 \begin{align}
D_{x^i} = \partial_{x^i} + \xi_i \partial_\xi + \xi_{ij} \partial_{\xi_j} + \xi_{ijk} \partial_{\xi_{jk}} + ...
 \end{align}
with corresponding truncated vector field $D^{(r)}_{x^i}$ on $J^r$.
For example, $D^{(1)}_{x^i} = \partial_{x^i} + \xi_i \partial_\xi$, $D^{(2)}_{x^i} = \partial_{x^i} + \xi_i \partial_\xi + \xi_{ij} \partial_{\xi_j}$, etc.

A vector field $\bS$ on $J^r(\R^{p|0},\R^{0|1})$ is {\sl contact} if $\cL_\bS \cC_r \subset \cC_r$.  By the Lie--B\"acklund theorem,
any such vector field is projectable over $J^1$, and $\bS$ is canonically determined from this projection via {\sl prolongation}.  On $J^1$, fixing $\sigma = d\xi - (dx^i)\xi_i$ generating $\cC_1$, any contact vector field $\bS$ is uniquely determined by its {\sl generating superfunction} $f = \iota_\bS \sigma$ (which has opposite parity to $\bS$ since $\sigma$ is odd), and conversely any local superfunction $f = f(x^i,\xi,\xi_i)$ determines a contact vector field:

  \begin{prop} Given a superfunction $f = f(x^i,\xi,\xi_i)$ with pure parity $|f|$, its associated contact vector field has parity $|f|+1$ and is given by the formula
 \begin{align} \label{contact-vf}
\bS_f
&= (-1)^{|f|} (\partial_{\xi_i} f ) D^{(1)}_{x^i} + f \partial_\xi + (D^{(1)}_{x^i} f) \partial_{\xi_i}.
 \end{align}
We have $[\bS_f, \bS_g] = \bS_{[f,g]}$ with
 \begin{align} \label{lagrange-bracket}
[f,g] = f(\partial_{\xi}g)+(-1)^{|f|}(\partial_{\xi}f)g +
(D^{(1)}_{x^j}f)(\partial_{\xi_j}g) +(-1)^{|f|}(\partial_{\xi_j}f)(D^{(1)}_{x^j}g).
 \end{align}
The prolonged vector field $\bS_f^{(\infty)}$ on $J^\infty$ is obtained via the prolongation formula
 \begin{align}
 \bS_f^{(\infty)} = \bS_f + h_{jk} \partial_{\xi_{jk}} + h_{jk\ell} \partial_{\xi_{jk\ell}} + ...,
 \end{align}
 with coefficients $h_{jk} = D_{x^j} D_{x^k} f$, \,\,$h_{jk\ell} = D_{x^j} D_{x^k} D_{x^\ell} f$, etc. (with obvious truncations on each $J^r$).
 \end{prop}

 This result is analogous to \cite[Prop.4.3]{KST}.  Details are left as an exercise for the reader since it is proved similarly.


\end{document}